\documentclass[leqno]{amsart}
\usepackage{amsmath}
\usepackage{amsfonts}
\usepackage{amssymb}
\usepackage{pdfsync}
\usepackage{color,graphicx,tikz-cd}
\usepackage{a4wide}
\usepackage{amscd,amsthm,latexsym}
\usepackage{hyperref}
\usepackage{kotex}
\usepackage{relsize}
\usepackage{tikz}
\usepackage[misc]{ifsym}
\usetikzlibrary{snakes}
\usepackage{cite}

\allowdisplaybreaks %

\title{Generalized Schur function determinants using the Bazin identity}

\author{Jang Soo Kim}
\address{
Department of Mathematics, Sungkyunkwan University, Suwon 16420,
South Korea}
\email{jangsookim@skku.edu}

\author{Meesue Yoo}
\address{
Department of Mathematics, Chungbuk National University, Cheongju 28644,
South Korea}
\email{meesueyoo@chungbuk.ac.kr (\Letter)}

\date{\today}
\thanks{The first author was supported by NRF grants \#2019R1F1A1059081 and \#2016R1A5A1008055.
The second author was supported by NRF grant \#2020R1F1A1A01064138.}

\keywords{Jacobi--Trudi formula, Schur function, Bazin identity, Giambelli formula}

\subjclass[2010]{05E05}

\date{\today}

\newtheorem{thm}{Theorem}[section]
\newtheorem{lem}[thm]{Lemma}

\newtheorem{cor}[thm]{Corollary}
\theoremstyle{definition}
\newtheorem{exam}[thm]{Example}
\newtheorem{defn}[thm]{Definition}

\newtheorem{remark}[thm]{Remark}

\numberwithin{equation}{section}
\newcommand\Cont{\operatorname{Cont}}

\newcommand\ZZ{\mathbb{Z}}

\newcommand\ts{\mathbf{S}}

\newcommand\lm{{\lambda/\mu}}
\newcommand\nl{{\nu/\lambda}}

\newcommand\inv{\operatorname{inv}}

\newcommand\Par{\mathrm{Par}}

\renewcommand\vec[1]{\mathbf{#1}}

\definecolor{oxfordblue}{rgb}{0.0, 0.13, 0.28}
\definecolor{chartreuse(web)}{rgb}{0.5, 1.0, 0.0}
\definecolor{etonblue}{rgb}{0.59, 0.78, 0.64}
\definecolor{frenchblue}{rgb}{0.0, 0.45, 0.73}
\definecolor{cerulean}{rgb}{0.0, 0.48, 0.65}
\definecolor{asparagus}{rgb}{0.53, 0.66, 0.42}
\definecolor{dredcolor}{rgb}{0.9,0.3,0.4}
\definecolor{dbluecolor}{rgb}{0.01,0.02,0.7}

\begin{document}

\begin{abstract}
  In the literature there are several determinant formulas for Schur functions:
  the Jacobi--Trudi formula, the dual Jacobi--Trudi formula, the Giambelli
  formula, the Lascoux--Pragacz formula, and the Hamel--Goulden formula, where
  the Hamel--Goulden formula implies the others. In this paper we use an
  identity proved by Bazin in 1851 to derive determinant identities involving
  Macdonald's 9th variation of Schur functions. As an application we prove a
  determinant identity for factorial Schur functions conjectured by Morales,
  Pak, and Panova. We also obtain a generalization of the Hamel--Goulden
  formula, which contains a result of Jin, and prove a converse of the
  Hamel--Goulden theorem and its generalization.
\end{abstract}

\maketitle


\section{Introduction}

The Schur functions $s_\lambda$ are an important family of symmetric functions.
They form a linear basis of the space of symmetric functions and have
connections to different areas of mathematics including combinatorics and
representation theory. Schur functions have been extensively studied and there
are numerous generalizations and variations of them in the literature. In
particular, Macdonald \cite{Macdonald_Schur} introduced nine variations of Schur
functions. Macdonald's 9th variation of Schur functions generalize Schur
functions and many of their variations. The main focus of this paper is to find
determinant identities for Macdonald's 9th variation of Schur functions.

We briefly review several known determinant formulas for Schur functions. The
classical Jacobi--Trudi formula and its dual formula express a Schur function
$s_\lambda$ as a determinant in terms of complete homogeneous symmetric functions
$h_k$ and elementary symmetric functions $e_k$:
\[
  s_\lambda = \det \left( h_{\lambda_i+j-i} \right)_{i,j=1}^{\ell(\lambda)}, \qquad
  s_\lambda = \det \left( e_{\lambda'_i+j-i} \right)_{i,j=1}^{\ell(\lambda')},
\]
where $\ell(\lambda)$ is the number of parts in the partition $\lambda$ and
$\lambda'$ is the transpose of $\lambda$. Observe that $h_k$ and $e_k$ are also
Schur functions whose shapes are partitions with one row and one column,
respectively. The Giambelli formula \cite{Giambelli} and the Lascoux--Pragacz
formula \cite{Lascoux1988} express a Schur function as a determinant of Schur
functions whose shapes are hooks and border strips, respectively. Using
so-called outside decompositions, Hamel and Goulden \cite{Hamel_1995} found a
determinant formula for a Schur function, which generalizes all of the
aforementioned formulas. Chen, Yan, and Yang \cite{Chen_2005} restated the
Hamel--Goulden formula using certain border strips called cutting strips, and
showed that all these formulas are equivalent up to simple matrix operations.
Recently, Jin \cite{Jin_2018} generalized the Hamel--Goulden formula by
considering more general border strips called thickened border strips.

\medskip

In this paper, inspired by the original proof of the Lascoux--Pragacz formula in
\cite[Section~2]{Lascoux1988}, we find determinant identities for Macdonald's
9th variation of Schur functions using the Bazin identity \cite{Bazin}. Our
results restricted to Schur functions are also new and include the
Hamel--Goulden formula \cite{Hamel_1995} and its generalization due to Jin
\cite{Jin_2018}. As an application we prove a determinant identity (after some
correction) for factorial Schur functions conjectured by Morales, Pak, and
Panova \cite{MPP2}. Another special case of our results is the Hamel--Goulden
formula for Macdonald's 9th variation of Schur functions. We note that recently
Bachmann and Charlton \cite{Bachmann_2020}, and Foley and King \cite{Foley20}
also proved this using the Lindstr\"om--Gessel--Viennot lemma.

To illustrate, we compare one of our results with the Lascoux--Pragacz formula
\cite{Lascoux1988}.

\begin{thm}[Lascoux--Pragacz formula]\label{thm:LP}
  Suppose that $\mu\subseteq \lambda$ and $\theta=(\theta_1,\dots,\theta_k)$ is
  the Lascoux--Pragacz decomposition of $\lm$. Then
\[
s_{\lambda/\mu} = \det \left( s_{\lambda^0[p_j,q_i]}\right)_{i,j=1}^k, 
\]
where $p_i$'s and $q_i$'s are the contents of the starting and ending cells of
$\theta_i$, respectively.
\end{thm}

In the theorem above, the shape $\lambda^0[p_j,q_i]$ in each entry is a border
strip. See Sections~\ref{sec:definitions} and \ref{sec:lasc-prag-kreim} for the
undefined terms. The following theorem will be proved in
Section~\ref{sec:lasc-prag-kreim}.

\begin{thm}\label{thm:main_LP_intro}
  Let $\lambda$, $\mu$, and $\nu$ be partitions. Suppose that $\mu\subseteq \lambda$
  and $\theta=(\theta_1,\dots,\theta_k)$ is the Lascoux--Pragacz decomposition
  of $\lm$. Then we have
  \begin{align}
    \label{eq:main1intro}
    \ts_{\lambda/\nu}^{k-1} \ts_{\mu/\nu}
    &= \det \left( (-1)^{\chi(p_j>q_i)} \ts_{\lambda(q_i,p_j-1)/\nu}\right)_{i,j=1}^k,  \\
    \label{eq:main2intro}
    \ts_{\nu/\lambda}^{k-1} \ts_{\nu/\mu}
    &= \det \left( (-1)^{\chi(p_j>q_i)} \ts_{\nu/\lambda(q_i,p_j-1)}\right)_{i,j=1}^k.
  \end{align}
\end{thm}

Here $\ts_{\lm}$ is Macdonald's 9th variation of Schur functions and
$\lambda(q_i,p_j-1)$ is a partition obtained from $\lambda$ by either adding or
deleting a border strip. See Section~\ref{sec:definitions} for their precise
definitions. If $\nu=\lambda$ in \eqref{eq:main2intro} we obtain the
Lascoux--Pragacz formula for $\ts_{\lm}$, thus Theorem~\ref{thm:LP} follows. If
$\nu=\emptyset$ and Macdonald's 9th variation is specialized to the factorial
Schur functions in \eqref{eq:main1intro}, then we obtain a corrected version of
the conjecture of Morales, Pak, and Panova \cite{MPP2}.

Now we recall the reformulation of the Hamel--Goulden formula \cite{Hamel_1995}
due to Chen, Yan, and Yang \cite{Chen_2005}.
  
\begin{thm}[Hamel--Goulden formula]\label{thm:HG}
  Let $\lambda$ and $\mu$ be partitions and let $\gamma$ be a border strip.
  Suppose that $\mu\subseteq \lambda$ and $\theta=(\theta_1,\dots,\theta_k)$ is
  the decomposition of $\lm$ determined by the cutting strip $\gamma$. Then we
  have
\[
s_{\lambda/\mu} = \det \left( s_{\gamma[p_j,q_i]}\right)_{i,j=1}^k,
\]
where $p_i$ and $q_i$ are the contents of the starting and ending cells of
$\theta_i$, respectively.
\end{thm}

In Theorem~\ref{thm:HG}, $\gamma[p_j,q_i]$ is the set of cells in $\gamma$ whose
contents are in the closed interval $[p_j,q_i]$. See
Sections~\ref{sec:definitions} and \ref{sec:lasc-prag-kreim} for the undefined
terms. Observe that in Theorem~\ref{thm:HG} the $p_i$'s and $q_i$'s are
particular integers determined by $\lambda$, $\mu$, and $\gamma$. A natural
question is whether the determinant in this theorem can represent a skew Schur
function for arbitrary $p_i$'s and $q_i$'s. We show that this is indeed true up
to sign. This may be considered as a converse of the Hamel--Goulden theorem.
More generally, we can take $\gamma$ to be any connected skew shape, not
necessarily a border strip.

\begin{thm}\label{thm:HG3}
  Let $\alpha$ be any connected skew shape. Suppose that $(a_1,\dots,a_k)$ and
  $(b_1,\dots,b_k)$ are sequences of integers such that $\alpha[a_j,b_i]$ is a
  skew shape for all $i$ and $j$. Then either $\det \left(
    s_{\alpha[a_j,b_i]}\right)_{i,j=1}^k = 0$ or there exists a skew shape
  $\rho$ such that
\[
\det \left( s_{\alpha[a_j,b_i]}\right)_{i,j=1}^k = \pm s_{\rho}.
\]
\end{thm}

\medskip

The rest of this paper is organized as follows. In Section~\ref{sec:definitions}
we give basic definitions. In Section~\ref{sec:bazin-sylv-ident} we derive a
determinant identity involving $\ts_{\lm}$ using the Bazin identity.
In Section~\ref{sec:lasc-prag-kreim} we restate the result in the previous
section using border strip decompositions. As an application we prove a
conjecture of Morales, Pak, and Panova \cite{MPP2}. In
Section~\ref{sec:gener-hamel-gould} we prove a generalization of the
Hamel--Goulden formula. In Section~\ref{sec:conv-hamel-gould} we prove Theorem
\ref{thm:HG3}.

\section{Definitions}
\label{sec:definitions}

In this section we give basic definitions which will be used throughout this paper.

A \emph{partition} is a weakly decreasing sequence
$\lambda=(\lambda_1,\dots,\lambda_k)$ of positive integers. Each $\lambda_i>0$
is called a \emph{part} of $\lambda$. The \emph{length} $\ell(\lambda)$ of
$\lambda$ is the number of parts in $\lambda$. For integers $r>\ell(\lambda)$,
we use the convention $\lambda_r=0$. The set of partitions with at most $n$
parts is denoted by $\Par_n$. By appending zeros at the end if necessary we will
write each element $\lambda\in \Par_n$ as $\lambda=(\lambda_1,\dots,\lambda_n)$.

A pair $(i,j)$ of integers is called a \emph{cell}. The \emph{Young diagram} of
a partition $\lambda$ is the set of cells $(i,j)$ with $1\le i\le \ell(\lambda)$
and $1\le j\le \lambda_i$. We will often identity a partition $\lambda$ with its
Young diagram. The Young diagram $\lambda$ is visualized as an array of squares
so that there is a square in row $i$ and column $j$ for each $\lambda\in (i,j)$,
according to the matrix coordinates. See Figure~\ref{fig:yd}.

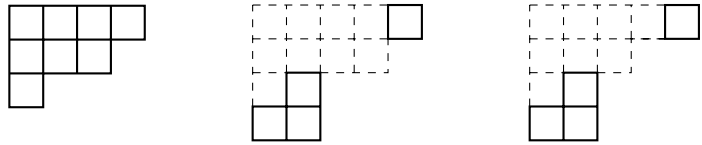
\begin{figure}
  \centering
  \begin{tikzpicture}[scale=.45]
\draw[thick] (0,0)--(1,0)--(1,1)--(3,1)--(3,2)--(4,2)--(4,3)--(0,3)--(0,0)--cycle;
\draw[thick] (0,1)--(1,1)
		(0,2)--(3,2)
		(1,1)--(1,3)
		(2,1)--(2,3)
		(3,2)--(3,3);
\node[] at (0,-.5) {\phantom{1}};
\end{tikzpicture}\qquad\qquad
 \begin{tikzpicture}[scale=.45]
\draw[thick] (0,0)--(2,0)--(2,2)--(1,2)--(1,1)--(0,1)--(0,0)--cycle
		(4,3)--(5,3)--(5,4)--(4,4)--(4,3)--cycle
		(1,0)--(1,1)--(2,1);
\draw[dashed] (0,1)--(0,4)
		(1,2)--(1,4)
		(2,2)--(2,4)
		(3,2)--(3,4)
		(0,4)--(4,4)
		(0,3)--(4,3)
		(0,2)--(1,2)
		(2,2)--(4,2)--(4,3);
\end{tikzpicture}\qquad\qquad
 \begin{tikzpicture}[scale=.45]
\draw[thick] (0,0)--(2,0)--(2,2)--(1,2)--(1,1)--(0,1)--(0,0)--cycle
		(4,3)--(5,3)--(5,4)--(4,4)--(4,3)--cycle
		(1,0)--(1,1)--(2,1);
\draw[dashed] (0,1)--(0,4)
		(1,2)--(1,4)
		(2,2)--(2,4)
		(0,4)--(4,4)
		(0,3)--(4,3)
		(0,2)--(1,2)
		(2,2)--(3,2)--(3,3)--(4,3)
		(3,3)--(3,4);
\end{tikzpicture}
  \caption{The Young diagram of the partition $(4,3,1)$ (left), the skew shape
    $(5,4,2,2)/(4,4,1)$ (middle), and the skew shape $(5,3,2,2)/(4,3,1)$
    (right).}
  \label{fig:yd}
\end{figure}

For two partitions $\lambda$ and $\mu$, we write $\mu\subseteq\lambda$ to mean
that the Young diagram of $\mu$ is contained in that of $\lambda$. A \emph{skew
  shape}, denoted by $\lm$, is a pair $(\lambda,\mu)$ of partitions satisfying
$\mu\subseteq\lambda$. We also consider the skew shape $\lm$ as the
set-theoretic difference $\lambda\setminus\mu$ of their Young diagrams. Note,
however, that when we consider a skew shape $\lm$ we have the information on the
partitions $\lambda$ and $\mu$ as well as the difference $\lambda-\mu$ of their
Young diagrams. For example, the two skew shapes in Figure~\ref{fig:yd} have the
same set of cells but are considered as different skew shapes.

For a cell $x=(i,j)$, the \emph{content} $c(x)$ of $x$ is defined by $c(x)=j-i$.
For a skew shape $\alpha$, we define
\[
\Cont(\alpha) = \{c(x):x\in \alpha\}.
\]
For a skew shape $\alpha$ we define $\alpha+(r,s)$ to be the skew shape obtained
by shifting $\alpha$ by $(r,s)$, i.e.,
\[
\alpha + (r,s) = \{x+(r,s): x\in \alpha\}.
\]

In this paper ``connected'' means edgewise connected. A \emph{connected
  component} of a skew shape $\alpha$ is a maximal connected subset of $\alpha$,
see Figure~\ref{fig:connected}. 

\begin{figure}
  \centering
 \begin{tikzpicture}[scale=.45]
\draw[thick] (0,0)--(0,2)--(2,2)--(2,3)--(3,3)--(3,4)--(6,4)--(6,1)--(4,1)--(4,0)--(0,0)
		(3,3)--(6,3)
		(2,2)--(6,2)
		(0,1)--(4,1)
		(1,0)--(1,2)
		(2,0)--(2,2)
		(3,0)--(3,3)
		(4,1)--(4,4)
		(5,1)--(5,4);
\draw[dashed] (0,2)--(0,4)--(3,4)
		(0,3)--(2,3)
		(1,2)--(1,4)
		(2,3)--(2,4);
\node[] at (0,-.5) {\phantom{1}};
\end{tikzpicture}\qquad\qquad\qquad
 \begin{tikzpicture}[scale=.45]
\draw[thick] (0,0)--(0,2)--(2,2)--(2,3)--(3,3)--(3,1)--(2,1)--(2,0)--(0,0)--cycle
		(0,1)--(2,1)--(2,2)--(3,2)
		(1,0)--(1,2)
		(3,3)--(5,3)--(5,5)--(4,5)--(4,4)--(3,4)--(3,3)
		(4,3)--(4,4)--(5,4);
\draw[dashed] (0,2)--(0,5)--(4,5)
		(0,4)--(3,4)
		(0,3)--(2,3)
		(1,2)--(1,5)
		(2,3)--(2,5)
		(3,4)--(3,5);
\end{tikzpicture}
  \caption{The skew shape on the left is connected, whereas the skew shape on
    the right is disconnected. There are two connected components in the right
    skew shape.}
  \label{fig:connected}
\end{figure}
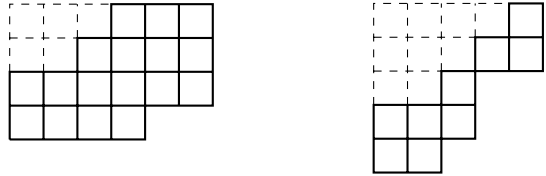

A \emph{border strip} is a connected skew shape that contains no $2\times 2$ block of 
squares. For a border strip $\gamma$, the contents of the starting and ending
cells of $\gamma$ are denoted by $p(\gamma)$ and $q(\gamma)$, respectively:
\[
p(\gamma)=\min(\Cont(\gamma)),\qquad q(\gamma)=\max(\Cont(\gamma)).
\]

For a border strip $\gamma$ and integers $a$ and $b$, we define
\[
\gamma[a,b] = 
\begin{cases}
\{x\in \gamma: a\le c(x)\le b\} & \mbox{if $a\le b$},\\
\emptyset & \mbox{if $a=b+1$}, \\
\mbox{undefined} & \mbox{if $a>b+1$}.
\end{cases}
\]
See Figure~\ref{fig:gamma}. If $\gamma[a,b]$ is undefined, then we define
$\ts_{\gamma[a,b]}$ to be $0$.

\begin{figure}
  \centering
 \begin{tikzpicture}[scale=.5]
\draw[densely dashed] (0,1)--(0,4);
\draw[densely dashed] (1,1)--(1,4);
\draw[densely dashed] (2,2)--(2,4);
\draw[densely dashed] (3,3)--(3,4);
\draw[densely dashed] (4,3)--(4,4);
\draw[densely dashed] (0,4)--(4,4);
\draw[densely dashed] (0,3)--(4,3);
\draw[densely dashed] (0,2)--(2,2);
\draw[densely dashed] (0,1)--(1,1);
\draw[thick] (0,0)--(0,1)--(1,1)--(1,2)--(3,2)--(3,3)--(4,3)--(4,4)--(5,4)--(5,3)--(4,3)--(4,1)--(2,1)--(2,0)--(0,0)--cycle;
\draw[thick] (4,2)--(5,2)--(5,3)--(6,3)--(6,4)--(5,4);
\draw[thick] (1,1)--(4,1);
\draw[thick] (2,2)--(4,2);
\draw[thick] (1,0)--(1,1);
\draw[thick] (2,1)--(2,2);
\draw[thick] (3,1)--(3,3);
\node[] at (4.5, 3.5) {$4$};
\node[] at (5.5, 3.5) {$5$};
\node[] at (3.5, 2.5) {$2$};
\node[] at (4.5, 2.5) {$3$};
\node[] at (1.5, 1.5) {$-1$};
\node[] at (2.5, 1.5) {$0$};
\node[] at (3.5, 1.5) {$1$};
\node[] at (.5, .5) {$-3$};
\node[] at (1.5, .5) {$-2$};
\node[] at (-1.6,2) { $ \gamma=$};
\end{tikzpicture} \qquad
\begin{tikzpicture}[scale=.5]
\draw[thick] (0,0)-- (0,2)--(2,2)--(2,3)--(4,3)--(4,2)--(3,2)--(3,1)--(1,1)--(1,0)--(0,0)--cycle;
\draw[thick] (0,1)--(1,1)
		(1,1)--(1,2)
		(2,1)--(2,2)
		(2,2)--(3,2)
		(3,2)--(3,3);
\node[] at (.5, .5) {$-2$};
\node[] at (.5, 1.5) {$-1$};
\node[] at (1.5, 1.5) {$0$};
\node[] at (2.5, 1.5) {$1$};
\node[] at (2.5, 2.5) {$2$};
\node[] at (3.5, 2.5) {$3$};
\node[] at (-2.5, 2)  {$\gamma [-2, 3] $=};
\end{tikzpicture}
\caption{The left diagram shows a border strip $\gamma$ with $p(\gamma)=-3$ and
  $q(\gamma)=5$. The right diagram shows $\gamma[-2,3]$. In both diagrams the
  contents of the cells are shown. We have $\gamma[3,2]=\emptyset$ and
  $\gamma[2,-1]$ is undefined.}
  \label{fig:gamma}
\end{figure}
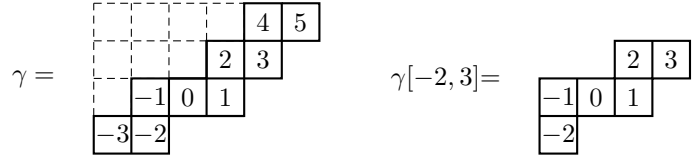

Consider a skew shape $\alpha$ and a border strip $\gamma$ such that
$\Cont(\alpha)\subseteq\Cont(\gamma)$. Consider the diagonal shifts
$\gamma+(i,i)$ of $\gamma$, for $i\in\ZZ$, that cover $\alpha$. The intersection
of each diagonal shift of $\gamma$ with $\alpha$ is a union of border strips.
Let $\theta$ be the collection of the border strips obtained in this way. Then
$\theta$ is a decomposition of $\alpha$. In this case we say that $\gamma$ is
the \emph{cutting strip} of $\theta$. See Figure~\ref{fig:cutting} for an
example.

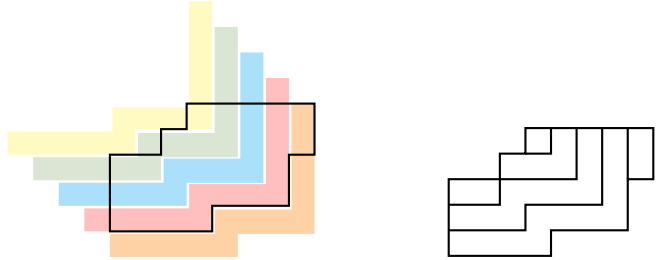
\begin{figure}
  \centering
\begin{tikzpicture}[scale=.34]
\fill[fill=yellow!30!white] (-4,3) rectangle (1,3.9)
				   (.1, 3.9) rectangle (1, 4)
				     (.1,4) rectangle (4,4.85)
				     (3.1,4.85) rectangle (4,9);
\fill[fill=asparagus!30!white] (-3,2) rectangle (2,2.9)
				     (1.1,2.9) rectangle (5,3.85)
				     (4.1,3.85) rectangle (5,8);
\fill[fill=cyan!30!white] (-2,1) rectangle (3,1.9)
				     (2.1,1.9) rectangle (6,2.85)
				     (5.1,2.85) rectangle (6,7);
\fill[fill=pink] (-1,0) rectangle (4,.9)
				     (3.1, .9) rectangle (4, 1)
				     (3.1,1) rectangle (7,1.85)
				     (6.1,1.85) rectangle (7,6);
\fill[fill=orange!35!white] (0,-1) rectangle (5,-.1)
				     (4.1,-.1) rectangle (8,.85)
				     (7.1,.85) rectangle (8,5);
\draw[thick] (0,0)--(4,0)--(4,1)--(7,1)--(7,3)--(8,3)--(8,5)--(3,5)--(3,4)--(2,4)--(2,3)--(0,3)--(0,0)--cycle;
\end{tikzpicture}\qquad\qquad\quad
\begin{tikzpicture}[scale=.34]
\draw[thick] (0,0)--(4,0)--(4,1)--(7,1)--(7,3)--(8,3)--(8,5)--(3,5)--(3,4)--(2,4)--(2,3)--(0,3)--(0,0)--cycle;
\draw[thick] (7,3)--(7,5);
\draw[thick] (0,1)--(3,1)--(3,2)--(6,2)--(6,5);
\draw[thick] (0,2)--(2,2)--(2,3)--(5,3)--(5,5);
\draw[thick] (3,4)--(4,4)--(4,5);
\end{tikzpicture}
\caption{The left diagram shows the diagonal shifts of a border strip $\gamma$
  covering a skew shape $\alpha$. The right diagram shows the decomposition
  $\theta$ with respect to the cutting strip $\gamma$.}
  \label{fig:cutting}
\end{figure}

For a partition $\lambda$ with at most $n$ parts, let
\[
C_n(\lambda)=\{\lambda_i-i: 1\le i\le n\}.
\]
See Figure~\ref{fig:C}. Note that for any partition $\lambda\in\Par_n$, we have
$|C_n(\lambda)|=n$ and $\min(C_n(\lambda))\ge -n$. Conversely, one can easily
see that for any set $C$ of integers with $|C|=n$ and $\min(C)\ge -n$, there is
a unique partition $\lambda\in\Par_n$ with $C_n(\lambda)=C$.

\begin{figure}
  \centering
  \begin{tikzpicture}[scale=.5] 
\draw[thick] (0,0)--(0,6)--(6,6)--(6,4)--(4,4)--(4,3)--(2,3)--(2,2)--(0,2);
\draw[color=gray] (0,5)--(6,5)
			(0,4)--(4,4)
			(0,3)--(2,3)
			(1,2)--(1,6)
			(2,3)--(2,6)
			(3,3)--(3,6)
			(4,4)--(4,6)
			(5,4)--(5,6);
\node[] at (5.5, 5.5) {$5$};
\node[] at (4.5, 5.5) {$4$};
\node[] at (3.5, 5.5) {$3$};
\node[] at (2.5, 5.5) {$2$};
\node[] at (1.5, 5.5) {$1$};
\node[] at (.5, 5.5) {$0$};
\node[] at (-.5, 5.5) {$-1$};
\node[] at (5.5, 4.5) {$4$};
\node[] at (4.5, 4.5) {$3$};
\node[] at (3.5, 4.5) {$2$};
\node[] at (2.5, 4.5) {$1$};
\node[] at (1.5, 4.5) {$0$};
\node[] at (.5, 4.5) {$-1$};
\node[] at (-.5, 4.5) {$-2$};
\node[] at (3.5, 3.5) {$1$};
\node[] at (2.5, 3.5) {$0$};
\node[] at (1.5, 3.5) {$-1$};
\node[] at (.5, 3.5) {$-2$};
\node[] at (-.5, 3.5) {$-3$};
\node[] at (1.5, 2.5) {$-2$};
\node[] at (.5, 2.5) {$-3$};
\node[] at (-.5, 2.5) {$-4$};
\node[] at (-.5, 1.5) {$-5$};
\node[] at (-.5, .5) {$-6$};
\draw[thick, dredcolor] (-.5, .5) circle (13pt);
\draw[thick, dredcolor] (-.5, 1.5) circle (13pt);
\draw[thick, dredcolor] (1.5, 2.5) circle (13pt);
\draw[thick, dredcolor] (3.5, 3.5) circle (13pt);
\draw[thick, dredcolor] (5.5, 4.5) circle (13pt);
\draw[thick, dredcolor] (5.5, 5.5) circle (13pt);
\end{tikzpicture}
\caption{The Young diagram of $\lambda=(6,6,4,2)$, where the contents of the
  cells in $\lambda\cup\{(i,0):1\le i\le 6\}$ are shown. The circled integers
  are the elements in $C_6(\lambda)$.}
  \label{fig:C}
\end{figure}
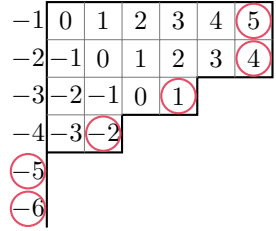

\begin{defn}
  The \emph{outer strip} $\lambda^0$ of a partition $\lambda$ is the set of
  cells $x\in \lambda$ such that $x+(1,1)\not\in \lambda$. The \emph{extended
    outer strip} $\lambda^+$ of $\lambda$ is the (infinite) set of cells $x\in
  \ZZ_{>0}^2\setminus\lambda$ such that $x+(-1,-1)\not\in
  \ZZ_{>0}^2\setminus\lambda$. See Figure~\ref{fig:outer2}.
\end{defn}

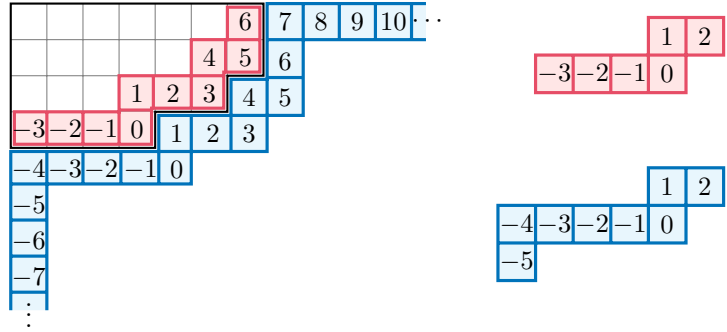
\begin{figure}
  \centering
 \begin{tikzpicture}[scale=.48] 
\draw[color=gray] (0,6)--(4,6)
			(0,7)--(6,7)
			(0,8)--(7,8)
			(1,5)--(1,9)
			(2,5)--(2,9)
			(3,5)--(3,9)
			(4,6)--(4,9)
			(5,6)--(5,9)
			(6,7)--(6,9);
\draw[thick] (0,5)--(4,5)--(4,6)--(6,6)--(6,7)--(7,7)--(7,9)--(0,9)--(0,5)--cycle;
\fill[fill=pink!40!white] (.1, 5.1) rectangle (3.9,6)
				(3, 6.1) rectangle (5.9, 7)
				(5, 7) rectangle (5.95, 8)
				(6, 7.1) rectangle (6.9, 8.9);
\draw[very thick, color=dredcolor] (0.1,5.1)--(3.9,5.1)--(3.9,6.1)--(5.9, 6.1)--(5.9,7.1)--(6.9,7.1)--(6.9,8.9)--(6,8.9)--(6,8)--(5,8)--(5,7)--(3,7)--(3,6)--(.1, 6)--(.1, 5.1)--cycle
						(1,5.1)--(1,6)
						(2,5.1)--(2,6)
						(3,5.1)--(3,6)
						(3,6.05)--(3.9, 6.05)
						(3.95, 6.1)--(3.95, 7)
						(5, 6.1)--(5, 7)--(5.9, 7)
						(5.95, 7.1)--(5.95, 8)--(6.9, 8);
\fill[fill=cyan!8!white] (0, .5) rectangle (1, 4.9)
				(1, 4) rectangle (5, 4.9)
				(4.1, 4.9) rectangle (7.1, 5.9)
				(6.1, 5.9) rectangle (7.1, 6.9)
				(7.1, 5.9) rectangle (8.1,9)
				(8.1, 8) rectangle (11.5, 9);
\draw[very thick, color=frenchblue] (0, .5)--(0, 4.9)--(4.1,4.9)--(4.1, 5.9)--(6.1, 5.9)--(6.1, 6.9)--(7.1, 6.9)--(7.1, 9)--(11.5, 9)
						(1, .5)--(1,4.9)
						(0, 4)--(5,4)--(5, 5.9)
						(4.1, 4.9)--(7.1,4.9)--(7.1, 6.9)
						(6.1, 5.9)--(8.1, 5.9)--(8.1, 9)
						(7.1, 8)--(11.5, 8)
						(0,1)--(1,1)
						(0,2)--(1,2)
						(0,3)--(1,3)
						(2,4)--(2,4.9)
						(3.05,4)--(3.05,4.9)
						(4.1, 4)--(4.1, 4.9)
						(6.1, 4.9)--(6.1, 5.9)
						(7.1, 6.95)--(8.1,6.95)
						(9.1, 8)--(9.1,9)
						(10.1, 8)--(10.1,9)
						(11.1, 8)--(11.1,9);
\node [] at (.5, 5.5) {$-3$};
\node [] at (1.5, 5.5) {$-2$};
\node [] at (2.5, 5.5) {$-1$};
\node [] at (3.5, 5.5) {$0$};
\node [] at (3.5, 6.5) {$1$};
\node [] at (4.5, 6.5) {$2$};
\node [] at (5.5, 6.5) {$3$};
\node [] at (5.5, 7.5) {$4$};
\node [] at (6.5, 7.5) {$5$};
\node [] at (6.5, 8.5) {$6$};
\node [] at (.5, .6) {$\vdots$};
\node [] at (.5, 1.4) {$-7$};
\node [] at (.5, 2.4) {$-6$};
\node [] at (.5, 3.4) {$-5$};
\node [] at (.5, 4.4) {$-4$};
\node [] at (1.5, 4.4) {$-3$};
\node [] at (2.5, 4.4) {$-2$};
\node [] at (3.6, 4.4) {$-1$};
\node [] at (4.6, 4.4) {$0$};
\node [] at (4.6, 5.4) {$1$};
\node [] at (5.6, 5.4) {$2$};
\node [] at (6.6, 5.4) {$3$};
\node [] at (6.6, 6.4) {$4$};
\node [] at (7.6, 6.4) {$5$};
\node [] at (7.6, 7.4) {$6$};
\node [] at (7.6, 8.5) {$7$};
\node [] at (8.6, 8.5) {$8$};
\node [] at (9.6, 8.5) {$9$};
 \node [] at (10.6, 8.5) {$10$};
\node [] at (11.6, 8.5) {$\cdots$};
\end{tikzpicture}\quad
 \begin{tikzpicture}[scale=.5]  
 \fill[fill=cyan!8!white] (0,0) rectangle (1,2)
 				(1,1) rectangle (5,2)
				(4,2) rectangle (6,3);
 \draw[very thick, color=frenchblue] (0,0)--(1,0)--(1,1)--(5,1)--(5,2)--(6,2)--(6,3)--(4,3)--(4,2)--(0,2)--(0,0)--cycle
 						(0,1)--(1,1)--(1,2)
						(2,1)--(2,2)
						(3,1)--(3,2)
						(4,1)--(4,2)--(5,2)--(5,3);
\node [] at (.5, .5) {$-5$};
\node [] at (.5, 1.5) {$-4$};
\node [] at (1.5, 1.5) {$-3$};
\node [] at (2.5, 1.5) {$-2$};
\node [] at (3.5, 1.5) {$-1$};
\node [] at (4.5, 1.5) {$0$};
\node [] at (4.5, 2.5) {$1$};
\node [] at (5.5, 2.5) {$2$};
\fill[fill=pink!40!white] (1,5) rectangle (5,6)
				(4,6) rectangle (6,7);
\draw[very thick, color=dredcolor] (1,5)--(5,5)--(5,6)--(6,6)--(6,7)--(4,7)--(4,6)--(1,6)--(1,5)--cycle
						(2,5)--(2,6)
						(3,5)--(3,6)
						(4,5)--(4,6)--(5,6)--(5,7);
\node [] at (1.5, 5.5) {$-3$};
\node [] at (2.5, 5.5) {$-2$};
\node [] at (3.5, 5.5) {$-1$};
\node [] at (4.5, 5.5) {$0$};
\node [] at (4.5, 6.5) {$1$};
\node [] at (5.5, 6.5) {$2$};
\node[] at (0,-1) {\phantom{1}};
 \end{tikzpicture}  
 \caption{The left diagram shows the partition $\lambda=(7,7,6,4)$, where its
   outer strip $\lambda^0$ and extended outer strip $\lambda^+$ are colored red
   and blue, respectively. On the right are shown $\lambda^0[-5,2]$ at the top
   and $\lambda^+[-5,2]$ at the bottom.}
  \label{fig:outer2}
\end{figure}

Note that by definition, for $p\le q$, we always have
$\lambda^0[p,q]\subseteq\lambda$, but $\lambda^+[p,q]\not\subseteq\lambda$.

\begin{defn}
  Consider a partition $\lambda\in\Par_n$ and integers $a,b\ge -n$ such that
  $a\in C_n(\lambda)$ and $b\not\in C_n(\lambda)\setminus\{a\}$. We define
  $\lambda(a,b)$ to be the unique partition $\mu\in\Par_n$ satisfying
\begin{equation}
  \label{eq:C_n}
C_n(\mu) =(C_n(\lambda) \setminus\{a\}) \cup \{b\}.  
\end{equation}
\end{defn}

\begin{figure}
  \centering
  \begin{tikzpicture}[scale=.48] 
\fill[fill=pink!50!white] (3, 3) rectangle (4, 5);
\fill[fill=pink!50!white] (3, 4) rectangle (6, 5);
\draw[color=gray] (0,5)--(6,5)
			(0,4)--(3.9,3.9)
			(0,3)--(2,3)
			(1,2)--(1,6)
			(2,3)--(2,6)
			(3,3)--(3,6)
			(4,4.1)--(4,6)
			(5,4.1)--(5,6);
\node[] at (5.5, 5.5) {$5$};
\node[] at (4.5, 5.5) {$4$};
\node[] at (3.5, 5.5) {$3$};
\node[] at (2.5, 5.5) {$2$};
\node[] at (1.5, 5.5) {$1$};
\node[] at (.5, 5.5) {$0$};
\node[] at (-.6, 5.5) {$-1$};
\node[] at (5.5, 4.5) {$4$};
\node[] at (4.5, 4.5) {$3$};
\node[] at (3.5, 4.5) {$2$};
\node[] at (2.5, 4.5) {$1$};
\node[] at (1.5, 4.5) {$0$};
\node[] at (.5, 4.5) {$-1$};
\node[] at (-.6, 4.5) {$-2$};
\node[] at (3.5, 3.5) {$1$};
\node[] at (2.5, 3.5) {$0$};
\node[] at (1.5, 3.5) {$-1$};
\node[] at (.5, 3.5) {$-2$};
\node[] at (-.6, 3.5) {$-3$};
\node[] at (1.5, 2.5) {$-2$};
\node[] at (.5, 2.5) {$-3$};
\node[] at (-.6, 2.5) {$-4$};
\node[] at (-.6, 1.5) {$-5$};
\node[] at (-.6, .5) {$-6$};
\draw[thick, color=dredcolor!60!white] (3,3)--(4,3)--(4, 4)--(6, 4)--(6, 5)--(3, 5)--(3,3)--cycle;
\draw[thick] (0,0)--(0,6)--(6,6)--(6,5)--(3,5)--(3,3)--(2,3)--(2,2)--(0,2);
\draw[thick, dredcolor] (-.6, .5) circle (13pt);
\draw[thick, dredcolor] (-.6, 1.5) circle (13pt);
\draw[thick, dredcolor] (1.5, 2.5) circle (13pt);
\draw[thick, dredcolor] (2.5, 3.5) circle (13pt);
\draw[thick, dredcolor] (2.5, 4.5) circle (13pt);
\draw[thick, dredcolor] (5.5, 5.5) circle (13pt);
\node[] at (5.6, 3.5) {\small\color{red}removed};
\end{tikzpicture}\qquad\qquad\quad
\begin{tikzpicture}[scale=.48] 
\fill[fill=cyan!14!white] (0,1) rectangle (3,2);
\fill[fill=cyan!14!white] (2,2) rectangle (5,3);
\fill[fill=cyan!14!white] (4,3) rectangle (5,4);
\draw[thick, cyan] (0,1)--(3,1)--(3,2)--(5,2)--(5,4)--(4,4)--(4,3)--(2,3)--(2,2)--(0,2)--(0,1)--cycle;
\draw[thick] (0,0)--(0,6)--(6,6)--(6,4)--(5,4)--(5,2)--(3,2)--(3,1)--(0,1);
\draw[color=gray] (0,5)--(6,5)
			(0,4)--(5,4)
			(0,3)--(5,3)
			(0,2)--(3,2)
			(1,1)--(1,6)
			(2,1)--(2,6)
			(3,2)--(3,6)
			(4,2)--(4,6)
			(5,4)--(5,6);
\node[] at (5.5, 5.5) {$5$};
\node[] at (4.5, 5.5) {$4$};
\node[] at (3.5, 5.5) {$3$};
\node[] at (2.5, 5.5) {$2$};
\node[] at (1.5, 5.5) {$1$};
\node[] at (.5, 5.5) {$0$};
\node[] at (-.6, 5.5) {$ -1$};
\node[] at (5.5, 4.5) {$4$};
\node[] at (4.5, 4.5) {$3$};
\node[] at (3.5, 4.5) {$2$};
\node[] at (2.5, 4.5) {$1$};
\node[] at (1.5, 4.5) {$0$};
\node[] at (.5, 4.5) {$-1$};
\node[] at (-.6, 4.5) {$-2$};
\node[] at (4.5, 3.5) {$2$};
\node[] at (3.5, 3.5) {$1$};
\node[] at (2.5, 3.5) {$0$};
\node[] at (1.5, 3.5) {$-1$};
\node[] at (.5, 3.5) {$-2$};
\node[] at (-.6, 3.5) {$-3$};
\node[] at (1.5, 2.5) {$-2$};
\node[] at (.5, 2.5) {$-3$};
\node[] at (-.6, 2.5) {$-4$};
\node[] at (-.6, 1.5) {$-5$};
\node[] at (-.6, .5) {$-6$};
\node[] at (2.5, 1.5) {$-2$};
\node[] at (1.5, 1.5) {$-3$};
\node[] at (.5, 1.5) {$-4$};
\node[] at (4.5, 2.5) {$1$};
\node[] at (3.5, 2.5) {$0$};
\node[] at (2.5, 2.5) {$-1$};
\draw[thick, dredcolor] (-.6, .5) circle (13pt);
\draw[thick, dredcolor] (2.5, 1.5) circle (13pt);
\draw[thick, dredcolor] (4.5, 2.5) circle (13pt);
\draw[thick, dredcolor] (4.5, 3.5) circle (13pt);
\draw[thick, dredcolor] (5.5, 4.5) circle (13pt);
\draw[thick, dredcolor] (5.5, 5.5) circle (13pt);
\node[] at (4.3, 1.5) {\small\color{blue}added};
\end{tikzpicture}
\caption{For the partition $\lambda=(6,6,4,2)$ in Figure~\ref{fig:C}, the left
  is $\lambda(4,0)=\lambda\setminus \lambda^0[1,4]$ and the right is
  $\lambda(-5,2)=\lambda\cup \lambda^+ [-4,2]$.}
  \label{fig:la(a,b)}
\end{figure}
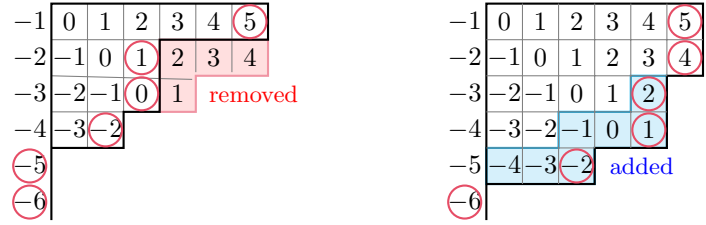

See Figure~\ref{fig:la(a,b)}. It is easy to check that
under the above conditions on $a$ and $b$,
the partition $\lambda(a,b)$ is well defined and
\begin{equation}\label{eq:la(a,b)}
\lambda(a,b) =
\begin{cases}
\lambda & \mbox{if $a=b$},\\
\lambda\setminus\lambda^0[b+1,a] & \mbox{if $a>b$}, \\
\lambda\cup\lambda^+[a+1,b] & \mbox{if $a<b$}.
\end{cases}
\end{equation}

\begin{remark}
  The notations $\lambda^0[p,q]$ and $\lambda(p,q)$ will be used frequently
  throughout this paper. Note that $\lambda^0[p,q]$ is a border strip and
  $\lambda(p,q)$ is a partition obtained from $\lambda$ by adding or deleting a
  (possibly empty) border strip. The notation $\lambda/\mu(a,b)$ always means
  $\lambda/(\mu(a,b))$ and not $(\lambda/\mu)(a,b)$.
\end{remark}

Finally, we define Macdonald's 9th variation of Schur functions.

\begin{defn}
Let $h_{r,s}$, for $r, s\in\ZZ$ with $r\ge1$, be independent indeterminates.
Define $h_{0,s}=1$ and $h_{r,s}=0$ for all $r<0$ and $s\in \ZZ$. For any
partitions $\lambda$ and $\mu$ with at most $n$ parts, the \emph{Macdonald's
  9th variation} $\ts_{\lm}$ of Schur functions is defined by
\[
\ts_{\lm} = \det\left(h_{\lambda_i-\mu_j-i+j,\mu_j-j+1}\right)_{i,j=1}^n.
\]
\end{defn}

If $h_{r,s}$ is set to be equal to the complete homogeneous symmetric function
$h_r$ for all $r$ and $s$, then $\ts_{\lm}$ reduces to the Schur function
$s_{\lm}$. Note that the Schur functions have the property that
$s_{\alpha}=s_{\alpha+(r,s)}$ for any skew shape $\alpha$ and integers $r,s$
such that $\alpha+(r,s)$ is a skew shape. However,
$\ts_{\alpha}\ne\ts_{\alpha+(r,s)}$ unless $r=s$.

\section{The Bazin identity}
\label{sec:bazin-sylv-ident}

In this section we recall the Bazin identity and derive 
a determinant identity for Macdonald's 9th variation of Schur functions from it.

\begin{defn}
Let $M=(M_{ij})_{i\in \ZZ, 1\le j\le n}$ be a matrix whose rows and columns are indexed by $\ZZ$ and $\{1,2,\dots,n\}$ respectively. 
For any sequence $\vec a=(a_1,\dots,a_n)$ of $n$ integers (not necessarily distinct), we define
\[
M[\vec a] = \det (M_{a_i,j})_{1\le i,j\le n}.
\]
Note that if $\vec a$ has repeated elements, then $[\vec a]=0$, and otherwise $[\vec a]$ is, up to sign, equal to the minor of $M$ obtained by selecting the rows of $M$ indexed by the integers in $\vec a$. 
  
\end{defn}

For sequences $\vec a = (a_1,\dots,a_r)$ and $\vec b = (b_1,\dots,b_s)$ of
integers, let
\begin{align*}
\vec a \sqcup \vec b &= (a_1,\dots,a_r,b_1,\dots,b_s),\\
\vec a \setminus a_i &= (a_1,\dots,a_{i-1},a_{i+1},\dots,a_r).
\end{align*}
An integer $b$ is also considered as the sequence $(b)$ consisting of only one
element. For example, 
\[
b \sqcup \vec a= (b)\sqcup \vec a = (b,a_1,\dots,a_r).
\]

The key lemma in this paper is the following result proved by Bazin \cite{Bazin}
in 1851, see also \cite[Lemma~2.1]{Lascoux1988}.

\begin{lem}[Bazin identity]\label{lem:BS}
Let $\vec a=(a_1,\dots,a_k), \vec b=(b_1,\dots,b_k)$ and $\vec c=(c_1,\dots,c_{n-k})$ be any sequences of integers. Then
\[ [\vec a \sqcup \vec c]^{k-1} [\vec b\sqcup\vec c] = (-1)^{\binom k2}\det
  ([b_j\sqcup (\vec a\setminus a_i)\sqcup \vec c])_{i,j=1}^k.
\]
\end{lem}

\begin{remark}
  The original statement of the Bazin identity is 
\[
  [\vec a \sqcup \vec c]^{k-1} [\vec b\sqcup\vec c]
  = \det ([(a_1,\dots,a_{i-1},b_j,a_{i+1},\dots,a_k) \sqcup \vec c])_{i,j=1}^k.
\]
Since
\[
[(a_1,\dots,a_{i-1},b_j,a_{i+1},\dots,a_k) \sqcup \vec c] = 
(-1)^{i-1}[b_j\sqcup (\vec a\setminus a_i)\sqcup \vec c],
\]
the above identity is equivalent to the one in Lemma~\ref{lem:BS}. The Bazin
identity is also attributed to Sylvester (see \cite[p.~563]{Lascoux1988}), Reiss,
and Picquet (see \cite[p.~195]{Doubilet_1974}).

We note that Okada~\cite[Corollary~3.2]{okada17:sylvester} found a more general
determinant identity and derived the Bazin identity from it.
\end{remark}

For a permutation $\pi$ of $\{1,2,\dots,n\}$ we denote by $\inv(\pi)$ the number
of pairs $(i,j)$ of integers $1\le i<j\le n$ satisfying $\pi(i)>\pi(j)$. For two
sequences $\vec x = (x_1,\dots, x_r)$ and $\vec y = (y_1,\dots,y_s)$ of
integers, let $\inv(\vec x, \vec y)$ denote the number of pairs $(i,j)$ of
integers $1\le i\le r$ and $1\le j\le s$ satisfying $x_i>y_j$.

For a statement $p$ we define $\chi(p)=1$ if $p$ is true and $\chi(p)=0$
otherwise.

Now we derive a determinant identity for Macdonald's 9th variation of Schur
functions using the Bazin identity. In later sections we will give
some applications of this identity.

\begin{thm}\label{thm:main}
  Let $\lambda$, $\mu$, and $\nu$ be partitions with at most $n$ parts. Suppose
  that $\vec a=(a_1,\dots,a_k)$ and $\vec b=(b_1,\dots,b_k)$ are the sequences
  defined by 
\begin{align*}
  C_n(\lambda)\setminus C_n(\mu) &= \{a_1>a_2>\dots>a_k\},\\
  C_n(\mu)\setminus C_n(\lambda) &= \{b_1>b_2>\dots>b_k\}.
\end{align*}
Then we have
\begin{equation}
  \label{eq:main1}
\ts_{\lambda/\nu}^{k-1} \ts_{\mu/\nu} =(-1)^{\inv(\vec b,\vec a)+\binom k2}\det \left( (-1)^{\chi(b_j>a_i)} \ts_{\lambda(a_i,b_j)/\nu}\right)_{i,j=1}^k,  
\end{equation}
and
\begin{equation}
  \label{eq:main2}
\ts_{\nu/\lambda}^{k-1} \ts_{\nu/\mu} =(-1)^{\inv(\vec b,\vec a) +\binom k2}\det \left( (-1)^{\chi(b_j>a_i)} \ts_{\nu/\lambda(a_i,b_j)}\right)_{i,j=1}^k.
\end{equation}
\end{thm}

\begin{proof}
Let 
\[
M=\left(h_{i-\nu_j+j,\nu_j-j+1}\right)_{i\in\ZZ, 1\le j\le n}
\]
and $\vec c =(c_1,\dots,c_{n-k})$, where
$c_1>\cdots>c_{n-k}$ are the elements of $C_n(\lambda)\setminus \vec a = C_n(\mu)\setminus \vec b$ written in decreasing order. Then Lemma~\ref{lem:BS} says that
\begin{equation}
  \label{eq:2}
[\vec a \sqcup \vec c]^{k-1} [\vec b\sqcup\vec c] = (-1)^{\binom k2}\det ([b_j\sqcup (\vec a\setminus a_i)\sqcup \vec c])_{i,j=1}^k.  
\end{equation}

Since $\vec a \sqcup \vec c$ is a rearrangment of the decreasing sequence
$(\lambda_1-1,\lambda_2-2, \dots,\lambda_n-n)$, we have
\begin{equation}
  \label{eq:5}
[\vec a \sqcup\vec c] = (-1)^{\inv(\vec c,\vec a)} \det(h_{\lambda_i-i - \nu_j+j,\nu_j-j+1})_{i,j=1}^n 
= (-1)^{\inv(\vec c,\vec a)}  \ts_{\lambda/\nu}.
\end{equation}
Similarly, we have
\begin{equation}
  \label{eq:6}
[\vec b\sqcup\vec c] =  (-1)^{\inv(\vec c,\vec b)} \ts_{\mu/\nu}.
\end{equation}

Now, we claim that
\begin{equation}
  \label{eq:3}
[b_j\sqcup (\vec a\setminus a_i)\sqcup \vec c] = 
(-1)^{\inv(\vec c, \vec a) - \inv(\vec c, a_i)+\inv(\vec c, b_j) + \inv(\vec a, b_j)+\chi(a_i>b_j)} \ts_{\lambda(a_i,b_j)/\nu}.
\end{equation}
If $b_j\in \vec a\setminus a_i$, both sides of \eqref{eq:3} are zero. Suppose
$b_j\not\in \vec a\setminus a_i$. Then $b_j\sqcup (\vec a\setminus a_i)\sqcup
\vec c$ is a rearrangment of $(\rho_1-1,\rho_2-2,\dots,\rho_n-n)$, where
$\rho=(\rho_1,\dots,\rho_n)$ is the partition $\lambda(a_i,b_j)$. Thus
\[
[b_j\sqcup (\vec a\setminus a_i)\sqcup \vec c] = (-1)^t \ts_{\lambda(a_i,b_j)/\nu},
\]
where
\begin{align*}
t &= \inv(\vec a\setminus a_i, b_j) +\inv(\vec c, b_j) + \inv(\vec c, \vec a\setminus a_i)\\  
&=\inv(\vec a, b_j) - \inv(a_i, b_j) +\inv(\vec c, b_j) + \inv(\vec c, \vec a) - \inv(\vec c, a_i).
\end{align*}
Since $\inv(a_i, b_j) = \chi(a_i>b_j)$, we obtain  \eqref{eq:3}. 

By factoring out common factors from each row and each column using \eqref{eq:3}, we get
\begin{equation}
  \label{eq:1}
\det ([b_j\sqcup (\vec a\setminus a_i)\sqcup \vec c])_{i,j=1}^k 
= (-1)^{ (k-1)\inv(\vec c,\vec a) + \inv(\vec c, \vec b)+\inv(\vec a, \vec b) } \det ((-1)^{\chi(a_i>b_j)} \ts_{\lambda(a_i,b_j)/\nu}).
\end{equation}
By substituting \eqref{eq:5}, \eqref{eq:6} and \eqref{eq:1} to \eqref{eq:2}, we obtain 
\begin{equation}
  \label{eq:4}
\ts_{\lambda/\nu}^{k-1} \ts_{\mu/\nu} =(-1)^{\inv(\vec a,\vec b) +\binom k2}\det \left( (-1)^{\chi(a_i>b_j)} \ts_{\lambda(a_i,b_j)/\nu}\right)_{i,j=1}^k.
\end{equation}
Since $a_i\ne b_j$ for all $i,j$, we have $\inv(\vec a,\vec b) = k^2-\inv(\vec b,\vec a)$ and $\chi(a_i>b_j)=1-\chi(b_j>a_i)$. Thus \eqref{eq:4} is equivalent to \eqref{eq:main1}, which completes the proof of the first identity. 

The second identity \eqref{eq:main2} is proved by the same arguments except that
in this case we use the matrix $N=\left(h_{\nu_j-j+i,i+1}\right)_{i\in\ZZ, 1\le
  j\le n}$ in place of $M$.
\end{proof}

\section{Lascoux--Pragacz and Kreiman decompositions}
\label{sec:lasc-prag-kreim}

In this section we restate Theorem~\ref{thm:main} using the Lascoux--Pragacz and
Kreiman decompositions for the case $\mu\subseteq\lambda$. As a corollary we
prove (a corrected version of) a conjecture of Morales, Pak, and Panova
\cite{MPP2}.

Recall that for a border strip $\gamma$, we denote by $p(\gamma)$
(resp.~$q(\gamma)$) the content of the starting (resp.~ending) cell of $\gamma$.

\begin{defn}
  A \emph{(border strip) decomposition} of a skew shape $\alpha$ is a sequence
  $\theta=(\theta_1,\dots,\theta_k)$ of border strips satisfying the following
  conditions.
\begin{itemize}
\item $\theta_i\cap \theta_j =\emptyset$ for all $i\ne j$,
\item $\theta_1\cup \cdots\cup \theta_k = \alpha$.
\end{itemize}
If there is no possible confusion, we will simply write $p_i=p(\theta_i)$,
$q_i=q(\theta_i)$, $\vec p=(p_1,\dots,p_k)$ and $\vec q=(q_1,\dots,q_k)$. By
convention we will always label the border strips so that $q_1\ge q_2\ge \dots \ge q_k$.
\end{defn}

Recall that the outer strip of a partition $\lambda$ is the set of cells
$x\in\lambda$ satisfying $x+(-1,-1)\not\in\lambda$. For a skew shape $\lm$, the
\emph{outer strip} of $\lm$ is defined to be the outer strip of $\lambda$. Let
$\gamma$ be the outer strip of $\lm$ and let $\rho$ be the set of cells $x\in
\lm$ with $x+(-1,-1)\not\in\lm$. We define the \emph{inner strip} of $\lm$ to be
the set
\[
\rho \cup \{x\in \gamma: c(x)\notin \Cont(\rho) \}.
\]
See Figure~\ref{fig:inner}.

\begin{figure}
  \centering
\begin{tikzpicture}[scale=.46]
\fill[fill=cyan!27!white, rounded corners] (.2, .2) rectangle (1.8, .8)
				(1.2, .2) rectangle (1.8, 1.8)
				(1.2, 1.2) rectangle (2.8, 1.8)
				(2.2, 1.2) rectangle (2.8, 2.8)
				(2.2, 2.2) rectangle (3.8, 2.8)
				(3.2, 2.2) rectangle (3.8, 4.8)
				(3.2, 4.2) rectangle (6.8, 4.8)
				(6.2, 4.2) rectangle (6.8, 6.8)
				(6.2, 6.2) rectangle (9.8, 6.8)
				(9.2, 6.2) rectangle (9.8, 7.8)
				(9.2, 7.2) rectangle (10.8, 7.8)
				(10.2, 7.2) rectangle (10.8, 8.8);
\draw[dashed, color=gray] (0,3)--(0,9)--(11,9)--(11,8)
		(4,4)--(6,4)
		(7,5)--(7,6)--(8,6)
		(0,8)--(8,8)
		(0,7)--(8,7)
		(0,6)--(7,6)
		(0,5)--(6,5)
		(0,4)--(1,4)
		(1,4)--(1,9)
		(2,4)--(2,9)
		(3,4)--(3,9)
		(4,4)--(4,9)
		(5,4)--(5,9)
		(6,5)--(6,9)
		(7,6)--(7,9)
		(8,8)--(8,9)
		(9,8)--(9,9)
		(10, 8)--(10,9);
\draw[thick] (0,0)--(2,0)--(2,1)--(3,1)--(3,2)--(4,2)--(4,4)--(1,4)--(1,3)--(0,3)--(0,0)--cycle
		(6,4)--(7,4)--(7,5)--(6,5)--(6,4)--cycle
		(8,6)--(10,6)--(10,7)--(11,7)--(11,8)--(8,8)--(8,6)--cycle
		(0,1)--(2,1)
		(0,2)--(3,2)
		(1,3)--(4,3)
		(1,0)--(1,3)
		(2,1)--(2,4)
		(3,2)--(3,4)
		(9,6)--(9,8)
		(10,7)--(10,8)
		(8,7)--(10,7);
\end{tikzpicture}\qquad\quad
\begin{tikzpicture}[scale=.46]
\fill[fill=asparagus!35!white, rounded corners] (.2, .2) rectangle (.8, 2.8)
			(.2, 2.2) rectangle (1.8, 2.8)
			(1.2, 2.2) rectangle (1.8, 3.8)
			(1.2, 3.2) rectangle (3.8, 3.8)
			(3.2, 3.2) rectangle (3.8, 4.8)
			(3.2, 4.2) rectangle (6.8, 4.8)
			(6.2, 4.2) rectangle (6.8, 6.8)
			(6.2, 6.2) rectangle (8.8, 6.8)
			(8.2, 6.2) rectangle (8.8, 7.8)
			(8.2, 7.2) rectangle (10.8, 7.8)
			(10.2, 7.2) rectangle (10.8, 8.8);
\draw[dashed, color=gray] (0,3)--(0,9)--(11,9)--(11,8)
		(4,4)--(6,4)
		(7,5)--(7,6)--(8,6)
		(0,8)--(8,8)
		(0,7)--(8,7)
		(0,6)--(7,6)
		(0,5)--(6,5)
		(0,4)--(1,4)
		(1,4)--(1,9)
		(2,4)--(2,9)
		(3,4)--(3,9)
		(4,4)--(4,9)
		(5,4)--(5,9)
		(6,5)--(6,9)
		(7,6)--(7,9)
		(8,8)--(8,9)
		(9,8)--(9,9)
		(10, 8)--(10,9);
\draw[thick] (0,0)--(2,0)--(2,1)--(3,1)--(3,2)--(4,2)--(4,4)--(1,4)--(1,3)--(0,3)--(0,0)--cycle
		(6,4)--(7,4)--(7,5)--(6,5)--(6,4)--cycle
		(8,6)--(10,6)--(10,7)--(11,7)--(11,8)--(8,8)--(8,6)--cycle
		(0,1)--(2,1)
		(0,2)--(3,2)
		(1,3)--(4,3)
		(1,0)--(1,3)
		(2,1)--(2,4)
		(3,2)--(3,4)
		(9,6)--(9,8)
		(10,7)--(10,8)
		(8,7)--(10,7);
\end{tikzpicture}
  \caption{The outer strip of $\lm$ on the left and the inner strip of $\lm$ on
    the right.}
  \label{fig:inner}
\end{figure}
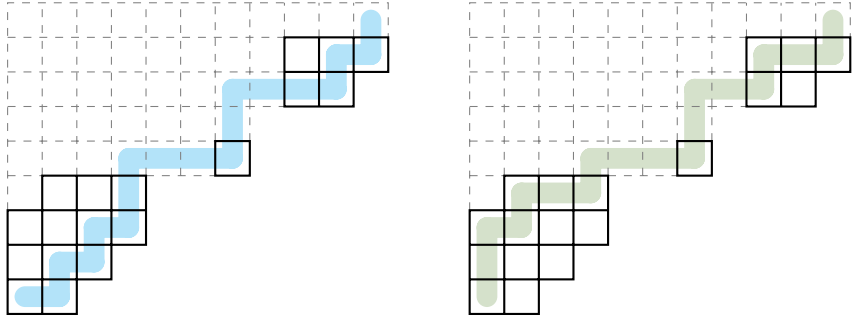

\begin{defn}
  Let $\mu$ and $\lambda$ be partitions with $\mu\subseteq\lambda$. The
  \emph{Lascoux--Pragacz decomposition} (resp.~\emph{Kreiman decomposition}) of
  $\lm$ is the decomposition of $\lm$ obtained by using the outer (resp.~inner)
  strip of $\lambda$ as the cutting strip. See Figure~\ref{fig:LPK}.
\end{defn}

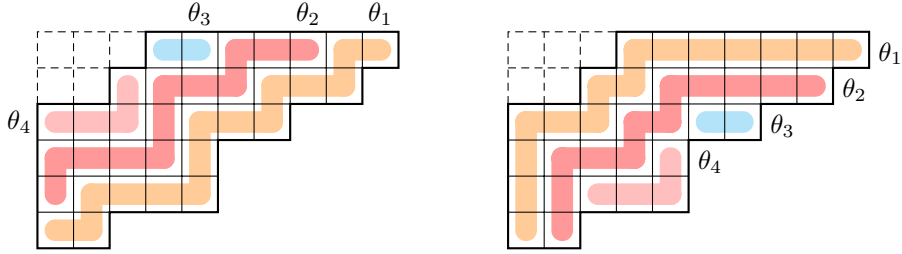
\begin{figure}
  \centering
\begin{tikzpicture}[scale=.48]
\fill[fill=orange!40!white, rounded corners] (.2, .2) rectangle (1.8, .8);
\fill[fill=orange!40!white, rounded corners] (1.2, .2) rectangle (1.8, 1.8);
\fill[fill=orange!40!white, rounded corners] (1.2, 1.2) rectangle (4.8, 1.8);
\fill[fill=orange!40!white, rounded corners] (4.2, 1.2) rectangle (4.8, 3.8);
\fill[fill=orange!40!white, rounded corners] (4.2, 3.2) rectangle (6.8, 3.8);
\fill[fill=orange!40!white, rounded corners] (6.2, 3.2) rectangle (6.8, 4.8);
\fill[fill=orange!40!white, rounded corners] (6.2, 4.2) rectangle (8.8, 4.8);
\fill[fill=orange!40!white, rounded corners] (8.2, 4.2) rectangle (8.8, 5.8);
\fill[fill=orange!40!white, rounded corners] (8.2, 5.2) rectangle (9.8, 5.8);
\fill[fill=red!40!white, rounded corners] (.2, 1.2) rectangle (.8, 2.8);
\fill[fill=red!40!white, rounded corners] (.2, 2.2) rectangle (3.8, 2.8);
\fill[fill=red!40!white, rounded corners] (3.2,2.2) rectangle (3.8, 4.8);
\fill[fill=red!40!white, rounded corners] (3.2, 4.2) rectangle (5.8, 4.8);
\fill[fill=red!40!white, rounded corners] (5.2, 4.2) rectangle (5.8, 5.8);
\fill[fill=red!40!white, rounded corners] (5.2, 5.2) rectangle (7.8, 5.8);
\fill[fill=cyan!27!white, rounded corners] (3.2, 5.2) rectangle (4.8, 5.8);
\fill[fill=pink, rounded corners] (.2, 3.2) rectangle (2.8, 3.8);
\fill[fill=pink, rounded corners] (2.2, 3.2) rectangle (2.8, 4.8);
\draw[thick] (0,0)--(2,0)--(2,1)--(5,1)--(5,3)--(7,3)--(7,4)--(9,4)--(9,5)--(10,5)--(10,6)--(3,6)--(3,5)--(2,5)--(2,4)--(0,4)--(0,0)--cycle;
\draw[densely dashed] (0,4)--(0,6)--(3,6);
\draw[densely dashed] (0,5)--(2,5);
\draw[densely dashed] (1,4)--(1,6);
\draw[densely dashed] (2,5)--(2,6);
\draw (3,5)--(9,5)
	(2,4)--(7,4)
	(0,3)--(5,3)
	(0,2)--(5,2)
	(0,1)--(2,1)
	(1,0)--(1,4)
	(2,1)--(2,4)
	(3,1)--(3,5)
	(4,1)--(4,6)
	(5,3)--(5,6)
	(6,3)--(6,6)
	(7,4)--(7,6)
	(8,4)--(8,6)
	(9,5)--(9,6);
\node[] at (9.5, 6.5) {$\theta_1$};
\node[] at (7.5, 6.5) {$\theta_2$};
\node[] at (4.5, 6.5) {$\theta_3$};
\node[] at (-.5, 3.5) {$\theta_4$};
\node[] at (2, -1) {\phantom{d}};
\end{tikzpicture}\qquad\qquad
\begin{tikzpicture}[scale=.48]
\fill[fill=orange!40!white, rounded corners] (.2, .2) rectangle (.8, 3.8)
							(.2, 3.2) rectangle (2.8, 3.8)
							(2.2, 3.2) rectangle (2.8, 4.8)
							(2.2, 4.2) rectangle (3.8, 4.8)
							(3.2, 4.2) rectangle (3.8, 5.8)
							(3.2, 5.2) rectangle (9.8, 5.8);
\fill[fill=red!40!white, rounded corners] (1.2, .2) rectangle (1.8, 2.8)
							(1.2, 2.2) rectangle (3.8, 2.8)
							(3.2, 2.2) rectangle (3.8, 3.8)
							(3.2, 3.2) rectangle (4.8, 3.8)
							(4.2, 3.2) rectangle (4.8, 4.8)
							(4.2, 4.2) rectangle (8.8, 4.8);
\fill[fill=pink, rounded corners] (2.2, 1.2) rectangle (4.8, 1.8)
						(4.2, 1.2) rectangle (4.8, 2.8);
\fill[fill=cyan!27!white, rounded corners] (5.2, 3.2) rectangle (6.8, 3.8);
\draw[thick] (0,0)--(2,0)--(2,1)--(5,1)--(5,3)--(7,3)--(7,4)--(9,4)--(9,5)--(10,5)--(10,6)--(3,6)--(3,5)--(2,5)--(2,4)--(0,4)--(0,0)--cycle;
\draw[densely dashed] (0,4)--(0,6)--(3,6);
\draw[densely dashed] (0,5)--(2,5);
\draw[densely dashed] (1,4)--(1,6);
\draw[densely dashed] (2,5)--(2,6);
\draw (3,5)--(9,5)
	(2,4)--(7,4)
	(0,3)--(5,3)
	(0,2)--(5,2)
	(0,1)--(2,1)
	(1,0)--(1,4)
	(2,1)--(2,4)
	(3,1)--(3,5)
	(4,1)--(4,6)
	(5,3)--(5,6)
	(6,3)--(6,6)
	(7,4)--(7,6)
	(8,4)--(8,6)
	(9,5)--(9,6);
\node[] at (10.6, 5.4) {$\theta_1$};
\node[] at (9.6, 4.4) {$\theta_2$};
\node[] at (7.6, 3.4) {$\theta_3$};
\node[] at (5.6, 2.4) {$\theta_4$};
\node[] at (2, -1) {\phantom{d}};
\end{tikzpicture}
\caption{The Lascoux--Pragacz decompositions of a connected skew shape
  (left) and the Kreiman decomposition of the same skew shape (right)}
  \label{fig:LPK}
\end{figure}

\begin{lem}\label{lem:C/C}
  Let $\lambda,\mu\in \Par_n$ with $\mu\subseteq\lambda$. Suppose that
  $\theta=(\theta_1,\dots,\theta_k)$ is the Lascoux--Pragacz or Kreiman
  decomposition of $\lm$. Then $p_i\ne p_j$, $q_i\ne q_j$, $p_i\ne q_j$, and
  $p_i-1\ne q_j$ for all $1\le i\ne j\le k$, and
  \begin{align*}
C_n(\lambda)\setminus C_n(\mu) &= \{q_1>q_2>\dots>q_k\},\\
C_n(\mu)\setminus C_n(\lambda) &= \{p_1-1,p_2-1,\dots,p_k-1\}.
  \end{align*}
\end{lem}
\begin{proof}
  We will only consider the case that $\theta$ is the Lascoux--Pragacz
  decomposition of $\lm$ since it can be proved similarly for the case of the
  Kreiman decomposition. We proceed by induction on $k$. If $k=0$,
  then $\lambda=\mu$ and there is nothing to prove. Suppose that $k\ge1$ and the
  lemma is true for $k-1$.

  By the definition of the Lascoux--Pragacz decomposition, the border strip
  $\theta_1$ is a connected component of $(\lm)\cap \lambda^0$. This implies
  that $p_1-1,q_1+1\not\in \Cont(\lm)$ and the starting and ending cells of
  $\theta_1$ are the only cells in $\lm$ whose contents are $p_1$ and $q_1$,
  respectively. Therefore, if $\Cont(\theta_i)\cap\Cont(\theta_1)\ne\emptyset$
  for some $i\ge2$, then $\Cont(\theta_i)\subseteq[p_1+1,q_1-1]$ because
  $\Cont(\theta_i)$ is a set of consecutive integers. This shows that $p_i\ne
  p_1$ $q_i\ne q_1$, $p_i\ne q_1$, $p_i-1\ne q_1$, and $p_1-1\ne q_i$ for all
  $2\le i\le k$.

  Let $\rho=\lambda\setminus\theta_1$. Then $\mu\subseteq\rho$ and
  $(\theta_2,\theta_3,\dots,\theta_k)$ is the Lascoux--Pragacz decomposition of
  $\rho/\mu$. By the induction hypothesis, we have $p_i\ne p_j$, $q_i\ne q_j$,
  $p_i\ne q_j$, and $p_i-1\ne q_j$ for all $2\le i\ne j\le k$, and
  \begin{align*}
    C_n(\rho)\setminus C_n(\mu) &= \{q_2>q_3>\dots>q_k\},\\
    C_n(\mu)\setminus C_n(\rho) &= \{p_2-1,p_3-1,\dots,p_k-1\}.
  \end{align*}

  It is straightforward to check that $p_1-1\in C_n(\mu)$, $q_1\not\in
  C_n(\mu)$, $p_1-1\not\in C_n(\lambda)$, $q_1\in C_n(\lambda)$, and
\[
C_n(\rho) = (C_n(\lambda)\setminus\{q_1\})\cup\{p_1-1\}.
\]
Combining the above results we obtain the lemma for $k$. The induction then
completes the proof.
\end{proof}

By Lemma~\ref{lem:C/C}, if $\mu\subseteq\lambda$, then we can restate
Theorem~\ref{thm:main} using the Lascoux--Pragacz or Kreiman decompositions as
follows.

\begin{thm}[Theorem~\ref{thm:main_LP_intro}]\label{thm:main_LP}
  Let $\lambda$, $\mu$, and $\nu$ be partitions and suppose that $\mu\subseteq
  \lambda$ and $\theta=(\theta_1,\dots,\theta_k)$ is the Lascoux--Pragacz
  decomposition of $\lm$. Then we have
  \begin{align}
  \label{eq:main_LP1}
    \ts_{\lambda/\nu}^{k-1} \ts_{\mu/\nu}
    &= \det \left( (-1)^{\chi(p_j>q_i)} \ts_{\lambda(q_i,p_j-1)/\nu}\right)_{i,j=1}^k,  \\
  \label{eq:main_LP2}
    \ts_{\nu/\lambda}^{k-1} \ts_{\nu/\mu}
    &= \det \left( (-1)^{\chi(p_j>q_i)} \ts_{\nu/\lambda(q_i,p_j-1)}\right)_{i,j=1}^k.
  \end{align}
\end{thm}
\begin{proof}
  Let $\vec a=(a_1,\dots,a_k)$ and $\vec b=(b_1,\dots,b_k)$ be the sequences
  defined by
\begin{align*}
  C_n(\lambda)\setminus C_n(\mu) &= \{a_1>a_2>\dots>a_k\},\\
  C_n(\mu)\setminus C_n(\lambda) &= \{b_1>b_2>\dots>b_k\}.
\end{align*}
By \eqref{eq:main1}, we have
\[
  \ts_{\lambda/\nu}^{k-1} \ts_{\mu/\nu}
  =(-1)^{\inv(\vec b,\vec a)+\binom k2}\det
    \left( (-1)^{\chi(b_j>a_i)} \ts_{\lambda(a_i,b_j)/\nu}\right)_{i,j=1}^k.
\]
By Lemma~\ref{lem:C/C}, we have $\vec a = \vec q$ and
$\{p_1-1,p_2-1,\dots,p_k-1\}=\{b_1,b_2,\dots,b_k\}$. Let $\vec b' =
(b_1',b_2',\dots,b_k') = (b_k,b_{k-1},\dots,b_1)$ be the increasing
rearrangement of $\vec b$. Then we can rewrite the above equation as
\begin{align}
  \notag
  \ts_{\lambda/\nu}^{k-1} \ts_{\mu/\nu}
  &=(-1)^{\inv(\vec b,\vec q)}\det
    \left( (-1)^{\chi(b'_j>q_i)} \ts_{\lambda(q_i,b'_j)/\nu}\right)_{i,j=1}^k\\
\label{eq:bq}
  &=(-1)^{\inv(\vec b,\vec q)+\inv(\pi)} \det
    \left( (-1)^{\chi(p_j-1>q_i)} \ts_{\lambda(q_i,p_j-1)/\nu}\right)_{i,j=1}^k,
\end{align}
where $\pi$ is the permutation of $\{1,2,\dots,k\}$ satisfying
\[
\vec p - \vec 1:=(p_1-1,p_2-1,\dots,p_k-1) = (b'_{\pi(1)}, b'_{\pi(2)},\dots,b'_{\pi(k)}).
\]

Note that $\inv(\pi)$ is the number of pairs $(i,j)$ with $i<j$ and
$\pi(i)>\pi(j)$, where $\pi(i)>\pi(j)$ is equivalent to
$b'_{\pi(i)}>b'_{\pi(j)}$, which in turn is equivalent to $p_i>p_j$. Thus
$\inv(\pi)$ is equal to the number of pairs $(i,j)$ with $i<j$ and $p_i>p_j$. By
the construction of the Lascoux--Pragacz decomposition, if $i<j$, then we must
have either $p_j<q_j<p_i<q_i$ or $p_i<p_j\le q_j<q_i$. This shows that
$\inv(\pi)$ is equal to the number of pairs $(\theta_i,\theta_j)$ of border
strips such that $p(\theta_i)>q(\theta_j)$, which automatically implies $i<j$.
Therefore $\inv(\pi)=\inv(\vec p,\vec q)$. On the other hand, by
Lemma~\ref{lem:C/C}, we have $p_i-1\ne q_j$ for all $1\le i,j\le n$. Thus
$\inv(\pi) = \inv(\vec p,\vec q)=\inv(\vec p-\vec 1,\vec q)=\inv(\vec b,\vec q)$
and $\chi(p_j-1>q_i)=\chi(p_j>q_i)$. Therefore the right hand sides of
\eqref{eq:bq} and \eqref{eq:main_LP1} are equal, which completes the proof of
\eqref{eq:main_LP1}.

The second identity \eqref{eq:main_LP2} can be proved similarly.
\end{proof}

\begin{thm}\label{thm:main_K}
  Let $\lambda,\mu$, and $\nu$ be partitions. Suppose that $\mu\subseteq \lambda$
  and $\theta=(\theta_1,\dots,\theta_k)$ is the Kreiman decomposition of $\lm$.
  Then we have 
  \begin{align}
  \label{eq:main_K1}
    \ts_{\mu/\nu}^{k-1} \ts_{\lambda/\nu}
    &= \det \left( (-1)^{\chi(p_i> q_j)} \ts_{\mu(p_i-1,q_j)/\nu}\right)_{i,j=1}^k,  \\
  \label{eq:main_K2}
    \ts_{\nu/\lambda}^{k-1} \ts_{\nu/\mu}
    &= \det \left( (-1)^{\chi(p_i> q_j)} \ts_{\nu/\mu(p_i-1,q_j)}\right)_{i,j=1}^k.
  \end{align}
\end{thm}

\begin{proof}
  Let $\vec a=(a_1,\dots,a_k)$ and $\vec b=(b_1,\dots,b_k)$ be the sequences
  defined by
\begin{align*}
  C_n(\mu)\setminus C_n(\lambda) &= \{a_1>a_2>\dots>a_k\},\\
  C_n(\lambda)\setminus C_n(\mu) &= \{b_1>b_2>\dots>b_k\}.
\end{align*}
Then by \eqref{eq:main1} with the roles of $\lambda$ and $\mu$ switched, we have
\[
  \ts_{\mu/\nu}^{k-1} \ts_{\lambda/\nu}
  =(-1)^{\inv(\vec b,\vec a)+\binom k2}\det
    \left( (-1)^{\chi(b_j>a_i)} \ts_{\mu(a_i,b_j)/\nu}\right)_{i,j=1}^k.
\]
By Lemma~\ref{lem:C/C}, we have $\vec b = \vec q$ and 
$\{p_1-1,p_2-1,\dots,p_k-1\}=\{a_1,a_2,\dots,a_k\}$.
Let $\vec a' = (a_1',a_2',\dots,a_k') = (a_k,a_{k-1},\dots,a_1)$ be the
increasing rearrangement of $\vec a$. Then we can rewrite the above equation as
\begin{align}
  \notag
  \ts_{\mu/\nu}^{k-1} \ts_{\lambda/\nu}
  &=(-1)^{\inv(\vec q,\vec a)}\det
    \left( (-1)^{\chi(q_j>a'_i)} \ts_{\mu(a'_i,q_j)/\nu}\right)_{i,j=1}^k\\
\label{eq:bq2}
  &=(-1)^{\inv(\vec q,\vec a)+\inv(\pi)}\det
    \left( (-1)^{\chi(q_j>p_i-1)} \ts_{\mu(p_i-1,q_j)/\nu}\right)_{i,j=1}^k,
\end{align}
where $\pi$ is the permutation of $\{1,2,\dots,k\}$ satisfying
\[
\vec p-\vec 1:= (p_1-1,p_2-1,\dots,p_k-1) = (a'_{\pi(1)}, a'_{\pi(2)},\dots,a'_{\pi(k)}).
\]

By the same argument as in the proof of Theorem~\ref{thm:main_LP}, we have
\[
  \inv(\pi) = \inv(\vec p,\vec q)=\inv(\vec p-\vec 1,\vec q)
  =\inv(\vec a,\vec q) = k^2- \inv(\vec q,\vec a).
\]
Thus we can rewrite \eqref{eq:bq2} as
\[
  \ts_{\mu/\nu}^{k-1} \ts_{\lambda/\nu}
  =\det \left( (-1)^{1-\chi(q_j>p_i-1)} \ts_{\mu(p_i-1,q_j)/\nu}\right)_{i,j=1}^k.
\]
Since $q_j\ne p_i-1$ for all $1\le i,j\le k$ by Lemma~\ref{lem:C/C}, we have
$1-\chi(q_j>p_i-1)=\chi(q_j\le p_i-1)=\chi(q_j<p_i)$, which together with the
above equation shows \eqref{eq:main_K1}.

The second identity \eqref{eq:main_K2} can be proved similarly.
\end{proof}

\begin{exam}\label{exam:LSthm} Let $\lambda=(6,6,6,3,3)$, $\mu=(4,3,2)$ and
  $\nu=(7,6,6,5,3,2)$. The
  Lascoux--Pragacz decomposition of $\lambda/\mu$ and the skew shape $\nl$ are
  shown in Figure~\ref{fig:LP}. Then \eqref{eq:main_LP2} of Theorem
  \ref{thm:main_LP} implies that $\ts_{\nu/\lambda}^2\ts_{\nu/\mu}$ is equal to
  the determinant shown in Figure~\ref{fig:det}.

  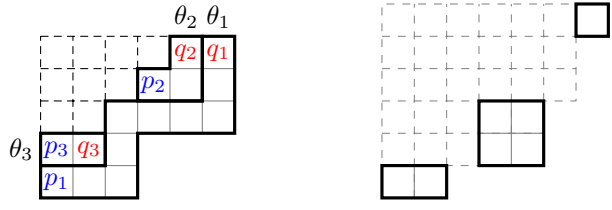
\begin{figure}
    \centering
\begin{tikzpicture}[scale=.43]
\draw[densely dashed] (0,5)--(4,5)
				   (0,4)--(3,4)
				   (0,3)--(2,3)
				   (0,2)--(0,5)
				   (1,2)--(1,5)
				   (2,3)--(2,5)
				   (3,4)--(3,5);
\draw[color=gray] (4,5)--(6,5)
			    (3,4)--(6,4)
			    (2,3)--(6,3)
			    (0,2)--(6,2)
			    (0,1)--(3,1)
			    (0,0)--(3,0)
			    (0,0)--(0,2)
			    (1,0)--(1,2)
			    (2,0)--(2,3)
			    (3,0)--(3,4)
			    (4,2)--(4,5)
			    (5,2)--(5,5)
			    (6,2)--(6,5);
\draw[thick, line width=1.2pt] (0,1)--(2,1)--(2,2)--(0,2)--(0,1)--cycle;
\draw[thick, line width=1.2pt] (0,0)--(3,0)--(3,2)--(6,2)--(6,5)--(5,5)--(5,3)--(2,3)--(2,1)--(0,1)--(0,0)--cycle;
\draw[thick, line width=1.2pt] (3,3)--(5,3)--(5,5)--(4,5)--(4,4)--(3,4)--(3,3)--cycle;
\node[] at (.5, .5) {$\color{blue} p_1$};
\node[] at (5.5, 4.5) {$\color{red} q_1$};
\node[] at (3.5, 3.5) {$\color{blue} p_2$};
\node[] at (4.5, 4.5) {$\color{red} q_2$};
\node[] at (.5, 1.5) {$\color{blue} p_3$};
\node[] at (1.5, 1.5) {$\color{red} q_3$};
\node[] at (5.5, 5.6) {$\theta_1$};
\node[] at (-.6, 1.5) {$\theta_3$};
\node[] at (4.5, 5.6) {$\theta_2$};
\end{tikzpicture} \qquad\qquad\quad
\begin{tikzpicture}[scale=.43]
\draw[dashed, line width = .2pt, color=gray] (0,0)--(2,0)--(2,1)--(3,1)--(3,2)--(5,2)--(5,3)--(6,3)--(6,5)--(7,5)--(7,6)--(0,6)--(0,0)--cycle
		(0,5)--(6,5)
		(0,4)--(6,4)
		(0,3)--(5,3)
		(0,2)--(3,2)
		(0,1)--(2,1)
		(1,0)--(1,6)
		(2,1)--(2,6)
		(3,2)--(3,6)
		(4,2)--(4,6)
		(5,3)--(5,6)
		(6,5)--(6,6);
\draw[color=gray] (1,0)--(1,1)
		(4,1)--(4,3)
		(3,2)--(5,2);
\draw[line width=1.2pt] (0,0)--(2,0)--(2,1)--(0,1)--(0,0)--cycle
			(3,1)--(5,1)--(5,3)--(3,3)--(3,1)--cycle
			(6,5)--(7,5)--(7,6)--(6,6)--(6,5)--cycle;
\end{tikzpicture}
    \caption{The Lascoux--Pragacz decomposition of $\lambda/\mu$ on the left
      and the skew shape $\nl$ on the right. }
    \label{fig:LP}
  \end{figure}

  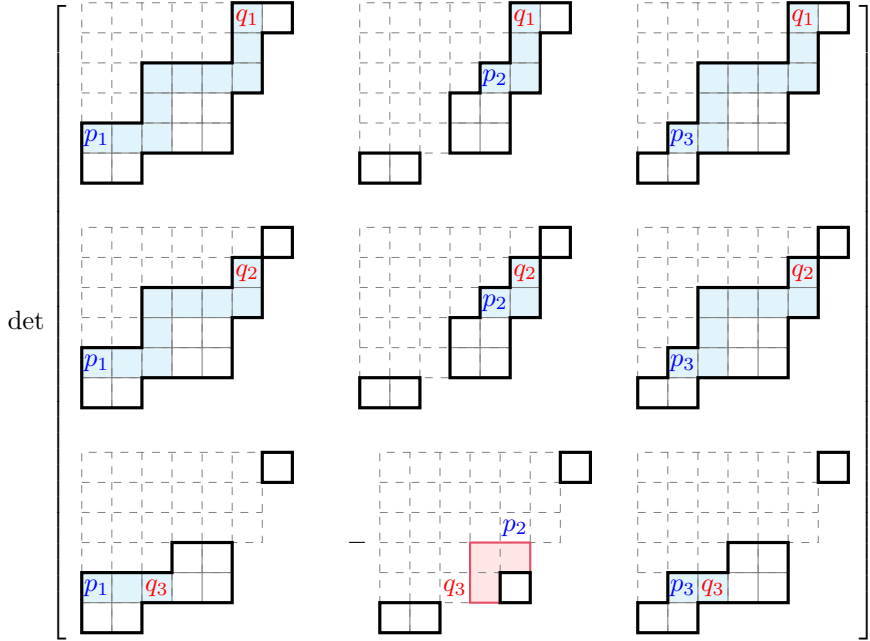
\begin{figure}
    \centering
\[
\det \begin{bmatrix}
~\begin{tikzpicture}[scale=.4]
\fill[fill=cyan!10!white] (0,1)--(3,1)--(3,3)--(6,3)--(6,6)--(5,6)--(5,4)--(2,4)--(2,2)--(0,2)--(0,1)--cycle;
\draw[dashed, line width = .2pt, color=gray] (0,0)--(2,0)--(2,1)--(3,1)--(3,2)--(5,2)--(5,3)--(6,3)--(6,5)--(7,5)--(7,6)--(0,6)--(0,0)--cycle
		(0,5)--(6,5)
		(0,4)--(6,4)
		(0,3)--(5,3)
		(0,2)--(3,2)
		(0,1)--(2,1)
		(1,0)--(1,6)
		(2,1)--(2,6)
		(3,2)--(3,6)
		(4,2)--(4,6)
		(5,3)--(5,6)
		(6,5)--(6,6);
\draw[color=gray] (5,5)--(6,5)
		(5,4)--(6,4)
		(2,3)--(5,3)
		(2,2)--(5,2)
		(0,1)--(2,1)
		(1,0)--(1,2)
		(2,1)--(2,2)
		(3,1)--(3,4)
		(4,1)--(4,4)
		(5,3)--(5,4)
		(6,5)--(6,6);
\draw[line width=1.2pt] (0,0)--(2,0)--(2,1)--(3,1)--(5,1)--(5,3)--(6,3)--(6,5)--(7,5)--(7,6)--(5,6)--(5,4)--(2,4)--(2,2)--(0,2)--(0,0)--cycle;
\node[] at (.5, 1.5) {$\color{blue} p_1$};
\node[] at (5.5, 5.5) {$\color{red} q_1$};
\end{tikzpicture}~&~\begin{tikzpicture}[scale=.4]
\fill[fill=cyan!10!white] (4,3)--(6,3)--(6,6)--(5,6)--(5,4)--(4,4)--(4,3)--cycle;
\draw[dashed, line width = .2pt, color=gray] (0,0)--(2,0)--(2,1)--(3,1)--(3,2)--(5,2)--(5,3)--(6,3)--(6,5)--(7,5)--(7,6)--(0,6)--(0,0)--cycle
		(0,5)--(6,5)
		(0,4)--(6,4)
		(0,3)--(5,3)
		(0,2)--(3,2)
		(0,1)--(2,1)
		(1,0)--(1,6)
		(2,1)--(2,6)
		(3,2)--(3,6)
		(4,2)--(4,6)
		(5,3)--(5,6)
		(6,5)--(6,6);
\draw[color=gray] (1,0)--(1,1)
		(3,2)--(5,2)
		(4,1)--(4,3)
		(5,3)--(5,4)
		(6,5)--(6,6)
		(5,5)--(6,5)
		(5,4)--(6,4)
		(4,3)--(5,3);
\draw[line width=1.2pt] (0,0)--(2,0)--(2,1)--(0,1)--(0,0)--cycle
		(3,1)--(5,1)--(5,3)--(6,3)--(6,5)--(7,5)--(7,6)--(5,6)--(5,4)--(4,4)--(4,3)--(3,3)--(3,1)--cycle;
\node[] at (4.5, 3.5) {$\color{blue} p_2$};
\node[] at (5.5, 5.5) {$\color{red} q_1$};
\end{tikzpicture}~&~\begin{tikzpicture}[scale=.4]
\fill[fill=cyan!10!white] (1,1)--(3,1)--(3,3)--(6,3)--(6,6)--(5,6)--(5,4)--(2,4)--(2,2)--(1,2)--(1,1)--cycle;
\draw[dashed, line width = .2pt, color=gray] (0,0)--(2,0)--(2,1)--(3,1)--(3,2)--(5,2)--(5,3)--(6,3)--(6,5)--(7,5)--(7,6)--(0,6)--(0,0)--cycle
		(0,5)--(6,5)
		(0,4)--(6,4)
		(0,3)--(5,3)
		(0,2)--(3,2)
		(0,1)--(2,1)
		(1,0)--(1,6)
		(2,1)--(2,6)
		(3,2)--(3,6)
		(4,2)--(4,6)
		(5,3)--(5,6)
		(6,5)--(6,6);
\draw[color=gray] (5,5)--(6,5)
		(3,2)--(5,2)
		(5,4)--(6,4)
		(2,3)--(5,3)
		(2,2)--(3,2)
		(0,1)--(2,1)
		(1,0)--(1,2)
		(2,1)--(2,2)
		(3,1)--(3,4)
		(4,1)--(4,4)
		(5,3)--(5,4)
		(6,5)--(6,6);
\draw[line width=1.2pt] (0,0)--(2,0)--(2,1)--(3,1)--(5,1)--(5,3)--(6,3)--(6,5)--(7,5)--(7,6)--(5,6)--(5,4)--(2,4)--(2,2)--(1,2)--(1,1)--(0,1)--(0,0)--cycle;
\node[] at (1.5, 1.5) {$\color{blue} p_3$};
\node[] at (5.5, 5.5) {$\color{red} q_1$};
\end{tikzpicture}~\\\\~\begin{tikzpicture}[scale=.4]
\fill[fill=cyan!10!white] (0,1)--(3,1)--(3,3)--(6,3)--(6,5)--(5,5)--(5,4)--(2,4)--(2,2)--(0,2)--(0,1)--cycle;
\draw[dashed, line width = .2pt, color=gray] (0,0)--(2,0)--(2,1)--(3,1)--(3,2)--(5,2)--(5,3)--(6,3)--(6,5)--(7,5)--(7,6)--(0,6)--(0,0)--cycle
		(0,5)--(6,5)
		(0,4)--(6,4)
		(0,3)--(5,3)
		(0,2)--(3,2)
		(0,1)--(2,1)
		(1,0)--(1,6)
		(2,1)--(2,6)
		(3,2)--(3,6)
		(4,2)--(4,6)
		(5,3)--(5,6)
		(6,5)--(6,6);
\draw[color=gray] (5,5)--(6,5)
		(5,4)--(6,4)
		(2,3)--(5,3)
		(2,2)--(3,2)
		(0,1)--(2,1)
		(1,0)--(1,2)
		(2,1)--(2,2)
		(3,2)--(5,2)
		(3,1)--(3,4)
		(4,1)--(4,4)
		(5,3)--(5,4)
		(6,5)--(6,6);
\draw[line width=1.2pt] (0,0)--(2,0)--(2,1)--(3,1)--(5,1)--(5,3)--(6,3)--(6,5)--(7,5)--(7,6)--(6,6)--(6,5)--(5,5)--(5,4)--(2,4)--(2,2)--(0,2)--(0,0)--cycle;
\node[] at (.5, 1.5) {$\color{blue} p_1$};
\node[] at (5.5, 4.5) {$\color{red} q_2$};
\end{tikzpicture}~&~\begin{tikzpicture}[scale=.4]
\fill[fill=cyan!10!white] (4,3)--(6,3)--(6,5)--(5,5)--(5,4)--(4,4)--(4,3)--cycle;
\draw[dashed, line width = .2pt, color=gray] (0,0)--(2,0)--(2,1)--(3,1)--(3,2)--(5,2)--(5,3)--(6,3)--(6,5)--(7,5)--(7,6)--(0,6)--(0,0)--cycle
		(0,5)--(6,5)
		(0,4)--(6,4)
		(0,3)--(5,3)
		(0,2)--(3,2)
		(0,1)--(2,1)
		(1,0)--(1,6)
		(2,1)--(2,6)
		(3,1)--(3,6)
		(4,1)--(4,6)
		(5,3)--(5,6)
		(6,5)--(6,6);
\draw[color=gray] (1,0)--(1,1)
		(3,2)--(5,2)
		(4,1)--(4,3)
		(5,3)--(5,4)
		(6,5)--(6,6)
		(5,5)--(6,5)
		(5,4)--(6,4)
		(4,3)--(5,3);
\draw[line width=1.2pt] (0,0)--(2,0)--(2,1)--(0,1)--(0,0)--cycle
		(3,1)--(5,1)--(5,3)--(6,3)--(6,5)--(7,5)--(7,6)--(6,6)--(6,5)--(5,5)--(5,4)--(4,4)--(4,3)--(3,3)--(3,1)--cycle;
\node[] at (4.5, 3.5) {$\color{blue} p_2$};
\node[] at (5.5, 4.5) {$\color{red} q_2$};
\end{tikzpicture}~&~\begin{tikzpicture}[scale=.4]
\fill[fill=cyan!10!white] (1,1)--(3,1)--(3,3)--(6,3)--(6,5)--(5,5)--(5,4)--(2,4)--(2,2)--(1,2)--(1,1)--cycle;
\draw[dashed, line width = .2pt, color=gray] (0,0)--(2,0)--(2,1)--(3,1)--(3,2)--(5,2)--(5,3)--(6,3)--(6,5)--(7,5)--(7,6)--(0,6)--(0,0)--cycle
		(0,5)--(6,5)
		(0,4)--(6,4)
		(0,3)--(5,3)
		(0,2)--(3,2)
		(0,1)--(2,1)
		(1,0)--(1,6)
		(2,1)--(2,6)
		(3,2)--(3,6)
		(4,2)--(4,6)
		(5,3)--(5,6)
		(6,5)--(6,6);
\draw[color=gray] (5,5)--(6,5)
		(5,4)--(6,4)
		(2,3)--(5,3)
		(2,2)--(3,2)
		(0,1)--(2,1)
		(1,0)--(1,2)
		(2,1)--(2,2)
		(3,1)--(3,4)
		(3,2)--(5,2)
		(4,1)--(4,4)
		(5,3)--(5,4)
		(6,5)--(6,6);
\draw[line width=1.2pt] (0,0)--(2,0)--(2,1)--(3,1)--(5,1)--(5,3)--(6,3)--(6,5)--(7,5)--(7,6)--(6,6)--(6,5)--(5,5)--(5,4)--(2,4)--(2,2)--(1,2)--(1,1)--(0,1)--(0,0)--cycle;
\node[] at (1.5, 1.5) {$\color{blue} p_3$};
\node[] at (5.5, 4.5) {$\color{red} q_2$};
\end{tikzpicture}~\\\\~
\begin{tikzpicture}[scale=.4]
\fill[fill=cyan!10!white] (0,1)--(3,1)--(3,2)--(0,2)--(0,1)--cycle;
\draw[dashed, line width = .2pt, color=gray] (0,0)--(2,0)--(2,1)--(3,1)--(3,2)--(5,2)--(5,3)--(6,3)--(6,5)--(7,5)--(7,6)--(0,6)--(0,0)--cycle
		(0,5)--(6,5)
		(0,4)--(6,4)
		(0,3)--(5,3)
		(0,2)--(3,2)
		(0,1)--(2,1)
		(1,0)--(1,6)
		(2,1)--(2,6)
		(3,2)--(3,6)
		(4,2)--(4,6)
		(5,3)--(5,6)
		(6,5)--(6,6);
\draw[color=gray] (0,1)--(2,1)
		(1,0)--(1,2)
		(2,1)--(2,2)
		(3,1)--(3,2)
		(3,2)--(5,2)
		(4,1)--(4,3);
\draw[line width=1.2pt] (0,0)--(2,0)--(2,1)--(5,1)--(5,3)--(3,3)--(3,2)--(0,2)--(0,0)--cycle
		(6,5)--(7,5)--(7,6)--(6,6)--(6,5)--cycle;
\node[] at (.5, 1.5) {$\color{blue} p_1$};
\node[] at (2.5, 1.5) {$\color{red} q_3$};
\end{tikzpicture}~&~
\begin{tikzpicture}[scale=.4]
\fill[fill=red!10!white] (3,1)--(4,1)--(4,2)--(5,2)--(5,3)--(3,3)--(3,1)--cycle;
\draw[dashed, line width = .2pt, color=gray] (0,0)--(2,0)--(2,1)--(3,1)--(3,2)--(5,2)--(5,3)--(6,3)--(6,5)--(7,5)--(7,6)--(0,6)--(0,0)--cycle
		(0,5)--(6,5)
		(0,4)--(6,4)
		(0,3)--(5,3)
		(0,2)--(3,2)
		(0,1)--(2,1)
		(1,0)--(1,6)
		(2,1)--(2,6)
		(3,2)--(3,6)
		(4,2)--(4,6)
		(5,3)--(5,6)
		(6,5)--(6,6);
\draw[color=gray] (1,0)--(1,1);
\draw[thick, color=dredcolor] (3,1)--(4,1)--(4,2)--(5,2)--(5,3)--(3,3)--(3,1)--cycle;
\draw[line width=1.2pt] (0,0)--(2,0)--(2,1)--(0,1)--(0,0)--cycle
		(4,1)--(5,1)--(5,2)--(4,2)--(4,1)--cycle
		(6,5)--(7,5)--(7,6)--(6,6)--(6,5)--cycle;
\node[] at (4.5, 3.5) {$\color{blue} p_2$};
\node[] at (2.5, 1.5) {$\color{red} q_3$};
\node[] at (-.75, 3) {$ -$};
\end{tikzpicture}~&~
\begin{tikzpicture}[scale=.4]
\fill[fill=cyan!10!white] (1,1)--(3,1)--(3,2)--(1,2)--(1,1)--cycle;
\draw[dashed, line width = .2pt, color=gray] (0,0)--(2,0)--(2,1)--(3,1)--(3,2)--(5,2)--(5,3)--(6,3)--(6,5)--(7,5)--(7,6)--(0,6)--(0,0)--cycle
		(0,5)--(6,5)
		(0,4)--(6,4)
		(0,3)--(5,3)
		(0,2)--(3,2)
		(0,1)--(2,1)
		(1,0)--(1,6)
		(2,1)--(2,6)
		(3,2)--(3,6)
		(4,2)--(4,6)
		(5,3)--(5,6)
		(6,5)--(6,6);
\draw[color=gray] (0,1)--(2,1)
		(1,0)--(1,2)
		(2,1)--(2,2)
		(3,1)--(3,2)
		(3,2)--(5,2)
		(4,1)--(4,3);
\draw[line width=1.2pt] (0,0)--(2,0)--(2,1)--(5,1)--(5,3)--(3,3)--(3,2)--(1,2)--(1,1)--(0,1)--(0,0)--cycle
		(6,5)--(7,5)--(7,6)--(6,6)--(6,5)--cycle;
\node[] at (1.5, 1.5) {$\color{blue} p_3$};
\node[] at (2.5, 1.5) {$\color{red} q_3$};
\end{tikzpicture}~
\end{bmatrix}
\]   
\caption{ An illustrate of the determinant $\det \left(
    (-1)^{\chi(p_j>q_i)} \ts_{\nu/\lambda(q_i,p_j-1)}\right)_{i,j=1}^k$. The
  $(i,j)$-entry shows the sign $(-1)^{\chi(p_j>q_i)}$ and the (not necessarily
  connected) skew shape $\nu/\lambda(q_i,p_j-1)$ in thick lines. The partition
  $\lambda(q_i,p_j-1)$ is obtained from $\lambda$ by removing (blue) or adding
  (red) a border strip depending on the sign.}
    \label{fig:det}
  \end{figure}
\end{exam}

If $\nu=\lambda$, then the $(i,j)$-entry of the matrix in \eqref{eq:main_LP2} is
\[
  (-1)^{\chi(p_j>q_i)} \ts_{\nu/\lambda(q_i,p_j-1)} =
  \begin{cases}
    \ts_{\lambda^0[p_j,q_i]} & \mbox{if $p_j\le q_i$},\\
    \ts_{\lambda} & \mbox{if $p_j-1 = q_i$},\\
    0 & \mbox{if $p_j-1\ge q_i$},
  \end{cases}
\]
where the right hand side is exactly the same as the definition of
$\ts_{\lambda^0[p_j,q_i]}$. Therefore we obtain the following Lascoux--Pragacz
identity for $\ts_\lm$, in which the case $\mu=\emptyset$ is proved by Macdonald
\cite[(9,9)]{Macdonald_Schur}.

\begin{cor}
  Suppose that $\mu\subseteq \lambda$ and $\theta=(\theta_1,\dots,\theta_k)$ is
  the Lascoux--Pragacz decomposition of $\lm$. Then we have
\[
     \ts_{\lambda/\mu} = \det \left(\ts_{\lambda^0[p_j,q_i]}\right)_{i,j=1}^k.
\]
\end{cor}

By setting $\nu=\emptyset$ in \eqref{eq:main_LP1} and \eqref{eq:main_K1}, we obtain the
following two corollaries.

\begin{cor}\label{cor:main_LP}
  Suppose that $\mu\subseteq \lambda$ and $\theta=(\theta_1,\dots,\theta_k)$ is
  the Lascoux--Pragacz decomposition of $\lm$. Then we have
\[
    \ts_{\lambda}^{k-1} \ts_{\mu}
    = \det \left( (-1)^{\chi(p_j>q_i)} \ts_{\lambda(q_i,p_j-1)}\right)_{i,j=1}^k.
\]
\end{cor}

\begin{cor}\label{cor:main_K}
  Suppose that $\mu\subseteq \lambda$ and $\theta=(\theta_1,\dots,\theta_k)$ is
  the Kreiman decomposition of $\lm$. Then we have 
\[
  \ts_{\mu}^{k-1} \ts_{\lambda} = \det \left( (-1)^{\chi(p_i> q_j)}
    \ts_{\mu(p_i-1,q_j)} \right)_{i,j=1}^k.
\]
\end{cor}

\begin{exam} 
  For $\lambda/\mu=(6,6,6,3,3)/(4,3,2)$, the Lascoux--Pragacz decomposition of
  $\lambda/\mu$ is shown in Figure~\ref{fig:LP}. Then
  Corollary~\ref{cor:main_LP} implies that 
  $\ts_{\lambda}^2 \ts_{\mu}$ is equal to the determinant shown in Figure~\ref{fig:det2}.

  \begin{figure}
    \centering
\[
\ts_{\lambda}^2 \ts_{\mu}= \det \begin{bmatrix}
~\begin{tikzpicture}[scale=.4]
\fill[fill=red!15!white] (0,0) rectangle (3,1);
\fill[fill=red!15!white] (2,1) rectangle (3,3);
\fill[fill=red!15!white] (3,2) rectangle (6,3);
\fill[fill=red!15!white] (5,3) rectangle (6,5);
\draw[thick, line width=1.2pt, pink] (0,0)--(3,0)--(3,2)--(6,2)--(6,5)--(5,5)--(5,3)--(2,3)--(2,1)--(0,1)--(0,0)--cycle;
\draw[color=gray] (0,0)--(3,0)--(3,2)--(6,2)--(6,5)--(0,5)--(0,0)--cycle;
\draw[color=gray] (0,4)--(6,4)
			   (0,3)--(6,3)
			   (0,2)--(3,2)
			   (0,1)--(3,1)
			   (1,0)--(1,5)
			   (2,0)--(2,5)
			   (3,2)--(3,5)
			   (4,2)--(4,5)
			   (5,2)--(5,5);
\draw[thick, line width=1.2pt] (0,1)--(2,1)--(2,3)--(5,3)--(5,5)--(0,5)--(0,1)--cycle;
\node[] at (.5, .5) {$\color{blue} p_1$};
\node[] at (5.5, 4.5) {$\color{red} q_1$};
\node[] at (4.2, 2.5) {\tiny\color{red}removed};
\end{tikzpicture}~
&~ \begin{tikzpicture}[scale=.4]
\fill[fill=red!15!white] (4,2) rectangle (6,3);
\fill[fill=red!15!white] (5,3) rectangle (6,5);
\draw[color=gray] (0,0)--(3,0)--(3,2)--(6,2)--(6,5)--(0,5)--(0,0)--cycle;
\draw[color=gray] (0,4)--(6,4)
			   (0,3)--(6,3)
			   (0,2)--(3,2)
			   (0,1)--(3,1)
			   (1,0)--(1,5)
			   (2,0)--(2,5)
			   (3,2)--(3,5)
			   (4,2)--(4,5)
			   (5,2)--(5,5);
\draw[thick, line width=1.2pt, pink] (4,2)--(6,2)--(6,5)--(5,5)--(5,3)--(4,3)--(4,2)--cycle;
\draw[thick, line width=1.2pt] (0,0)--(3,0)--(3,2)--(4,2)--(4,3)--(5,3)--(5,5)--(0,5)--(0,0)--cycle;
\node[] at (4.5, 2.5) {$\color{blue} p_2$};
\node[] at (5.5, 4.5) {$\color{red} q_1$};
\end{tikzpicture}~
& ~\begin{tikzpicture}[scale=.4]
\fill[fill=red!15!white] (1,0) rectangle (3,1);
\fill[fill=red!15!white] (2,1) rectangle (3,3);
\fill[fill=red!15!white] (3,2) rectangle (6,3);
\fill[fill=red!15!white] (5,3) rectangle (6,5);
\draw[color=gray] (0,0)--(3,0)--(3,2)--(6,2)--(6,5)--(0,5)--(0,0)--cycle;
\draw[color=gray] (0,4)--(6,4)
			   (0,3)--(6,3)
			   (0,2)--(3,2)
			   (0,1)--(3,1)
			   (1,0)--(1,5)
			   (2,0)--(2,5)
			   (3,2)--(3,5)
			   (4,2)--(4,5)
			   (5,2)--(5,5);
\draw[thick, line width=1.2pt, pink] (1,0)--(3,0)--(3,2)--(6,2)--(6,5)--(5,5)--(5,3)--(2,3)--(2,1)--(1,1)--(1,0)--cycle;
\draw[thick, line width=1.2pt] (0,0)--(1,0)--(1,1)--(2,1)--(2,3)--(5,3)--(5,5)--(0,5)--(0,0)--cycle;
\node[] at (1.5, .5) {$\color{blue} p_3$};
\node[] at (5.5, 4.5) {$\color{red} q_1$};
\end{tikzpicture}~\\\\
~\begin{tikzpicture}[scale=.4]
\fill[fill=red!15!white] (0,0) rectangle (3,1);
\fill[fill=red!15!white] (2,1) rectangle (3,3);
\fill[fill=red!15!white] (3,2) rectangle (6,3);
\fill[fill=red!15!white] (5,3) rectangle (6,4);
\draw[color=gray] (0,0)--(3,0)--(3,2)--(6,2)--(6,5)--(0,5)--(0,0)--cycle;
\draw[color=gray] (0,4)--(6,4)
			   (0,3)--(6,3)
			   (0,2)--(3,2)
			   (0,1)--(3,1)
			   (1,0)--(1,5)
			   (2,0)--(2,5)
			   (3,2)--(3,5)
			   (4,2)--(4,5)
			   (5,2)--(5,5);
\draw[thick, line width=1.2pt, pink] (0,0)--(3,0)--(3,2)--(6,2)--(6,4)--(5,4)--(5,3)--(2,3)--(2,1)--(0,1)--(0,0)--cycle;
\draw[thick, line width=1.2pt] (0,1)--(2,1)--(2,3)--(5,3)--(5,4)--(6,4)--(6,5)--(0,5)--(0,1)--cycle;
\node[] at (.5, .5) {$\color{blue} p_1$};
\node[] at (5.5, 3.5) {$\color{red} q_2$};
\end{tikzpicture}~
&~ \begin{tikzpicture}[scale=.4]
\fill[fill=red!15!white] (4,2) rectangle (6,3);
\fill[fill=red!15!white] (5,3) rectangle (6,4);
\draw[color=gray] (0,0)--(3,0)--(3,2)--(6,2)--(6,5)--(0,5)--(0,0)--cycle;
\draw[color=gray] (0,4)--(6,4)
			   (0,3)--(6,3)
			   (0,2)--(3,2)
			   (0,1)--(3,1)
			   (1,0)--(1,5)
			   (2,0)--(2,5)
			   (3,2)--(3,5)
			   (4,2)--(4,5)
			   (5,2)--(5,5);
\draw[thick, line width=1.2pt, pink] (4,2)--(6,2)--(6,4)--(5,4)--(5,3)--(4,3)--(4,2)--cycle;
\draw[thick, line width=1.2pt] (0,0)--(3,0)--(3,2)--(4,2)--(4,3)--(5,3)--(5,4)--(6,4)--(6,5)--(0,5)--(0,0)--cycle;
\node[] at (4.5, 2.5) {$\color{blue} p_2$};
\node[] at (5.5, 3.5) {$\color{red} q_2$};
\end{tikzpicture}~
&~ \begin{tikzpicture}[scale=.4]
\fill[fill=red!15!white] (1,0) rectangle (3,1);
\fill[fill=red!15!white] (2,1) rectangle (3,3);
\fill[fill=red!15!white] (3,2) rectangle (6,3);
\fill[fill=red!15!white] (5,3) rectangle (6,4);
\draw[color=gray] (0,0)--(3,0)--(3,2)--(6,2)--(6,5)--(0,5)--(0,0)--cycle;
\draw[color=gray] (0,4)--(6,4)
			   (0,3)--(6,3)
			   (0,2)--(3,2)
			   (0,1)--(3,1)
			   (1,0)--(1,5)
			   (2,0)--(2,5)
			   (3,2)--(3,5)
			   (4,2)--(4,5)
			   (5,2)--(5,5);
\draw[thick, line width=1.2pt, pink] (1,0)--(3,0)--(3,2)--(6,2)--(6,4)--(5,4)--(5,3)--(2,3)--(2,1)--(1,1)--(1,0)--cycle;
\draw[thick, line width=1.2pt] (0,0)--(1,0)--(1,1)--(2,1)--(2,3)--(5,3)--(5,4)--(6,4)--(6,5)--(0,5)--(0,0)--cycle;
\node[] at (1.5, .5) {$\color{blue} p_3$};
\node[] at (5.5, 3.5) {$\color{red} q_2$};
\end{tikzpicture}~\\\\
~\begin{tikzpicture}[scale=.4]
\fill[fill=red!15!white] (0,0) rectangle (3,1);
\draw[color=gray] (0,0)--(3,0)--(3,2)--(6,2)--(6,5)--(0,5)--(0,0)--cycle;
\draw[color=gray] (0,4)--(6,4)
			   (0,3)--(6,3)
			   (0,2)--(3,2)
			   (0,1)--(3,1)
			   (1,0)--(1,5)
			   (2,0)--(2,5)
			   (3,2)--(3,5)
			   (4,2)--(4,5)
			   (5,2)--(5,5);
\draw[thick, line width=1.2pt, pink] (0,0)--(3,0)--(3,1)--(0,1)--(0,0)--cycle;
\draw[thick, line width=1.2pt] (0,1)--(3,1)--(3,2)--(6,2)--(6,5)--(0,5)--(0,1)--cycle;
\node[] at (.5, .5) {$\color{blue} p_1$};
\node[] at (2.5, .5) {$\color{red} q_3$};
\end{tikzpicture}~
&~ \begin{tikzpicture}[scale=.4]
\fill[fill=cyan!15!white] (3,0) rectangle (4,2);
\fill[fill=cyan!15!white] (4,1) rectangle (5,2);
\draw[color=gray] (0,0)--(3,0)--(3,2)--(6,2)--(6,5)--(0,5)--(0,0)--cycle;
\draw[color=gray] (0,4)--(6,4)
			   (0,3)--(6,3)
			   (0,2)--(3,2)
			   (0,1)--(4,1)
			   (1,0)--(1,5)
			   (2,0)--(2,5)
			   (3,2)--(3,5)
			   (4,1)--(4,5)
			   (5,2)--(5,5);
\draw[thick, line width=1.2pt, cyan] (3,0)--(3,2)--(5,2)--(5,1)--(4,1)--(4,0)--(3,0)--cycle;
\draw[thick, line width=1.2pt] (0,0)--(4,0)--(4,1)--(5,1)--(5,2)--(6,2)--(6,5)--(0,5)--(0,0)--cycle;
\node[] at (4.5, 2.5) {$\color{blue} p_2$};
\node[] at (2.5, .5) {$\color{red} q_3$};
\node[] at (4, 1.5) {\tiny\color{blue}added};
\node[] at (-.75, 2.5) {$ -$};
\end{tikzpicture}~
& ~\begin{tikzpicture}[scale=.4]
\fill[fill=red!15!white] (1,0) rectangle (3,1);
\draw[color=gray] (0,0)--(3,0)--(3,2)--(6,2)--(6,5)--(0,5)--(0,0)--cycle;
\draw[color=gray] (0,4)--(6,4)
			   (0,3)--(6,3)
			   (0,2)--(3,2)
			   (0,1)--(3,1)
			   (1,0)--(1,5)
			   (2,0)--(2,5)
			   (3,2)--(3,5)
			   (4,2)--(4,5)
			   (5,2)--(5,5);
\draw[thick, line width=1.2pt, pink] (1,0)--(3,0)--(3,1)--(1,1)--(1,0)--cycle;
\draw[thick, line width=1.2pt] (0,0)--(1,0)--(1,1)--(3,1)--(3,2)--(6,2)--(6,5)--(0,5)--(0,0)--cycle;
\node[] at (1.5, .5) {$\color{blue} p_3$};
\node[] at (2.5, .5) {$\color{red} q_3$};
\end{tikzpicture}~
 \end{bmatrix}
\]
\caption{An illustrate of the determinant $\det \left( (-1)^{\chi(p_j>q_i)}
    \ts_{\lambda(q_i,p_j-1)}\right)_{i,j=1}^k$. The $(i,j)$-entry shows the sign
  $(-1)^{\chi(p_j>q_i)}$ and the partition $\lambda(q_i,p_j-1)$, which is
  obtained from $\lambda$ by either removing (red) or adding (blue) a border
  strip.}
    \label{fig:det2}
  \end{figure}
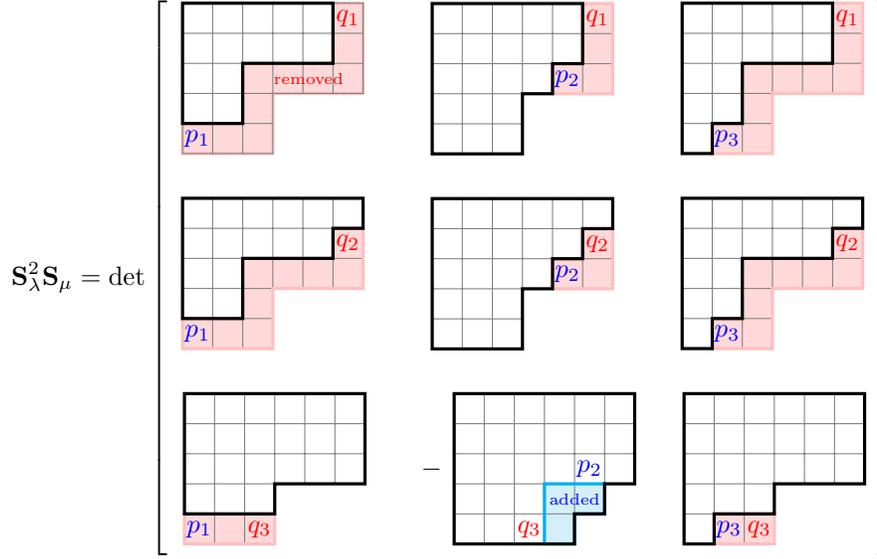
\end{exam}

We now show that Corollary~\ref{cor:main_LP} implies (a corrected version of) a
conjecture of Morales, Pak, and Panova \cite{MPP2}.

Macdonald's 6th variation $s_\lambda(\vec x|\vec a)$ of Schur functions
\cite{Macdonald_Schur}, also known as \emph{factorial Schur functions}, are
defined by
\[
  s_\lambda(\vec x|\vec a) = \frac{\det \left(
      (x_i-a_1)(x_i-a_2)\cdots(x_i-a_{\lambda_j+d-j})\right)_{i,j=1}^d}{\prod_{1\le
      i<j\le d}(x_i-x_j)},
\]
where $\vec x = (x_1,\dots,x_d)$ is a sequence of variables and $\vec a =
(a_1,a_2,\dots)$ is a sequence of parameters. Note that $s_\lambda(\vec x|\vec
a)$ is a symmetric polynomial in the variables $\vec x$ with parameters $\vec
a$. If $a_i=0$ for all $i$, then $s_\lambda(\vec x|\vec a)$ becomes the Schur
polynomial $s_\lambda(\vec x)$. The factorial Schur functions $s_\lambda(\vec
x|\vec a)$ are a special case of Macdonald's 9th variation $\ts_\lambda$ of
Schur functions.

By specializing the Macdonald's 9th variation to the factorial Schur functions
in Corollary~\ref{cor:main_LP}, we obtain the following result, which is a
corrected version of a conjecture proposed by Morales, Pak, and Panova
\cite{MPP2}.

\begin{cor}
  \label{cor:mpp}
  Let $\theta=(\theta_1,\dots,\theta_k)$ be the Lascoux--Pragacz decomposition
  of a skew shape $\lm$ with $\ell(\lambda)\le d$. Then
\begin{equation}\label{eq:mpp2}
  s_\mu(\vec x|\vec a) s_\lambda(\vec x|\vec a)^{k-1} =
  \det\left((-1)^{\chi(p_j>q_i)} s_{\lambda(q_i,p_j-1)}(\vec x|\vec a)
  \right)_{i,j=1}^k,
\end{equation}
where $\vec x = (x_1,\dots,x_d)$ and $\vec a = (a_1,a_2,\dots)$.
\end{cor}

\begin{remark}
  The original conjecture of Morales, Pak, and Panova
  \cite[Conjecture~7.9]{MPP2} states that, under the same assumption in
  Corollary~\ref{cor:mpp},
  \begin{equation}\label{eq:mpp}
  s_\mu(\vec x|\vec a) s_\lambda(\vec x|\vec a)^{k-1} =
  \det\left(s_{\lambda\setminus \lambda^0[p_j,q_i]}(\vec x|\vec a)
  \right)_{i,j=1}^k.
\end{equation}  
In fact \eqref{eq:mpp} is not true even for the Schur function case. As a
counterexample, if $\lambda=(2,1)$, $\mu=(1)$, $\vec x = (x_1,x_2)$, and $\vec
a=(0,0,\dots)$, then the Lascoux--Pragacz decomposition
$\theta=(\theta_1,\theta_2)$ has two border strips with $p_1=q_1=1$ and
$p_2=q_2=-1$. One can easily check using Pieri's rule that the left hand side of
\eqref{eq:mpp} is
\begin{equation}\label{eq:ce}
  s_\mu(\vec x) s_\lambda(\vec x) = s_{(3,1)}(\vec x) + s_{(2,2)}(\vec x)
  +s_{(2,1,1)}(\vec x) = s_{(3,1)}(\vec x) + s_{(2,2)}(\vec x).
\end{equation}
However, the right hand side of \eqref{eq:mpp} is
\begin{align*}
  \det\left(s_{\lambda\setminus \lambda^0[p_j,q_i]}(\vec x) \right)_{i,j=1}^2
  &= \det
\begin{pmatrix}
 s_{\lambda\setminus \lambda^0[1,1]}(\vec x)  & s_{\lambda\setminus \lambda^0[-1,1]}(\vec x) \\
 s_{\lambda\setminus \lambda^0[1,-1]}(\vec x)  & s_{\lambda\setminus \lambda^0[-1,-1]}(\vec x) \\
\end{pmatrix}\\
&= \det
\begin{pmatrix}
 s_{(1,1)}(\vec x)  & s_{\emptyset}(\vec x) \\
 0  & s_{(2)}(\vec x) \\
\end{pmatrix}\\
 &=  s_{(3,1)}(\vec x) + s_{(2,1,1)}(\vec x) = s_{(3,1)}(\vec x).
\end{align*}

Note that the right hand side of \eqref{eq:mpp2} for the running example is
\begin{align*}
\det\left((-1)^{\chi(p_j>q_i)} s_{\lambda(q_i,p_j-1)}(\vec x) \right)_{i,j=1}^2
  &= \det
\begin{pmatrix}
 s_{(1,1)}(\vec x)  & s_{\emptyset}(\vec x) \\
 -s_{(2,2)}(\vec x)  & s_{(2)}(\vec x) \\
\end{pmatrix}\\
&= s_{(3,1)}(\vec x) + s_{(2,1,1)}(\vec x) + s_{(2,2)}(\vec x)= s_{(3,1)}(\vec x)
+s_{(2,2)}(\vec x),
\end{align*}
which is equal to $s_\mu(\vec x) s_\lambda(\vec x)$ as shown in \eqref{eq:ce}.
\end{remark}

\section{A generalized Hamel--Goulden formula}
\label{sec:gener-hamel-gould}

In this section we give a generalization of the Hamel--Goulden formula. Our
result involves generalizations of cutting strips as well as Schur functions. As corollaries we
obtain Jin's result \cite{Jin_2018} and a generalized Giambelli formula.

We first introduce some definitions.

\begin{defn}\label{defn:compatible}
  Let $\lambda$ and $\nu$ be partitions with $\lambda\subseteq\nu$. A border
  strip $\gamma$ is \emph{$\nl$-compatible} if the following condition holds:
  \begin{itemize}
  \item $\Cont(\nl)\subseteq \Cont(\lambda)$, and 
  \item for any connected component $\alpha$ of $\nl$, we have
\[
\lambda^0[a-1,b+1] = \gamma[a-1,b+1],
\]
where $a=\min(\Cont(\alpha))$ and $b=\max(\Cont(\alpha))$.
  \end{itemize}
See Figure~\ref{fig:compatible}.
\end{defn}

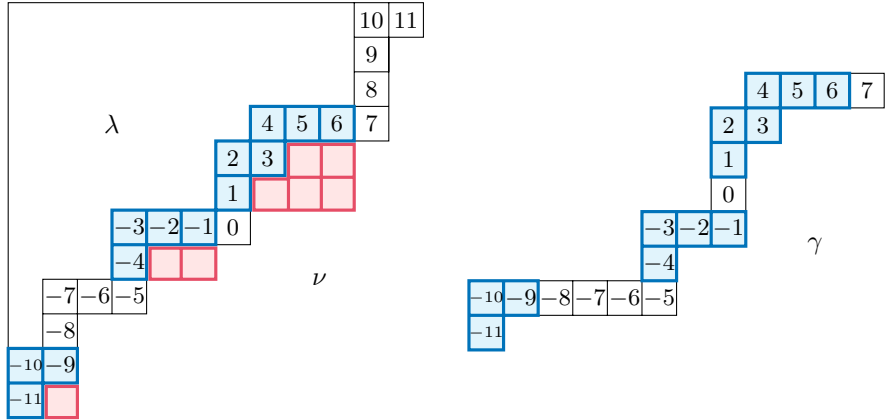
\begin{figure}
  \centering
\begin{tikzpicture}[scale=.46]
\fill[fill=cyan!10!white] (0,2) rectangle (1,4)
				(1,3) rectangle (2,4)
				(3,6)--(4,6)--(4,7)--(6,7)--(6,8)--(3,8)--(3,6)
				(6,8)--(7,8)--(7,9)--(8,9)--(8,10)--(10,10)--(10,11)--(7,11)--(7,10)--(6,10)--(6,8);
\draw (0,2)--(1,2)--(1,3)--(2,3)--(2,5)--(4,5)--(4,7)--(7,7)--(7,9)--(8,9)--(8,10)--(11,10)--(11,13)--(12,13)--(12,14)--(0,14)--(0,2)--cycle
	(1,4)--(1,6)--(3,6)--(3,5)
	(1,5)--(2,5)--(2,6)
	(11,11)--(10,11)--(10,14)
	(10,13)--(11,13)--(11,14)
	(10,12)--(11,12)--(11,13);
\draw[very thick, color=frenchblue] (0,2)--(1,2)--(1,3)--(2,3)--(2,4)--(0,4)--(0,2)--cycle
			(0,3)--(1,3)--(1,4)
			(3,6)--(4,6)--(4,7)--(6,7)--(6,8)--(3,8)--(3,6)--cycle
			(3,7)--(4,7)--(4,8)
			(5,7)--(5,8)
			(6,8)--(7,8)--(7,9)--(8,9)--(8,10)--(10,10)--(10,11)--(7,11)--(7,10)--(6,10)--(6,8)--cycle
			(6,9)--(7,9)--(7,10)--(8,10)--(8,11)
			(9,10)--(9,11);
\fill [fill=red!10!white] (1.1, 2)--(1.1, 2.9)--(2, 2.9)--(2,2)--(1.1,2)--cycle
			(4.1, 6)--(6,6)--(6,6.9)--(4.1, 6.9)--(4.1, 6)--cycle
			(7.1, 8)--(7.1, 8.9)--(8.1, 8.9)--(8.1, 9.9)--(10, 9.9)--(10, 8)--(7.1, 8)--cycle;
\draw [very thick, color=dredcolor] (1.1, 2)--(1.1, 2.9)--(2, 2.9)--(2,2)--(1.1,2)--cycle
			(4.1, 6)--(6,6)--(6,6.9)--(4.1, 6.9)--(4.1, 6)--cycle
			(5,6)--(5,6.9)
			(8.1,8.95)--(10,8.95)
			(8.1, 8)--(8.1,8.95)
			(9.05, 8)--(9.05, 9.9)
			(7.1, 8)--(7.1, 8.9)--(8.1, 8.9)--(8.1, 9.9)--(10, 9.9)--(10, 8)--(7.1, 8)--cycle;
\node[] at (3, 10.5) {$\lambda$};
\node[] at (9,6) {$\nu$};
\node[] at (.5, 2.5) {\tiny $-11$};
\node[] at (.5, 3.5) {\tiny $-10$};
\node[] at (1.5, 3.5) {\small $-9$};
\node[] at (1.5, 4.5) {\small $-8$};
\node[] at (1.5, 5.5) {\small $-7$};
\node[] at (2.5, 5.5) {\small $-6$};
\node[] at (3.5, 5.5) {\small $-5$};
\node[] at (3.5, 6.5) {\small $-4$};
\node[] at (3.5, 7.5) {\small $-3$};
\node[] at (4.5, 7.5) {\small $-2$};
\node[] at (5.5, 7.5) {\small $-1$};
\node[] at (6.5, 7.5) {\small $0$};
\node[] at (6.5, 8.5) {\small $1$};
\node[] at (6.5, 9.5) {\small $2$};
\node[] at (7.5, 9.5) {\small $3$};
\node[] at (7.5, 10.5) {\small $4$};
\node[] at (8.5, 10.5) {\small $5$};
\node[] at (9.5, 10.5) {\small $6$};
\node[] at (10.5, 10.5) {\small $7$};
\node[] at (10.5, 11.5) {\small $8$};
\node[] at (10.5, 12.5) {\small $9$};
\node[] at (10.5, 13.5) {\small $10$};
\node[] at (11.5, 13.5) {\small $11$};
\end{tikzpicture}\quad
\begin{tikzpicture}[scale=.46]
\fill[fill=cyan!10!white]  (0,0)--(1,0)--(1,1)--(2,1)--(2,2)--(0,2)--(0,0)--cycle
			(5,2)--(6,2)--(6,3)--(8,3)--(8,4)--(5,4)--(5,2)--cycle
			(7,5)--(8,5)--(8,6)--(9,6)--(9,7)--(11,7)--(11,8)--(8,8)--(8,7)--(7,7)--(7,5)--cycle;
\draw (2,1)--(6,1)--(6,2)--(2,2)--(2,1)--cycle
	(3,1)--(3,2)
	(4,1)--(4,2)
	(5,1)--(5,2)
	(7,4)--(7,5)
	(8,4)--(8,5)
	(11,7)--(12,7)--(12,8)--(11,8);
\draw[very thick, color=frenchblue] (0,0)--(1,0)--(1,1)--(2,1)--(2,2)--(0,2)--(0,0)--cycle
			(0,1)--(1,1)--(1,2)
			(5,2)--(6,2)--(6,3)--(8,3)--(8,4)--(5,4)--(5,2)--cycle
			(5,3)--(6,3)--(6,4)
			(7,3)--(7,4)
			(7,5)--(8,5)--(8,6)--(9,6)--(9,7)--(11,7)--(11,8)--(8,8)--(8,7)--(7,7)--(7,5)--cycle
			(7,6)--(8,6)--(8,7)--(9,7)--(9,8)
			(10,7)--(10,8); 
\node[] at (.5, .5) {\tiny $-11$};
\node[] at (.5, 1.5) {\tiny $-10$};
\node[] at (1.5, 1.5) {\small $-9$};
\node[] at (2.5, 1.5) {\small $-8$};
\node[] at (3.5, 1.5) {\small $-7$};
\node[] at (4.5, 1.5) {\small $-6$};
\node[] at (5.5, 1.5) {\small $-5$};
\node[] at (5.5, 2.5) {\small $-4$};
\node[] at (5.5, 3.5) {\small $-3$};
\node[] at (6.5, 3.5) {\small $-2$};
\node[] at (7.5, 3.5) {\small $-1$};
\node[] at (7.5, 4.5) {\small $0$};
\node[] at (7.5, 5.5) {\small $1$};
\node[] at (7.5, 6.5) {\small $2$};
\node[] at (8.5, 6.5) {\small $3$};
\node[] at (8.5, 7.5) {\small $4$};
\node[] at (9.5, 7.5) {\small $5$};
\node[] at (10.5, 7.5) {\small $6$};
\node[] at (11.5, 7.5) {\small $7$};
\node[] at (10, 3) {$\gamma$};
\node[] at (0,-1.5) {\phantom{1}};
\end{tikzpicture}
\caption{A skew shape $\nl$ and the outer strip $\lambda^0$ with contents are
  shown on the left, where the cells of $\nl$ are colored red. The border strip
  $\gamma$ on the right is $\nl$-compatible. }
  \label{fig:compatible}
\end{figure}

\begin{defn}\label{defn:comp2}
  Let $\lambda$ and $\nu$ be partitions with $\lambda\subseteq\nu$. A partition
  $\mu\subseteq\lambda$ is \emph{$\nl$-compatible} if the following condition
  holds: 
\begin{itemize}
  \item $\Cont(\nl)\subseteq \Cont(\lambda)$, and 
  \item for any connected component $\alpha$ of $\nl$, we have
\[
\lambda^0[a,b] = \mu^+[a,b],
\]
where $a=\min(\Cont(\alpha))$ and $b=\max(\Cont(\alpha))$.
\end{itemize}
See Figure~\ref{fig:compatible2}.
\end{defn}

In Definitions~\ref{defn:compatible} and \ref{defn:comp2} one can use $\lambda^+$
instead of $\lambda^0$; if $\alpha$ is a connected component of $\nl$, then
since $\Cont(\alpha)\subseteq\Cont(\lambda)$, we have
\[
  \lambda^0[a-1,b+1] = \lambda^+[a-1,b+1],
\]
where $a=\min(\Cont(\alpha))$ and $b=\max(\Cont(\alpha))$.

\begin{figure}
  \centering
  \begin{tikzpicture}[scale=.43]
\fill[fill=cyan!10!white] (0,1) rectangle (1,2)
			(2,6) rectangle ( 4,7)
			(4,9) rectangle (6,10)
			(5, 10) rectangle (7,11);
\draw (0,0)--(1,0)--(1,1)--(2,1)--(2,3)--(4,3)--(4,5)--(7,5)--(7,7)--(8,7)--(8,8)--(11,8)--(11,11)--(12,11)--(12,12)--(0,12)--(0,0)--cycle;
\draw [line width= 1.5pt, color=oxfordblue] (0,4)--(1,4)--(1,5)--(2,5)--(2,7)--(3,7)--(3,10)--(5,10)--(5,11)--(8,11)--(8,12)--(0,12)--(0,4)--cycle;
\draw[very thick, color=frenchblue] (0,1)--(1,1)--(1,2)--(0,2)--(0,1)--cycle
			(2,6)--(4,6)--(4,7)--(2,7)--(2,6)--cycle
			(3,6)--(3,7)
			(4,9)--(6,9)--(6,10)--(7,10)--(7,11)--(5,11)--(5,10)--(4,10)--(4,9)--cycle
			(5,9)--(5,10)--(6,10)--(6,11);
\fill [fill=red!10!white] (1.1,0)--(2,0)--(2,.9)--(1.1,.9)--(1.1,0)--cycle
			(4.1,4)--(6,4)--(6,4.9)--(4.1,4.9)--(4.1,4)--cycle
			(7.1,6)--(10,6)--(10,7.9)--(8.1,7.9)--(8.1,6.9)--(7.1,6.9)--(7.1,6)--cycle;
\draw [line width=1.3pt, color=dredcolor] (1.1,0)--(2,0)--(2,.9)--(1.1,.9)--(1.1,0)--cycle
			(4.1,4)--(6,4)--(6,4.9)--(4.1,4.9)--(4.1,4)--cycle
			(7.1,6)--(10,6)--(10,7.9)--(8.1,7.9)--(8.1,6.9)--(7.1,6.9)--(7.1,6)--cycle
			(5,4)--(5,4.9)
			(8.1, 6.9)--(10, 6.9)
			(8.1, 6)--(8.1, 7.1)
			(9.05,6)--(9.05,7.9);
\draw[thick, dashed, color=orange] (.5, 1.5)--(1.5, .5)
			(2.5, 6.5)--(4.5, 4.5)
			(3.5, 6.5)--(5.5, 4.5)
			(6.5, 10.5)--(9.5, 7.5)
			(4.5, 9.5)--(7.5, 6.5);
\filldraw[color=orange] (.5, 1.5) circle (3pt)
			(1.5, .5) circle (3pt)
			(2.5, 6.5) circle (3pt)
			(4.5, 4.5) circle (3pt)
			(3.5, 6.5) circle (3pt)
			(5.5, 4.5) circle (3pt)
			(4.5, 9.5) circle (3pt)
			(7.5, 6.5) circle (3pt)
			(6.5, 10.5) circle (3pt)
			(9.5, 7.5) circle (3pt);
\node[] at (1.5,9.6) {$\mu$};
\end{tikzpicture}
\caption{A $\nl$-compatible partition $\mu$ for the skew shape $\nl$ in
  Figure~\ref{fig:compatible}.}
  \label{fig:compatible2}
\end{figure}
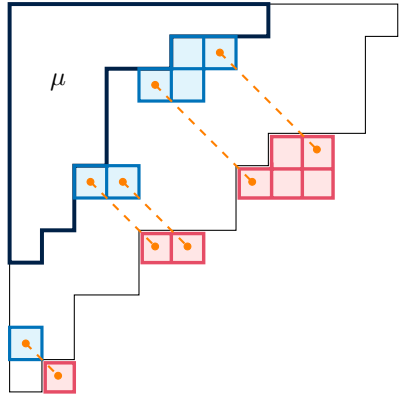

Note that if $\nu=\lambda$, then by definition every border strip $\gamma$ and
every partition $\mu\subseteq\lambda$ are $\nu/\lambda$-compatible.

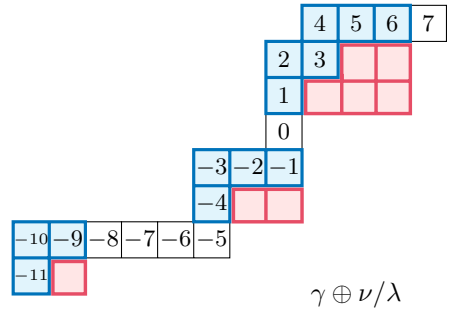
\begin{figure}
  \centering
  \begin{tikzpicture}[scale=.48]
\fill[fill=cyan!10!white]  (0,0)--(1,0)--(1,1)--(2,1)--(2,2)--(0,2)--(0,0)--cycle
			(5,2)--(6,2)--(6,3)--(8,3)--(8,4)--(5,4)--(5,2)--cycle
			(7,5)--(8,5)--(8,6)--(9,6)--(9,7)--(11,7)--(11,8)--(8,8)--(8,7)--(7,7)--(7,5)--cycle;
\draw (2,1)--(6,1)--(6,2)--(2,2)--(2,1)--cycle
	(3,1)--(3,2)
	(4,1)--(4,2)
	(5,1)--(5,2)
	(7,4)--(7,5)
	(8,4)--(8,5)
	(11,7)--(12,7)--(12,8)--(11,8);
\draw[very thick, color=frenchblue] (0,0)--(1,0)--(1,1)--(2,1)--(2,2)--(0,2)--(0,0)--cycle
			(0,1)--(1,1)--(1,2)
			(5,2)--(6,2)--(6,3)--(8,3)--(8,4)--(5,4)--(5,2)--cycle
			(5,3)--(6,3)--(6,4)
			(7,3)--(7,4)
			(7,5)--(8,5)--(8,6)--(9,6)--(9,7)--(11,7)--(11,8)--(8,8)--(8,7)--(7,7)--(7,5)--cycle
			(7,6)--(8,6)--(8,7)--(9,7)--(9,8)
			(10,7)--(10,8);
\fill [fill=red!10!white] (1.1, 0)--(1.1, .9)--(2, .9)--(2,0)--(1.1,0)--cycle
                         (6.1, 2)--(8,2)--(8,2.9)--(6.1, 2.9)--(6.1, 2)--cycle
			(8.1, 5)--(8.1, 5.9)--(9.1, 5.9)--(9.1, 6.9)--(11, 6.9)--(11, 5)--(8.1, 5)--cycle;
\draw [line width=1.3pt, color=dredcolor] (1.1,0)--(2,0)--(2,.9)--(1.1,.9)--(1.1,0)--cycle
			(6.1, 2)--(8,2)--(8,2.9)--(6.1, 2.9)--(6.1, 2)--cycle
			(8.1, 5)--(8.1, 5.9)--(9.1, 5.9)--(9.1, 6.9)--(11, 6.9)--(11, 5)--(8.1, 5)--cycle
			(7,2)--(7,2.9)
			(9.1, 5.9)--(11, 5.9)
			(9.1, 5)--(9.1, 6.1)
			(10.05,5)--(10.05,6.9);
\node[] at (.5, .5) {\tiny $-11$};
\node[] at (.5, 1.5) {\tiny $-10$};
\node[] at (1.5, 1.5) {\small $-9$};
\node[] at (2.5, 1.5) {\small $-8$};
\node[] at (3.5, 1.5) {\small $-7$};
\node[] at (4.5, 1.5) {\small $-6$};
\node[] at (5.5, 1.5) {\small $-5$};
\node[] at (5.5, 2.5) {\small $-4$};
\node[] at (5.5, 3.5) {\small $-3$};
\node[] at (6.5, 3.5) {\small $-2$};
\node[] at (7.5, 3.5) {\small $-1$};
\node[] at (7.5, 4.5) {\small $0$};
\node[] at (7.5, 5.5) {\small $1$};
\node[] at (7.5, 6.5) {\small $2$};
\node[] at (8.5, 6.5) {\small $3$};
\node[] at (8.5, 7.5) {\small $4$};
\node[] at (9.5, 7.5) {\small $5$};
\node[] at (10.5, 7.5) {\small $6$};
\node[] at (11.5, 7.5) {\small $7$};
\node[] at (9.5, 0) {$\gamma\oplus\nl$};
\end{tikzpicture}
\caption{The skew shape $\gamma\oplus \nl$ for the border strip $\gamma$ and the
  skew shape $\nl$ in Figure~\ref{fig:compatible}. The contents of the cells in
  $\gamma$ are shown and the cells coming from $\nl$ are colored red.}
  \label{fig:gamma+nl}
\end{figure}

Recall that we have defined $\gamma[a,b]=\{x\in \gamma: c(x)\in [a,b]\}$ for a
border strip $\gamma$. We extend this definition to any connected skew shape
$\alpha$, that is, 
\[
\alpha[a,b]=\{x\in \alpha: c(x)\in [a,b]\}.
\]

\begin{defn}
  Suppose that $\gamma$ is a $\nl$-compatible border strip. Define $\gamma\oplus
  \nl$ to be the skew shape obtained from $\gamma$ by gluing each connected
  component $\alpha$ of $\nl$ below $\gamma$ after shifting $\alpha$ diagonally
  so that the southeast boundary of $\gamma$ and the northwest boundary of
  $\alpha$ have common edges. See Figure~\ref{fig:gamma+nl}.
\end{defn}

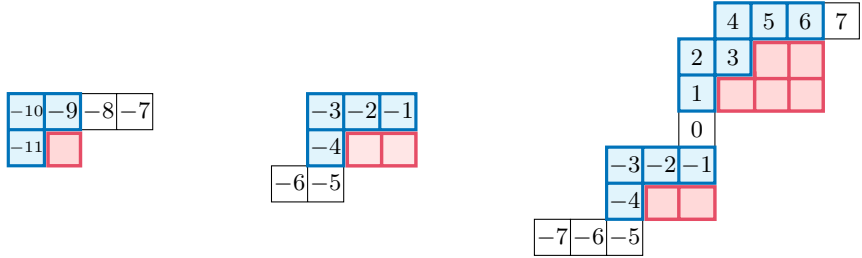
\begin{figure}
  \centering
    \begin{tikzpicture}[scale=.48]
    \draw (2,1)--(4,1)--(4,2)--(2,2)
    		(3,1)--(3,2);
    \fill[fill=cyan!10!white]  (0,0)--(1,0)--(1,1)--(2,1)--(2,2)--(0,2)--(0,0)--cycle;
    \draw[very thick, color=frenchblue] (0,0)--(1,0)--(1,1)--(2,1)--(2,2)--(0,2)--(0,0)--cycle
    			(0,1)--(1,1)--(1,2);
  \fill [fill=red!15!white] (1.1, 0)--(1.1, .9)--(2, .9)--(2,0)--(1.1,0)--cycle; 
    \draw [line width=1.3pt, color=dredcolor] (1.1, 0)--(1.1, .9)--(2, .9)--(2,0)--(1.1,0)--cycle;
   \node[] at (.5, .5) {\tiny $-11$};
\node[] at (.5, 1.5) {\tiny $-10$};
\node[] at (1.5, 1.5) {\small $-9$};
\node[] at (2.5, 1.5) {\small $-8$};
\node[] at (3.5, 1.5) {\small $-7$};
\node[] at (3, -2) {\phantom{d}};
    \end{tikzpicture}\qquad\qquad
     \begin{tikzpicture}[scale=.48]  
     \draw (5,2)--(4,2)--(4,1)--(6,1)--(6,2)
     		(5,1)--(5,2);
    \fill[fill=cyan!10!white]  (5,2)--(6,2)--(6,3)--(8,3)--(8,4)--(5,4)--(5,2)--cycle; 
      \draw[very thick, color=frenchblue] (5,2)--(6,2)--(6,3)--(8,3)--(8,4)--(5,4)--(5,2)--cycle
			(5,3)--(6,3)--(6,4)
			(7,3)--(7,4);
        \fill[fill=red!10!white] (6.1, 2)--(8,2)--(8,2.9)--(6.1, 2.9)--(6.1, 2)--cycle;
      \draw [line width=1.3pt, color=dredcolor] (6.1, 2)--(8,2)--(8,2.9)--(6.1, 2.9)--(6.1, 2)--cycle
      				(7.05, 2)--(7.05, 2.9);
    \node[] at (4.5, 1.5) {\small $-6$};
\node[] at (5.5, 1.5) {\small $-5$};
\node[] at (5.5, 2.5) {\small $-4$};
\node[] at (5.5, 3.5) {\small $-3$};
\node[] at (6.5, 3.5) {\small $-2$};
\node[] at (7.5, 3.5) {\small $-1$};
\node[] at (6, 0) {\phantom{d}};
     \end{tikzpicture}\qquad\qquad
      \begin{tikzpicture}[scale=.48]  
     \draw (5,2)--(4,2)--(4,1)--(6,1)--(6,2)
     		(5,1)--(5,2)
		(4,2)--(3,2)--(3,1)--(4,1)
		(7,4)--(7,5)
	(8,4)--(8,5)
	(11,7)--(12,7)--(12,8)--(11,8);
    \fill[fill=cyan!10!white]  (5,2)--(6,2)--(6,3)--(8,3)--(8,4)--(5,4)--(5,2)--cycle
    			(7,5)--(8,5)--(8,6)--(9,6)--(9,7)--(11,7)--(11,8)--(8,8)--(8,7)--(7,7)--(7,5)--cycle;
      \draw[very thick, color=frenchblue] (5,2)--(6,2)--(6,3)--(8,3)--(8,4)--(5,4)--(5,2)--cycle
			(5,3)--(6,3)--(6,4)
			(7,3)--(7,4)
			(7,5)--(8,5)--(8,6)--(9,6)--(9,7)--(11,7)--(11,8)--(8,8)--(8,7)--(7,7)--(7,5)--cycle
			(7,6)--(8,6)--(8,7)--(9,7)--(9,8)
			(10,7)--(10,8);
       \fill[fill=red!15!white] (6.1, 2)--(8,2)--(8,2.9)--(6.1, 2.9)--(6.1, 2)--cycle
       				  (8.1, 5)--(8.1, 5.9)--(9.1, 5.9)--(9.1, 6.9)--(11, 6.9)--(11, 5)--(8.1, 5)--cycle;
      \draw [line width=1.3pt, color=dredcolor] (6.1, 2)--(8,2)--(8,2.9)--(6.1, 2.9)--(6.1, 2)--cycle
      (8.1, 5)--(8.1, 5.9)--(9.1, 5.9)--(9.1, 6.9)--(11, 6.9)--(11, 5)--(8.1, 5)--cycle
      (7,2)--(7,2.9)
      (9.1, 5.9)--(11, 5.9)
			(9.1, 5)--(9.1, 6.1)
			(10.05,5)--(10.05,6.9);
   \node[] at (3.5, 1.5) {\small $-7$};
    \node[] at (4.5, 1.5) {\small $-6$};
\node[] at (5.5, 1.5) {\small $-5$};
\node[] at (5.5, 2.5) {\small $-4$};
\node[] at (5.5, 3.5) {\small $-3$};
\node[] at (6.5, 3.5) {\small $-2$};
\node[] at (7.5, 3.5) {\small $-1$};
\node[] at (7.5, 4.5) {\small $0$};
\node[] at (7.5, 5.5) {\small $1$};
\node[] at (7.5, 6.5) {\small $2$};
\node[] at (8.5, 6.5) {\small $3$};
\node[] at (8.5, 7.5) {\small $4$};
\node[] at (9.5, 7.5) {\small $5$};
\node[] at (10.5, 7.5) {\small $6$};
\node[] at (11.5, 7.5) {\small $7$};
     \end{tikzpicture}
  \caption{The diagrams $(\gamma\oplus\nl)[-11,-7]$ (left),
    $(\gamma\oplus\nl)[-6,-1]$ (middle), and $(\gamma\oplus\nl)[-7,7]$
    (right), for the $\gamma\oplus\nl$ in Figure~\ref{fig:gamma+nl}.}
  \label{fig:gamma+nl[a,b]}
\end{figure}

Observe that if $\gamma$ is a $\nl$-compatible border strip and
\begin{equation}
    \label{eq:a,b}
    a,b\in (\Cont(\lambda)\setminus\Cont(\nl))
    \cup\{\min(\Cont(\lambda)), \max(\Cont(\lambda))\},
\end{equation}
then $(\gamma\oplus\nl)[a,b]$ is a connected skew shape. See
Figure~\ref{fig:gamma+nl[a,b]}.

We now state the generalized Hamel--Goulden formula.

\begin{thm}\label{thm:main_HG}
  Let $\mu\subseteq \lambda\subseteq\nu$ be partitions. Suppose that $\mu$ is
  $\nl$-compatible and $\gamma$ is a $\nl$-compatible border strip. Let
  $\theta=(\theta_1,\dots,\theta_k)$ be the decomposition of $\lm$ determined by
  the cutting strip $\gamma$. If $\alpha_1,\dots,\alpha_\ell$ are the connected
  components of $\nl$, we have
\[
\ts_{\nu/\mu} \prod_{s=1}^\ell \ts_{\alpha_s}^{r_s-1} 
    = \det \left( \ts_{(\gamma\oplus\nl)[p_j,q_i]}\right)_{i,j=1}^k,
\]
where $r_s$ is the number of strips $\theta_i$ such that
$\Cont(\alpha_s)\subseteq\Cont(\theta_i)$.
\end{thm}

 Before proving Theorem~\ref{thm:main_HG} we give an example and some of its
 applications.

\begin{exam} 
  Consider the partitions $\nu=(8,8,8,6,6,5,4)$, $\lambda=(8,8,6,6,6,3,3)$, and
  $\mu=(4,3)$ as shown in Figure~\ref{fig:lmn}. The border strip $\gamma$ in
  Figure~\ref{fig:gt} is $\nu/\lambda$-compatible. The decomposition $\theta$ of
  $\lambda/\mu$ with cutting strip $\gamma$ is also shown in
  Figure~\ref{fig:gt}. Since $C(\alpha_1)\subseteq C(\theta_i)$ for $i=1,2,4$
  and $C(\alpha_2)\subseteq C(\theta_j)$ for $j=1,2$, we have $r_1=3$ and
  $r_2=2$. Thus by Theorem \ref{thm:main_HG}, $\ts_{\nu/\mu} \ts_{\alpha_1}^2
  \ts_{\alpha_2}$ is equal to the determinant shown in Figure~\ref{fig:det3}.
  
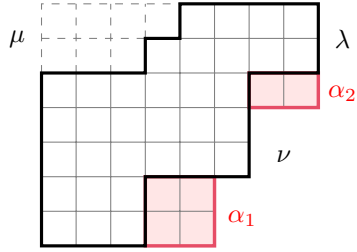
\begin{figure}
  \centering
\begin{tikzpicture}[scale=.46]
\fill[fill=red!10!white] (3,0) rectangle (5,2)
		(6,4) rectangle (8,5);
\draw[color=gray] (0,1)--(4,1)
			(0,2)--(3,2)
			(0,3)--(6,3)
			(0,4)--(6,4)
			(3,5)--(6,5)
			(4,6)--(8,6)
			(1,0)--(1,5)
			(2,0)--(2,5)
			(3,2)--(3,5)
			(4,1)--(4,6)
			(5,2)--(5,7)
			(6,5)--(6,7)
			(7,4)--(7,7)
			(4,0)--(4,1)--(5,1);
\draw[dashed, color=gray] (0,5)--(0,7)--(4,7)
		       (0,6)--(3,6)
		       (1,5)--(1,7)
		       (2,5)--(2,7)
		       (3,6)--(3,7);
  \draw [line width=1.3pt, color=dredcolor] (3,0)--(5,0)--(5,2)--(3,2)--(3,0)--cycle
  		(6,4)--(8,4)--(8,5)--(6,5)--(6,4)--cycle;
\draw[thick, line width=1.2pt] (0,0)--(3,0)--(3,2)--(6,2)--(6,5)--(8,5)--(8,7)--(4,7)--(4,6)--(3,6)--(3,5)--(0,5)--(0,0)--cycle;
\node[] at (-.7, 6) {$\mu$};
\node[] at (8.7, 6) {$\lambda$};
\node[] at (7, 2.6) {$\nu$};
\node[] at (5.8, .8) {\color{red}$\alpha_1$};
\node[] at (8.7, 4.4) {\color{red}$\alpha_2$};
 \end{tikzpicture}
 \caption{The skew shape $\lm$ (black) and the connected components $\alpha_1$
   and $\alpha_2$ of $\nl$ (red).}
  \label{fig:lmn}
\end{figure}

\begin{figure}
  \centering
\begin{tikzpicture}[scale=.46]
    \fill[fill=cyan!10!white] (0,2)--(1,2)--(1,4)--(3,4)--(3,5)--(0,5)--(0,2)--cycle
    			(3,6)--(4,6)--(4,7)--(6,7)--(6,8)--(3,8)--(3,6)--cycle;
    \draw[thick] (0,0)--(1,0)--(1,2)--(0,2)--(0,0)--cycle
    		      (0,1)--(1,1)
		      (2,5)--(3,5)--(3,7)--(2,7)--(2,5)--cycle
		      (2,6)--(3,6)
		      (6,7)--(7,7)--(7,8)--(6,8);
     \draw[very thick, color=frenchblue] (0,2)--(1,2)--(1,4)--(3,4)--(3,5)--(0,5)--(0,2)--cycle
     			(0,3)--(1,3)
			(0,4)--(1,4)--(1,5)
			(2,4)--(2,5)
     			(3,6)--(4,6)--(4,7)--(6,7)--(6,8)--(3,8)--(3,6)--cycle
			(3,7)--(4,7)--(4,8)
			(5,7)--(5,8);
   \node[] at (.5, .5) {\small $-6$};
\node[] at (.5, 1.5) {\small $-5$};
\node[] at (.5, 2.5) {\small $-4$};
\node[] at (.5, 3.5) {\small $-3$};
\node[] at (.5, 4.5) {\small $-2$};
\node[] at (1.5, 4.5) {\small $-1$};
\node[] at (2.5, 4.5) {\small $0$};
\node[] at (2.5, 5.5) {\small $1$};
\node[] at (2.5, 6.5) {\small $2$};
\node[] at (3.5, 6.5) {\small $3$};
\node[] at (3.5, 7.5) {\small $4$};
\node[] at (4.5, 7.5) {\small $5$};
\node[] at (5.5, 7.5) {\small $6$};
\node[] at (6.5, 7.5) {\small $7$};
\node[] at (5, 3.5) {$\gamma$};
    \end{tikzpicture}
\qquad\qquad
 \begin{tikzpicture}[scale=.46]
\draw[color=gray] (0,1)--(3,1)
			(0,2)--(3,2)
			(0,3)--(6,3)
			(0,4)--(6,4)
			(3,5)--(6,5)
			(4,6)--(8,6)
			(1,0)--(1,5)
			(2,0)--(2,5)
			(3,2)--(3,5)
			(4,2)--(4,6)
			(5,2)--(5,7)
			(6,5)--(6,7)
			(7,5)--(7,7);
\draw[dashed, color=gray] (0,5)--(0,7)--(4,7)
		       (0,6)--(3,6)
		       (1,5)--(1,7)
		       (2,5)--(2,7)
		       (3,6)--(3,7);
 \draw[thick, line width=1.2pt] (0,0)--(3,0)--(3,2)--(6,2)--(6,5)--(8,5)--(8,7)--(4,7)--(4,6)--(3,6)--(3,5)--(0,5)--(0,0)--cycle
 			(1,0)--(1,4)--(3,4)--(3,5)
			(2,0)--(2,3)--(4,3)--(4,5)--(5,5)--(5,6)--(8,6)
			(5,2)--(5,4)--(6,4);
\node[] at (1.5, .5) {$\color{blue} p_1$};
\node[] at (7.5, 6.5) {$\color{red} q_1$};
\node[] at (2.5, .5) {$\color{blue} p_2$};
\node[] at (7.5, 5.5) {$\color{red} q_2$};
\node[] at (5.5, 2.5) {$\color{blue} p_3$};
\node[] at (5.5, 3.5) {$\color{red} q_3$};
\node[] at (.5, .5) {$\color{blue} p_4$};
\node[] at (2.5, 4.5) {$\color{red} q_4$};
\node[] at (8.5, 6.5) {$\theta_1$};
\node[] at (8.5, 5.5) {$\theta_2$};
\node[] at (6.5, 3.5) {$\theta_3$};
\node[] at (-.5, .5) {$\theta_4$};
 \end{tikzpicture}
 \caption{A border strip $\gamma$ (left) and the decomposition $\theta$ (right)
   of $\lm$ with cutting strip $\gamma$. }
  \label{fig:gt}
\end{figure}
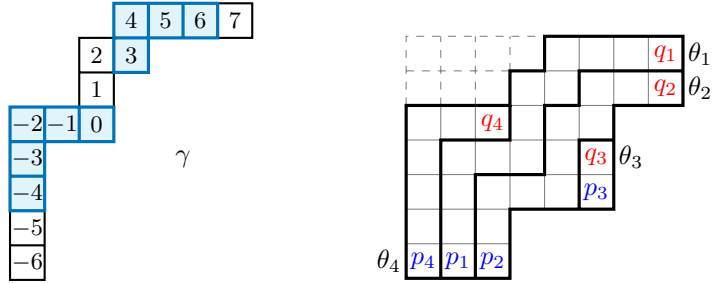

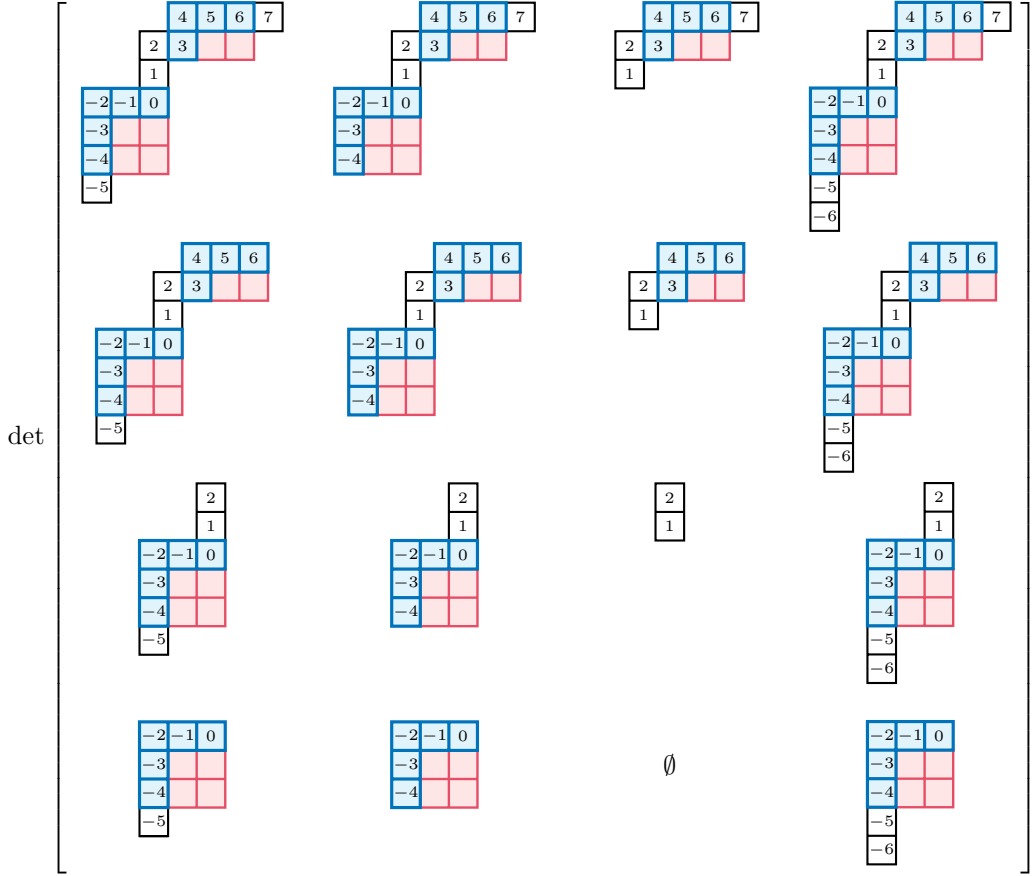
\begin{figure}
  \centering
\[
\det \begin{bmatrix}
~\begin{tikzpicture}[scale=.38]
\fill[fill=red!10!white] (1,2)--(3,2)--(3,4)--(1,4)--(1,2)--cycle
			(4,6)--(6,6)--(6,7)--(4,7)--(4,6)--cycle;
 \draw [line width=.9pt, color=dredcolor] (1,2)--(3,2)--(3,4)--(1,4)--(1,2)--cycle
 		(2,2)--(2,4)
		(1,3)--(3,3)
		(5,6)--(5,7)
  		(4,6)--(6,6)--(6,7)--(4,7)--(4,6)--cycle;
    \fill[fill=cyan!10!white] (0,2)--(1,2)--(1,4)--(3,4)--(3,5)--(0,5)--(0,2)--cycle
    			(3,6)--(4,6)--(4,7)--(6,7)--(6,8)--(3,8)--(3,6)--cycle;
    \draw[thick] (0,1)--(1,1)--(1,2)--(0,2)--(0,1)--cycle	      
		      (2,5)--(3,5)--(3,7)--(2,7)--(2,5)--cycle
		      (2,6)--(3,6)
		      (6,7)--(7,7)--(7,8)--(6,8);
     \draw[very thick, color=frenchblue] (0,2)--(1,2)--(1,4)--(3,4)--(3,5)--(0,5)--(0,2)--cycle
     			(0,3)--(1,3)
			(0,4)--(1,4)--(1,5)
			(2,4)--(2,5)
     			(3,6)--(4,6)--(4,7)--(6,7)--(6,8)--(3,8)--(3,6)--cycle
			(3,7)--(4,7)--(4,8)
			(5,7)--(5,8);
 \node[] at (.5,  .5) {\tiny \phantom{$-6$}};
\node[] at (.5, 1.5) {\tiny $-5$};
\node[] at (.5, 2.5) {\tiny $-4$};
\node[] at (.5, 3.5) {\tiny $-3$};
\node[] at (.5, 4.5) {\tiny $-2$};
\node[] at (1.5, 4.5) {\tiny $-1$};
\node[] at (2.5, 4.5) {\tiny $0$};
\node[] at (2.5, 5.5) {\tiny $1$};
\node[] at (2.5, 6.5) {\tiny $2$};
\node[] at (3.5, 6.5) {\tiny $3$};
\node[] at (3.5, 7.5) {\tiny $4$};
\node[] at (4.5, 7.5) {\tiny $5$};
\node[] at (5.5, 7.5) {\tiny $6$};
\node[] at (6.5, 7.5) {\tiny $7$};
 \end{tikzpicture}~& 
 ~\begin{tikzpicture}[scale=.38]
 \fill[fill=red!10!white] (1,2)--(3,2)--(3,4)--(1,4)--(1,2)--cycle
			(4,6)--(6,6)--(6,7)--(4,7)--(4,6)--cycle;
 \draw [line width=.9pt, color=dredcolor] (1,2)--(3,2)--(3,4)--(1,4)--(1,2)--cycle
 		(2,2)--(2,4)
		(1,3)--(3,3)
		(5,6)--(5,7)
  		(4,6)--(6,6)--(6,7)--(4,7)--(4,6)--cycle;
    \fill[fill=cyan!10!white] (0,2)--(1,2)--(1,4)--(3,4)--(3,5)--(0,5)--(0,2)--cycle
    			(3,6)--(4,6)--(4,7)--(6,7)--(6,8)--(3,8)--(3,6)--cycle;
    \draw[thick] (2,5)--(3,5)--(3,7)--(2,7)--(2,5)--cycle
		      (2,6)--(3,6)
		      (6,7)--(7,7)--(7,8)--(6,8);
     \draw[very thick, color=frenchblue] (0,2)--(1,2)--(1,4)--(3,4)--(3,5)--(0,5)--(0,2)--cycle
     			(0,3)--(1,3)
			(0,4)--(1,4)--(1,5)
			(2,4)--(2,5)
     			(3,6)--(4,6)--(4,7)--(6,7)--(6,8)--(3,8)--(3,6)--cycle
			(3,7)--(4,7)--(4,8)
			(5,7)--(5,8);
   \node[] at (1.5, .5) {\tiny \phantom{$-6$}};
\node[] at (.5, 2.5) {\tiny $-4$};
\node[] at (.5, 3.5) {\tiny $-3$};
\node[] at (.5, 4.5) {\tiny $-2$};
\node[] at (1.5, 4.5) {\tiny $-1$};
\node[] at (2.5, 4.5) {\tiny $0$};
\node[] at (2.5, 5.5) {\tiny $1$};
\node[] at (2.5, 6.5) {\tiny $2$};
\node[] at (3.5, 6.5) {\tiny $3$};
\node[] at (3.5, 7.5) {\tiny $4$};
\node[] at (4.5, 7.5) {\tiny $5$};
\node[] at (5.5, 7.5) {\tiny $6$};
\node[] at (6.5, 7.5) {\tiny $7$};
 \end{tikzpicture}~& 
 ~\begin{tikzpicture}[scale=.38]
 \fill[fill=red!10!white] 
			(4,6)--(6,6)--(6,7)--(4,7)--(4,6)--cycle;
 \draw [line width=.9pt, color=dredcolor] 
 		(5,6)--(5,7)
  		(4,6)--(6,6)--(6,7)--(4,7)--(4,6)--cycle;
    \fill[fill=cyan!10!white] 
    			(3,6)--(4,6)--(4,7)--(6,7)--(6,8)--(3,8)--(3,6)--cycle;
    \draw[thick] (2,5)--(3,5)--(3,7)--(2,7)--(2,5)--cycle
		      (2,6)--(3,6)
		      (6,7)--(7,7)--(7,8)--(6,8);
     \draw[very thick, color=frenchblue] 
        			(3,6)--(4,6)--(4,7)--(6,7)--(6,8)--(3,8)--(3,6)--cycle
			(3,7)--(4,7)--(4,8)
			(5,7)--(5,8);
   \node[] at (1.5, .5) {\tiny \phantom{$-6$}};
\node[] at (2.5, 5.5) {\tiny $1$};
\node[] at (2.5, 6.5) {\tiny $2$};
\node[] at (3.5, 6.5) {\tiny $3$};
\node[] at (3.5, 7.5) {\tiny $4$};
\node[] at (4.5, 7.5) {\tiny $5$};
\node[] at (5.5, 7.5) {\tiny $6$};
\node[] at (6.5, 7.5) {\tiny $7$};
 \end{tikzpicture}~&~\begin{tikzpicture}[scale=.38]
 \fill[fill=red!10!white] (1,2)--(3,2)--(3,4)--(1,4)--(1,2)--cycle
			(4,6)--(6,6)--(6,7)--(4,7)--(4,6)--cycle;
 \draw [line width=.9pt, color=dredcolor] (1,2)--(3,2)--(3,4)--(1,4)--(1,2)--cycle
 		(2,2)--(2,4)
		(1,3)--(3,3)
		(5,6)--(5,7)
  		(4,6)--(6,6)--(6,7)--(4,7)--(4,6)--cycle;
    \fill[fill=cyan!10!white] (0,2)--(1,2)--(1,4)--(3,4)--(3,5)--(0,5)--(0,2)--cycle
    			(3,6)--(4,6)--(4,7)--(6,7)--(6,8)--(3,8)--(3,6)--cycle;
    \draw[thick] (0,0)--(1,0)--(1,2)--(0,2)--(0,0)--cycle
    			(0,1)--(1,1)	      
		      (2,5)--(3,5)--(3,7)--(2,7)--(2,5)--cycle
		      (2,6)--(3,6)
		        (6,7)--(7,7)--(7,8)--(6,8);
     \draw[very thick, color=frenchblue] (0,2)--(1,2)--(1,4)--(3,4)--(3,5)--(0,5)--(0,2)--cycle
     			(0,3)--(1,3)
			(0,4)--(1,4)--(1,5)
			(2,4)--(2,5)
     			(3,6)--(4,6)--(4,7)--(6,7)--(6,8)--(3,8)--(3,6)--cycle
			(3,7)--(4,7)--(4,8)
			(5,7)--(5,8);
   \node[] at (.5, .5) {\tiny $-6$};
\node[] at (.5, 1.5) {\tiny $-5$};
\node[] at (.5, 2.5) {\tiny $-4$};
\node[] at (.5, 3.5) {\tiny $-3$};
\node[] at (.5, 4.5) {\tiny $-2$};
\node[] at (1.5, 4.5) {\tiny $-1$};
\node[] at (2.5, 4.5) {\tiny $0$};
\node[] at (2.5, 5.5) {\tiny $1$};
\node[] at (2.5, 6.5) {\tiny $2$};
\node[] at (3.5, 6.5) {\tiny $3$};
\node[] at (3.5, 7.5) {\tiny $4$};
\node[] at (4.5, 7.5) {\tiny $5$};
\node[] at (5.5, 7.5) {\tiny $6$};
\node[] at (6.5, 7.5) {\tiny $7$};
 \end{tikzpicture}~\\
 ~\begin{tikzpicture}[scale=.38]
 \fill[fill=red!10!white] (1,2)--(3,2)--(3,4)--(1,4)--(1,2)--cycle
			(4,6)--(6,6)--(6,7)--(4,7)--(4,6)--cycle;
 \draw [line width=.9pt, color=dredcolor] (1,2)--(3,2)--(3,4)--(1,4)--(1,2)--cycle
 		(2,2)--(2,4)
		(1,3)--(3,3)
		(5,6)--(5,7)
  		(4,6)--(6,6)--(6,7)--(4,7)--(4,6)--cycle;
    \fill[fill=cyan!10!white] (0,2)--(1,2)--(1,4)--(3,4)--(3,5)--(0,5)--(0,2)--cycle
    			(3,6)--(4,6)--(4,7)--(6,7)--(6,8)--(3,8)--(3,6)--cycle;
    \draw[thick] (0,1)--(1,1)--(1,2)--(0,2)--(0,1)--cycle	      
		      (2,5)--(3,5)--(3,7)--(2,7)--(2,5)--cycle
		      (2,6)--(3,6);
     \draw[very thick, color=frenchblue] (0,2)--(1,2)--(1,4)--(3,4)--(3,5)--(0,5)--(0,2)--cycle
     			(0,3)--(1,3)
			(0,4)--(1,4)--(1,5)
			(2,4)--(2,5)
     			(3,6)--(4,6)--(4,7)--(6,7)--(6,8)--(3,8)--(3,6)--cycle
			(3,7)--(4,7)--(4,8)
			(5,7)--(5,8);
 \node[] at (.5,  .5) {\tiny \phantom{$-6$}};
\node[] at (.5, 1.5) {\tiny $-5$};
\node[] at (.5, 2.5) {\tiny $-4$};
\node[] at (.5, 3.5) {\tiny $-3$};
\node[] at (.5, 4.5) {\tiny $-2$};
\node[] at (1.5, 4.5) {\tiny $-1$};
\node[] at (2.5, 4.5) {\tiny $0$};
\node[] at (2.5, 5.5) {\tiny $1$};
\node[] at (2.5, 6.5) {\tiny $2$};
\node[] at (3.5, 6.5) {\tiny $3$};
\node[] at (3.5, 7.5) {\tiny $4$};
\node[] at (4.5, 7.5) {\tiny $5$};
\node[] at (5.5, 7.5) {\tiny $6$};
 \end{tikzpicture}~&~\begin{tikzpicture}[scale=.38]
 \fill[fill=red!10!white] (1,2)--(3,2)--(3,4)--(1,4)--(1,2)--cycle
			(4,6)--(6,6)--(6,7)--(4,7)--(4,6)--cycle;
 \draw [line width=.9pt, color=dredcolor] (1,2)--(3,2)--(3,4)--(1,4)--(1,2)--cycle
 		(2,2)--(2,4)
		(1,3)--(3,3)
		(5,6)--(5,7)
  		(4,6)--(6,6)--(6,7)--(4,7)--(4,6)--cycle;
    \fill[fill=cyan!10!white] (0,2)--(1,2)--(1,4)--(3,4)--(3,5)--(0,5)--(0,2)--cycle
    			(3,6)--(4,6)--(4,7)--(6,7)--(6,8)--(3,8)--(3,6)--cycle;
    \draw[thick] (2,5)--(3,5)--(3,7)--(2,7)--(2,5)--cycle
		      (2,6)--(3,6);
     \draw[very thick, color=frenchblue] (0,2)--(1,2)--(1,4)--(3,4)--(3,5)--(0,5)--(0,2)--cycle
     			(0,3)--(1,3)
			(0,4)--(1,4)--(1,5)
			(2,4)--(2,5)
     			(3,6)--(4,6)--(4,7)--(6,7)--(6,8)--(3,8)--(3,6)--cycle
			(3,7)--(4,7)--(4,8)
			(5,7)--(5,8);
   \node[] at (1.5, .5) {\tiny \phantom{$-6$}};
\node[] at (.5, 2.5) {\tiny $-4$};
\node[] at (.5, 3.5) {\tiny $-3$};
\node[] at (.5, 4.5) {\tiny $-2$};
\node[] at (1.5, 4.5) {\tiny $-1$};
\node[] at (2.5, 4.5) {\tiny $0$};
\node[] at (2.5, 5.5) {\tiny $1$};
\node[] at (2.5, 6.5) {\tiny $2$};
\node[] at (3.5, 6.5) {\tiny $3$};
\node[] at (3.5, 7.5) {\tiny $4$};
\node[] at (4.5, 7.5) {\tiny $5$};
\node[] at (5.5, 7.5) {\tiny $6$};
 \end{tikzpicture}~& 
 ~\begin{tikzpicture}[scale=.38]
 \fill[fill=red!10!white] 
			(4,6)--(6,6)--(6,7)--(4,7)--(4,6)--cycle;
 \draw [line width=.9pt, color=dredcolor] 
 		(5,6)--(5,7)
  		(4,6)--(6,6)--(6,7)--(4,7)--(4,6)--cycle;
    \fill[fill=cyan!10!white] 
    			(3,6)--(4,6)--(4,7)--(6,7)--(6,8)--(3,8)--(3,6)--cycle;
    \draw[thick] (2,5)--(3,5)--(3,7)--(2,7)--(2,5)--cycle
		      (2,6)--(3,6);
     \draw[very thick, color=frenchblue] 
        			(3,6)--(4,6)--(4,7)--(6,7)--(6,8)--(3,8)--(3,6)--cycle
			(3,7)--(4,7)--(4,8)
			(5,7)--(5,8);
   \node[] at (1.5, .5) {\tiny \phantom{$-6$}};
\node[] at (2.5, 5.5) {\tiny $1$};
\node[] at (2.5, 6.5) {\tiny $2$};
\node[] at (3.5, 6.5) {\tiny $3$};
\node[] at (3.5, 7.5) {\tiny $4$};
\node[] at (4.5, 7.5) {\tiny $5$};
\node[] at (5.5, 7.5) {\tiny $6$};
 \end{tikzpicture}~&~\begin{tikzpicture}[scale=.38]
 \fill[fill=red!10!white] (1,2)--(3,2)--(3,4)--(1,4)--(1,2)--cycle
			(4,6)--(6,6)--(6,7)--(4,7)--(4,6)--cycle;
 \draw [line width=.9pt, color=dredcolor] (1,2)--(3,2)--(3,4)--(1,4)--(1,2)--cycle
 		(2,2)--(2,4)
		(1,3)--(3,3)
		(5,6)--(5,7)
  		(4,6)--(6,6)--(6,7)--(4,7)--(4,6)--cycle;
    \fill[fill=cyan!10!white] (0,2)--(1,2)--(1,4)--(3,4)--(3,5)--(0,5)--(0,2)--cycle
    			(3,6)--(4,6)--(4,7)--(6,7)--(6,8)--(3,8)--(3,6)--cycle;
    \draw[thick] (0,0)--(1,0)--(1,2)--(0,2)--(0,0)--cycle
    			(0,1)--(1,1)	      
		      (2,5)--(3,5)--(3,7)--(2,7)--(2,5)--cycle
		      (2,6)--(3,6);
     \draw[very thick, color=frenchblue] (0,2)--(1,2)--(1,4)--(3,4)--(3,5)--(0,5)--(0,2)--cycle
     			(0,3)--(1,3)
			(0,4)--(1,4)--(1,5)
			(2,4)--(2,5)
     			(3,6)--(4,6)--(4,7)--(6,7)--(6,8)--(3,8)--(3,6)--cycle
			(3,7)--(4,7)--(4,8)
			(5,7)--(5,8);
   \node[] at (.5, .5) {\tiny $-6$};
\node[] at (.5, 1.5) {\tiny $-5$};
\node[] at (.5, 2.5) {\tiny $-4$};
\node[] at (.5, 3.5) {\tiny $-3$};
\node[] at (.5, 4.5) {\tiny $-2$};
\node[] at (1.5, 4.5) {\tiny $-1$};
\node[] at (2.5, 4.5) {\tiny $0$};
\node[] at (2.5, 5.5) {\tiny $1$};
\node[] at (2.5, 6.5) {\tiny $2$};
\node[] at (3.5, 6.5) {\tiny $3$};
\node[] at (3.5, 7.5) {\tiny $4$};
\node[] at (4.5, 7.5) {\tiny $5$};
\node[] at (5.5, 7.5) {\tiny $6$};
 \end{tikzpicture}~\\ 
 ~\begin{tikzpicture}[scale=.38]
 \fill[fill=red!10!white] (1,2)--(3,2)--(3,4)--(1,4)--(1,2)--cycle;
 \draw [line width=.9pt, color=dredcolor] (1,2)--(3,2)--(3,4)--(1,4)--(1,2)--cycle
 		(2,2)--(2,4)
		(1,3)--(3,3);
    \fill[fill=cyan!10!white] (0,2)--(1,2)--(1,4)--(3,4)--(3,5)--(0,5)--(0,2)--cycle;
    \draw[thick] (0,1)--(1,1)--(1,2)--(0,2)--(0,1)--cycle	      
		      (2,5)--(3,5)--(3,7)--(2,7)--(2,5)--cycle
		      (2,6)--(3,6);
     \draw[very thick, color=frenchblue] (0,2)--(1,2)--(1,4)--(3,4)--(3,5)--(0,5)--(0,2)--cycle
     			(0,3)--(1,3)
			(0,4)--(1,4)--(1,5)
			(2,4)--(2,5);
   \node[] at (.5, .5) {\tiny \phantom{$-6$}};
\node[] at (.5, 1.5) {\tiny $-5$};
\node[] at (.5, 2.5) {\tiny $-4$};
\node[] at (.5, 3.5) {\tiny $-3$};
\node[] at (.5, 4.5) {\tiny $-2$};
\node[] at (1.5, 4.5) {\tiny $-1$};
\node[] at (2.5, 4.5) {\tiny $0$};
\node[] at (2.5, 5.5) {\tiny $1$};
\node[] at (2.5, 6.5) {\tiny $2$};
 \end{tikzpicture}~&~\begin{tikzpicture}[scale=.38]
  \fill[fill=red!10!white] (1,2)--(3,2)--(3,4)--(1,4)--(1,2)--cycle;
 \draw [line width=.9pt, color=dredcolor] (1,2)--(3,2)--(3,4)--(1,4)--(1,2)--cycle
 		(2,2)--(2,4)
		(1,3)--(3,3);
    \fill[fill=cyan!10!white] (0,2)--(1,2)--(1,4)--(3,4)--(3,5)--(0,5)--(0,2)--cycle;
    \draw[thick] 	      
		      (2,5)--(3,5)--(3,7)--(2,7)--(2,5)--cycle
		      (2,6)--(3,6);
     \draw[very thick, color=frenchblue] (0,2)--(1,2)--(1,4)--(3,4)--(3,5)--(0,5)--(0,2)--cycle
     			(0,3)--(1,3)
			(0,4)--(1,4)--(1,5)
			(2,4)--(2,5);
   \node[] at (.5, .5) {\tiny \phantom{$-6$}};
\node[] at (.5, 2.5) {\tiny $-4$};
\node[] at (.5, 3.5) {\tiny $-3$};
\node[] at (.5, 4.5) {\tiny $-2$};
\node[] at (1.5, 4.5) {\tiny $-1$};
\node[] at (2.5, 4.5) {\tiny $0$};
\node[] at (2.5, 5.5) {\tiny $1$};
\node[] at (2.5, 6.5) {\tiny $2$};
 \end{tikzpicture}~&~\begin{tikzpicture}[scale=.38]
    \draw[thick]       
		      (2,5)--(3,5)--(3,7)--(2,7)--(2,5)--cycle
		      (2,6)--(3,6);
   \node[] at (2.5, .5) {\tiny\phantom{ $-6$}};
\node[] at (2.5, 5.5) {\tiny $1$};
\node[] at (2.5, 6.5) {\tiny $2$};
 \end{tikzpicture}~&~\begin{tikzpicture}[scale=.38]
  \fill[fill=red!10!white] (1,2)--(3,2)--(3,4)--(1,4)--(1,2)--cycle;
 \draw [line width=.9pt, color=dredcolor] (1,2)--(3,2)--(3,4)--(1,4)--(1,2)--cycle
 		(2,2)--(2,4)
		(1,3)--(3,3);
    \fill[fill=cyan!10!white] (0,2)--(1,2)--(1,4)--(3,4)--(3,5)--(0,5)--(0,2)--cycle;
    \draw[thick] (0,0)--(1,0)--(1,2)--(0,2)--(0,0)--cycle
    			(0,1)--(1,1)	      
		      (2,5)--(3,5)--(3,7)--(2,7)--(2,5)--cycle
		      (2,6)--(3,6);
     \draw[very thick, color=frenchblue] (0,2)--(1,2)--(1,4)--(3,4)--(3,5)--(0,5)--(0,2)--cycle
     			(0,3)--(1,3)
			(0,4)--(1,4)--(1,5)
			(2,4)--(2,5);
   \node[] at (.5, .5) {\tiny $-6$};
\node[] at (.5, 1.5) {\tiny $-5$};
\node[] at (.5, 2.5) {\tiny $-4$};
\node[] at (.5, 3.5) {\tiny $-3$};
\node[] at (.5, 4.5) {\tiny $-2$};
\node[] at (1.5, 4.5) {\tiny $-1$};
\node[] at (2.5, 4.5) {\tiny $0$};
\node[] at (2.5, 5.5) {\tiny $1$};
\node[] at (2.5, 6.5) {\tiny $2$};
 \end{tikzpicture}~\\
 ~\begin{tikzpicture}[scale=.38]
  \fill[fill=red!10!white] (1,2)--(3,2)--(3,4)--(1,4)--(1,2)--cycle;
 \draw [line width=.9pt, color=dredcolor] (1,2)--(3,2)--(3,4)--(1,4)--(1,2)--cycle
 		(2,2)--(2,4)
		(1,3)--(3,3);
    \fill[fill=cyan!10!white] (0,2)--(1,2)--(1,4)--(3,4)--(3,5)--(0,5)--(0,2)--cycle;
    \draw[thick] (0,1)--(1,1)--(1,2)--(0,2)--(0,1)--cycle
    			(0,1)--(1,1);      
     \draw[very thick, color=frenchblue] (0,2)--(1,2)--(1,4)--(3,4)--(3,5)--(0,5)--(0,2)--cycle
     			(0,3)--(1,3)
			(0,4)--(1,4)--(1,5)
			(2,4)--(2,5);
   \node[] at (.5, .5) {\tiny \phantom{$-6$}};
\node[] at (.5, 1.5) {\tiny $-5$};
\node[] at (.5, 2.5) {\tiny $-4$};
\node[] at (.5, 3.5) {\tiny $-3$};
\node[] at (.5, 4.5) {\tiny $-2$};
\node[] at (1.5, 4.5) {\tiny $-1$};
\node[] at (2.5, 4.5) {\tiny $0$};
\node[] at (2.5, 5.5) {\tiny \phantom{$1$}};
 \end{tikzpicture}~& ~\begin{tikzpicture}[scale=.38]
  \fill[fill=red!10!white] (1,2)--(3,2)--(3,4)--(1,4)--(1,2)--cycle;
 \draw [line width=.9pt, color=dredcolor] (1,2)--(3,2)--(3,4)--(1,4)--(1,2)--cycle
 		(2,2)--(2,4)
		(1,3)--(3,3);
    \fill[fill=cyan!10!white] (0,2)--(1,2)--(1,4)--(3,4)--(3,5)--(0,5)--(0,2)--cycle;
     \draw[very thick, color=frenchblue] (0,2)--(1,2)--(1,4)--(3,4)--(3,5)--(0,5)--(0,2)--cycle
     			(0,3)--(1,3)
			(0,4)--(1,4)--(1,5)
			(2,4)--(2,5);
   \node[] at (.5, .5) {\tiny \phantom{$-6$}};
\node[] at (.5, 2.5) {\tiny $-4$};
\node[] at (.5, 3.5) {\tiny $-3$};
\node[] at (.5, 4.5) {\tiny $-2$};
\node[] at (1.5, 4.5) {\tiny $-1$};
\node[] at (2.5, 4.5) {\tiny $0$};
 \end{tikzpicture}~&
 ~\begin{tikzpicture}[scale=.38]
 \node[] at (.5, .5) {\tiny \phantom{$-6$}};
  \node[] at (.5, 3.5) {$\emptyset$};
   \end{tikzpicture}~& ~\begin{tikzpicture}[scale=.38]
    \fill[fill=red!10!white] (1,2)--(3,2)--(3,4)--(1,4)--(1,2)--cycle;
 \draw [line width=.9pt, color=dredcolor] (1,2)--(3,2)--(3,4)--(1,4)--(1,2)--cycle
 		(2,2)--(2,4)
		(1,3)--(3,3);
    \fill[fill=cyan!10!white] (0,2)--(1,2)--(1,4)--(3,4)--(3,5)--(0,5)--(0,2)--cycle;
    \draw[thick] (0,0)--(1,0)--(1,2)--(0,2)--(0,0)--cycle
    			(0,1)--(1,1);      
     \draw[very thick, color=frenchblue] (0,2)--(1,2)--(1,4)--(3,4)--(3,5)--(0,5)--(0,2)--cycle
     			(0,3)--(1,3)
			(0,4)--(1,4)--(1,5)
			(2,4)--(2,5);
   \node[] at (.5, .5) {\tiny $-6$};
\node[] at (.5, 1.5) {\tiny $-5$};
\node[] at (.5, 2.5) {\tiny $-4$};
\node[] at (.5, 3.5) {\tiny $-3$};
\node[] at (.5, 4.5) {\tiny $-2$};
\node[] at (1.5, 4.5) {\tiny $-1$};
\node[] at (2.5, 4.5) {\tiny $0$};
 \end{tikzpicture}~
\end{bmatrix}
\]  
\caption{An illustrate of the determinant $\det \left(
    \ts_{(\gamma\oplus\nl)[p_j,q_i]}\right)_{i,j=1}^k$. The $(i,j)$ entry only
  shows the skew shape $(\gamma\oplus\nl)[p_j,q_i]$, where the contents of the
  cells coming from $\gamma$ are shown and the cells coming from $\nl$ are
  colored red.}
  \label{fig:det3}
\end{figure}
 \end{exam}

 If $\nu=\lambda$ in Theorem~\ref{thm:main_HG}, then we obtain the following
 Hamel--Goulden formula for Macdonald's 9th variation of Schur functions, which
 was also proved by Bachmann and Charlton \cite{Bachmann_2020}, and Foley and
 King \cite{Foley20} using the Lindstr\"om--Gessel--Viennot lemma.

\begin{cor}\label{cor:HG}
  Let $\lambda$ and $\mu$ be partitions with $\mu\subseteq \lambda$. Suppose
  that $\gamma$ is a border strip with $\Cont(\lambda)\subseteq\Cont(\gamma)$
  and $\theta=(\theta_1,\dots,\theta_k)$ is the decomposition of $\lm$ with
  cutting strip $\gamma$. Then
\[
\ts_{\lambda/\mu} = \det \left( \ts_{\gamma[p_j,q_i]}\right)_{i,j=1}^k. 
\]
\end{cor}

If $\nu/\lambda$ is a disjoint union of single cells and if we restrict
Theorem~\ref{thm:main_HG} to Schur functions, then we obtain Jin's result
\cite[Theorem~2]{Jin_2018} for the cases when the ``enriched diagrams'' are not necessary.

For sequences $\vec a = (a_1,\dots,a_r)$ and $\vec b = (b_1,\dots,b_r)$ with
$a_1>\dots>a_r\ge0$ and $b_1>\dots>b_r\ge0$, the \emph{Frobenius notation}
$(\vec a|\vec b)$ denotes the partition
\[
\{(i,i): 1\le i\le r\} \cup 
\{(i,j): 1\le i\le r, i<j\le a_i\} \cup 
\{(i,j): 1\le j\le r, j<i\le b_i\}.
\]
See Figure~\ref{fig:Frobenius_ex}.

\begin{figure}
  \centering
\begin{tikzpicture}[scale=.48]
\fill[fill=red!10!white] (1.1, 3.1) rectangle (4.9, 3.9);
\fill[fill=red!10!white] (2.1, 2.1) rectangle (3.9, 2.9);
\fill[fill=red!10!white] (3.1, 1.1) rectangle (3.9, 1.9);
\fill[fill=blue!10!white] (.1, .1) rectangle (.9, 2.9);
\fill[fill=blue!10!white] (1.1, 1.1) rectangle (1.9, 1.9);
\draw [thick] (0,0)--(0,4)--(5,4)--(5,3)--(4,3)--(4,1)--(1,1)--(1,0)--(0,0)--cycle;
\draw (0,1)--(1,1);
\draw (0,2)--(4,2);
\draw (0,3)--(4,3);
\draw (1,1)--(1,4);
\draw (2,1)--(2,4);
\draw (3,1)--(3,4);
\draw (4,3)--(4,4);
\draw[thick, color=dredcolor] (1.1,3.1)--(4.9, 3.1)--(4.9, 3.9)--(1.1, 3.9)--(1.1, 3.1)--cycle;
\draw[thick, color=dredcolor] (2.1,2.1)--(3.9, 2.1)--(3.9, 2.9)--(2.1, 2.9)--(2.1, 2.1)--cycle;
\draw[thick, color=dredcolor] (3.1,1.1)--(3.9, 1.1)--(3.9, 1.9)--(3.1, 1.9)--(3.1, 1.1)--cycle;
\draw[thick, color=dbluecolor] (.1,.1)--(.1, 2.9)--(.9, 2.9)--(.9, .1)--(.1, .1)--cycle;
\draw[thick, color=dbluecolor] (1.1,1.1)--(1.1, 1.9)--(1.9, 1.9)--(1.9, 1.1)--(1.1, 1.1)--cycle;
\draw [dashed, color=gray] (0,4)--(3,1);
\end{tikzpicture}
\caption{An illustrate of the Frobenius notation $(4,2,1|3,1,0)$ for the
  partition $\lambda=(5,4,4,1)$.}\label{fig:Frobenius_ex}
\end{figure}
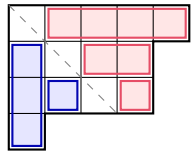

If we take $\nu=(\vec a\sqcup \vec c| \vec b\sqcup \vec d)$, $\lambda=(\vec a|
\vec b)$, $\mu=(\vec c|\vec d)$, and $\theta$ to be the Kreiman decomposition of
$\lm$ in Theorem~\ref{thm:main_HG}, then we obtain the following generalized
Giambelli formula.

\begin{cor}\label{cor:Giam}
  Let $\vec a = (a_1,\dots,a_r)$, $\vec b = (b_1,\dots,b_r)$,
  $\vec c = (c_1,\dots,c_s)$, and $\vec d = (d_1,\dots,d_s)$ be sequences of
  integers such that $r,s\ge0$ and
\[
a_1>\dots>a_r>c_1>\dots>c_s\ge0, \qquad
b_1>\dots>b_r>d_1>\dots>d_s\ge0.
\]
Then
\[
  \ts_{(\vec c|\vec d)}^{r-1} \ts_{(\vec a\sqcup \vec c| \vec b\sqcup \vec d)}
  =\det\left( \ts_{(a_i\sqcup\vec c| b_j\sqcup\vec d)} \right) _{1\le i,j\le r}.
\]
\end{cor}

If $\vec c = \vec d=\emptyset$, then Corollary~\ref{cor:Giam} reduces to the
Giambelli formula. If $\vec c = \vec d=(0)$ in Corollary~\ref{cor:Giam}, we
obtain
\[
\ts_{(1)}^{r-1} \ts_{(\vec a\sqcup 0|\vec b\sqcup 0)}
  =\det\left( \ts_{((a_i,0)| (b_j,0))} \right) _{1\le i,j\le r},
\]
where each entry has a \emph{near hook shape} $((a_i,0)| (b_j,0))$.

For the rest of this section we give a proof of Theorem~\ref{thm:main_HG}. We
first recall a known property of $\ts_{\lm}$.

Let $\alpha$ and $\beta$ be skew shapes and let $a$ and $b$ be, respectively,
the top-right corner of $\alpha$ and the bottom-left corner of $\beta$. We
define $\alpha\rightarrow \beta$ (resp.~$\alpha\uparrow \beta$) to be the skew
shape obtained from $\alpha$ by attaching $\beta$ so that $b$ is to the right
of $a$ (resp.~above $a$).

\begin{lem} \label{lem:combine}
  \cite[Chapter I, \S5, Example~30 (d)]{Macdonald}
Let $\alpha$ and $\beta$ be skew shapes such that the top-right corner $a$ of $\alpha$
and the bottom-left corner $b$ of $\beta$ satisfy $c(b)=c(a)+1$. Then
\[
\ts_\alpha \ts_\beta = \ts_{\alpha\rightarrow \beta} + \ts_{\alpha\uparrow \beta}.
\]
\end{lem}

Chen, Yan, and Yang \cite{Chen_2005} showed that the Hamel--Goulden formula can
be obtained from the Lascoux--Pragacz formula using simple matrix operations.
Their proof uses only the fact that the Schur functions satisfy
\[
s_\alpha s_\beta = s_{\alpha\rightarrow \beta} + s_{\alpha\uparrow \beta}.
\]
Hence, by Lemma~\ref{lem:combine}, their proof extends to the identity in
Theorem~\ref{thm:main_HG}. This implies that it is sufficient to prove this
theorem for the case when $\theta$ is the Lascoux--Pragacz decomposition of
$\lm$. To this end we need three lemmas.

\begin{lem}\label{lem:9th}
  The following properties hold.
  \begin{enumerate}
  \item If $\mu\not\subseteq\lambda$, then $\ts_{\lm}=0$.
  \item If $\mu\subseteq\lambda$ and $\alpha_1,\dots,\alpha_s$ are the connected
    components of $\lm$, then $\ts_{\lm}=\ts_{\alpha_1}\cdots \ts_{\alpha_s}$.
  \end{enumerate}
\end{lem}
\begin{proof}
  Both properties can be proved similarly as in the Schur function case
  \cite[Chapter I, \S5, (5.7)]{Macdonald}.
\end{proof}

\begin{lem}\label{lem:LP1}
  Suppose that $\mu\subseteq \lambda\subseteq\nu$ are partitions such that $\mu$
  is $\nl$-compatible. Let $\alpha_1,\dots,\alpha_\ell$ be the connected
  components of $\nl$ and let $\gamma$ be the outer strip of $\lambda$. Finally,
  let $\theta=(\theta_1,\dots,\theta_k)$ be the Lascoux--Pragacz decomposition
  of $\lm$. Then
\[
  \ts_{\nu/\lambda(q_i,p_j-1)}=\ts_{(\gamma\oplus\nl)[p_j,q_i]} \prod_{s=1}^{\ell} 
\ts_{\alpha_s}^{\chi(q_i<\min(\Cont(\alpha_s)))+\chi(p_j>\max(\Cont(\alpha_s)))}.
\]
\end{lem}

\begin{proof}
  By Lemma~\ref{lem:C/C}, we have $p_j-1\ne q_i$ for all $i,j$. Thus there are
  two cases $p_j\le q_i$ and $p_j-1>q_i$.

  \textbf{Case 1:} $p_j\le q_i$. Since
  $\lambda(q_i,p_j-1)=\lambda\setminus\lambda^0[p_j,q_i]$, the connected component
  $\alpha_s$ of $\nl$ remains the same in $\nu/\lambda(q_i,p_j-1)$ if
  $q_i<\min(\Cont(\alpha_s))$ or $p_j>\max(\Cont(\alpha_s))$. Otherwise,
  $\Cont(\alpha_s)\subseteq [p_j,q_i]$ and therefore $\alpha_s$ is contained in
  $(\gamma\oplus\nl)[p_j,q_i]$, which is a connected component of
  $\nu/\lambda(q_i,p_j-1)$. Thus by Lemma~\ref{lem:9th},
  \[
  \ts_{\nu/\lambda(q_i,p_j-1)}=\ts_{(\gamma\oplus\nl)[p_j,q_i]} \prod_{s=1}^{\ell} 
\ts_{\alpha_s}^{\chi(q_i<\min(\Cont(\alpha_s)) \mbox{\scriptsize{ or }} p_j>\max(\Cont(\alpha_s)))}.
\]
Since $q_i<\min(\Cont(\alpha_s))$ and $p_j>\max(\Cont(\alpha_s))$ cannot be satisfied 
at the same time, the above equation is the same as the one in the lemma.

\textbf{Case 2:} $p_j-1>q_i$. Since $\ts_{(\gamma\oplus\nl)[p_j,q_i]}=0$, it
suffices to show that $\nu\not\subseteq\lambda(q_i,p_j-1)$, which implies
$\ts_{\nu/\lambda(q_i,p_j-1)}=0$ by Lemma~\ref{lem:9th}. Since
$\lambda(q_i,p_j-1)=\lambda\cup\lambda^+[q_i+1,p_j-1]$ has a cell with content
$p_j-1$, to show $\nu\not\subseteq\lambda(q_i,p_j-1)$ it is enough to show that
$p_j-1\not\in \Cont(\nl)$.

For a contradiction suppose $p_j-1\in \Cont(\nl)$. Then
$p_j-1\in\Cont(\alpha_s)$ for some $s$. Since $\mu$ is $\nl$-compatible and
$p_j-1\in\Cont(\alpha_s)$, we must have $\lambda^+[p_j-1,p_j] =
\mu^+[p_j-1,p_j]$. On the other hand, we have $p_j-1\in C_n(\mu)\setminus
C_n(\lambda)$ by Lemma~\ref{lem:C/C}. Since $p_j-1\in C_n(\mu)$, the border
strip $\mu^+[p_j-1,p_j]$ must be a vertical domino. However, since $p_j-1\not\in
C_n(\lambda)$, the border strip $\lambda^+[p_j-1,p_j]$ must be a horizontal
domino. Then $\lambda^+[p_j-1,p_j] \ne \mu^+[p_j-1,p_j]$, which is a
contradiction. Therefore we must have $p_j-1\not\in \Cont(\nl)$, which completes
the proof.
\end{proof}

\begin{lem}\label{lem:LP2}
  Under the same assumptions in Lemma~\ref{lem:LP1}, we have
\[
\ts_{\nu/\mu} \prod_{s=1}^\ell \ts_{\alpha_s}^{k-1-I_s-J_s} 
    = \det \left( \ts_{(\gamma\oplus\nl)[p_j,q_i]}\right)_{i,j=1}^k,
\]
where $I_s$ is the number of $1\le i\le k$ such that
$q_i<\min(\Cont(\alpha_s))$
and $J_s$ is the number of $1\le j\le k$ such that
$p_j>\max(\Cont(\alpha_s))$.
\end{lem}
\begin{proof}
  By \eqref{eq:main_LP2}, we have
\[
     \ts_{\nu/\mu} \ts_{\nu/\lambda}^{k-1}
    = \det \left( (-1)^{\chi(p_j>q_i)} \ts_{\nu/\lambda(q_i,p_j-1)}\right)_{i,j=1}^k.
\]
By Lemmas~\ref{lem:9th} and \ref{lem:LP1}, the above
equation can be written as
\begin{equation}\label{eq:ss}
     \ts_{\nu/\mu} \prod_{s=1}^\ell \ts_{\alpha_s}^{k-1}
    = \det \left( 
\ts_{(\gamma\oplus\nl)[p_j,q_i]} \prod_{s=1}^{\ell} 
\ts_{\alpha_s}^{\chi(q_i<\min(\Cont(\alpha_s)))+\chi(p_j>\max(\Cont(\alpha_s)))}
\right)_{i,j=1}^k,
\end{equation}
where the factor $(-1)^{\chi(p_j>q_i)}$ can be omitted because
$\ts_{(\gamma\oplus\nl)[p_j,q_i]}=0$ if $p_j>q_i$. By factoring out the factor
$\prod_{s=1}^{\ell} \ts_{\alpha_s}^{\chi(q_i<\min(\Cont(\alpha_s)))}$ from each
row $i$ and the factor $\prod_{s=1}^{\ell}
\ts_{\alpha_s}^{\chi(p_j>\max(\Cont(\alpha_s)))}$ from each column $j$ and
dividing both sides of \eqref{eq:ss} by these factors we obtain the desired
identity.
\end{proof}

Now we are ready to prove Theorem~\ref{thm:main_HG}.

\begin{proof}[Proof of Theorem~\ref{thm:main_HG}]
  As we have already discussed, it suffices to show the theorem for the case
  when $\theta$ is the Lascoux--Pragacz decomposition of $\lm$. Then, by
  Lemma~\ref{lem:LP2}, it suffices to show that $k-I_s-J_s=r_s$ for all $s$. By
  definition, $k-I_s-J_s$ is the number of border strips $\theta_i$ such that
  $p(\theta_i)\le \max(\Cont(\alpha_s))$ and $\min(\Cont(\alpha_s))\le
  q(\theta_i)$. Since $\mu$ is $\nl$-compatible and $\theta$ is the
  Lascoux--Pragacz decomposition, every border strip $\theta_i$ with
  $\Cont(\alpha_s)\cap\Cont(\theta_i)\ne\emptyset$ must satisfy
  $\Cont(\alpha_s)\subseteq\Cont(\theta_i)$. This implies that $p(\theta_i)\le
  \max(\Cont(\alpha_s))$ and $\min(\Cont(\alpha_s))\le q(\theta_i)$ if and only
  if $\Cont(\alpha_s)\subseteq\Cont(\theta_i)$. Therefore the number of such
  border strips $\theta_i$ is equal to $r_s$, which completes the proof.
\end{proof}

\section{A generalization of a converse of Hamel--Goulden's theorem}
\label{sec:conv-hamel-gould}

In this section we prove Theorem~\ref{thm:HG3}, which is a generalization of a
converse of Hamel--Goulden's theorem. We restate Theorem~\ref{thm:HG3} as
follows.

\begin{thm}\label{thm:HG_conv}
  Let $\alpha$ be any connected skew shape. Suppose that $(a_1,\dots,a_k)$ and
  $(b_1,\dots,b_k)$ are sequences of integers such that $\alpha[a_j,b_i]$ is a
  skew shape for all $i,j$. Then either $\det \left(
    s_{\alpha[a_j,b_i]}\right)_{i,j=1}^k = 0$ or there exists a skew shape
  $\rho$ such that
\[
\det \left( s_{\alpha[a_j,b_i]}\right)_{i,j=1}^k = \pm s_{\rho}.
\]
\end{thm}

To prove this theorem we need the following lemma.

\begin{lem}\label{lem:gamma_to_la}
  Let $\gamma$ be a border strip and let $(a_1<\dots<a_k)$ and $(b_1<\dots<b_k)$
  be sequences of integers such that $a_i,b_i\in \Cont(\gamma)$ for all $i$, and
  for any integer $a$,
\begin{equation}\label{eq:dyck}
|\{i: a_i\le a\}| \ge |\{i: b_i\le a\}|.
\end{equation}
Then there is a skew shape $\tau$ such that if $\theta=(\theta_1,\dots,\theta_r)$
is the decomposition of $\tau$ with cutting strip $\gamma$, then $r=k$ and
$\{p(\theta_1),\dots,p(\theta_k)\}=\{a_1,\dots,a_k\}$ and
$\{q(\theta_1),\dots,q(\theta_k)\}=\{b_1,\dots,b_k\}$.
\end{lem}
\begin{proof}
  We proceed by induction on $k$. If $k=0$, we can take $\tau=\emptyset$. Let
  $k\ge1$ and suppose that the lemma is true for $k-1$.

  We consider the following two cases depending on the shape of
  $\gamma[b_1,b_1+1]$.

  \begin{description}
  \item[Case 1] $\gamma[b_1,b_1+1]$ is a horizontal domino. Let
    $\theta_0=\gamma[a_1,b_1]$. It is easy to see that the two sequences
    $(a_2<\dots<a_k)$ and $(b_2<\dots<b_k)$ satisfy \eqref{eq:dyck}. Thus, by
    the induction hypothesis, there is a skew shape $\sigma$ whose decomposition
    with cutting strip $\gamma$ is $(\theta_1,\dots,\theta_{k-1})$ such that
    $\{p(\theta_1),\dots,p(\theta_{k-1})\}=\{a_2,\dots,a_k\}$ and
    $\{q(\theta_1),\dots,q(\theta_{k-1})\}=\{b_2,\dots,b_k\}$. The fact that
    $a_1<a_2$ and $\gamma[b_1,b_1+1]$ is a horizontal domino guarantees that
    $\tau=\sigma\cup \theta_0$, where $\theta_0$ is attached below $\sigma$
    after an appropriate diagonal shift, is a skew shape satisfying the desired
    properties.
  \item[Case 2] $\gamma[b_1,b_1+1]$ is a vertical domino. Let
    $\theta_0=\gamma[a_m,b_1]$, where $m$ is the largest integer such that
    $a_m\le b_1$. It is easy to see that the two sequences
    $(a_1<\dots<a_{m-1}<a_{m+1}<\dots<a_k)$ and $(b_2<\dots<b_k)$ satisfy
    \eqref{eq:dyck}. Thus, by the induction hypothesis, there is a skew shape
    $\sigma$ whose decomposition with cutting strip $\gamma$ is
    $(\theta_1,\dots,\theta_{k-1})$ such that
    $\{p(\theta_1),\dots,p(\theta_{k-1})\}=\{a_1,\dots,a_{m-1},a_{m+1},\dots,a_k\}$
    and $\{q(\theta_1),\dots,q(\theta_{k-1})\}=\{b_2,\dots,b_k\}$. The fact that
    $a_i<a_m$, for all $1\le i<m$, and $\gamma[b_1,b_1+1]$ is a vertical domino
    guarantees that $\tau=\sigma\cup \theta_0$, where $\theta_0$ is attached
    above $\sigma$ after an appropriate diagonal shift, is a skew shape
    satisfying the desired properties.
  \end{description}

  The above two cases show that the lemma is true for $k$ and the proof is
  completed by induction.
\end{proof}

For an illustration of the construction in the proof of
Lemma~\ref{lem:gamma_to_la}, consider the border strip $\gamma$ and the two
sequences $(a_1,a_2,a_3)$ and $(b_1,b_2,b_3)$ in Figure~\ref{fig:to_skew}. Since
$\gamma[b_1,b_1+1]$ is a horizontal domino, we construct the border strip
$\gamma[a_1,b_1]$ as in Figure~\ref{fig:to_skew1} (a). Since $\gamma[b_2,b_2+1]$
is a vertical domino, we construct the border strip $\gamma[a_3,b_2]$ as in
Figure~\ref{fig:to_skew1} (b). Finally we construct the border strip
$\gamma[a_2,b_3]$ as in Figure~\ref{fig:to_skew1} (c). By combining these border
strips after appropriate diagonal shifts we obtain the skew shape shown in
Figure~\ref{fig:to_skew}.

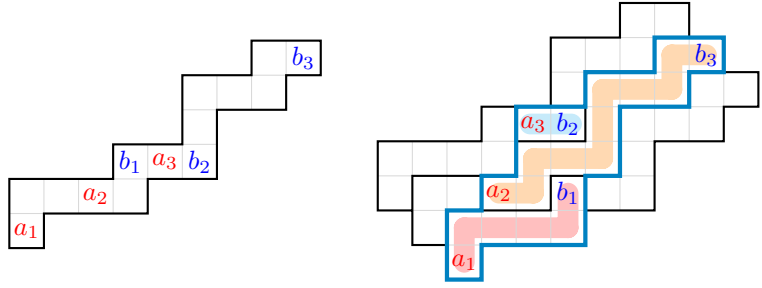
\begin{figure}
  \centering
  \begin{tikzpicture}[scale=.46]
  \draw[color=gray!30!white]	(0,1)--(1,1)--(1,2)
		(2,1)--(2,2)
		(3,1)--(3,2)--(4,2)--(4,3)
		(5,2)--(5,3)--(6,3)
		(5,4)--(6,4)--(6,5)
		(7,4)--(7,5)--(8,5)--(8,6);
   \draw[thick] (0,0)--(1,0)--(1,1)--(4,1)--(4,2)--(6,2)--(6,4)--(8,4)--(8,5)--(9,5)--(9,6)--(7,6)--(7,5)--(5,5)--(5,3)--(3,3)--(3,2)--(0,2)--(0,0)--cycle;
 \node[] at (.5, .5) {\color{red}$a_1$};
  \node[] at (2.5, 1.5) {\color{red}$a_2$};
   \node[] at (4.5, 2.5) {\color{red}$a_3$};
    \node[] at (3.5, 2.5) {\color{blue}$b_1$};
  \node[] at (5.5, 2.5) {\color{blue}$b_2$};
   \node[] at (8.5, 5.5) {\color{blue}$b_3$};
      \node[] at (4.5, -.5) {\phantom{\color{red}$a_3$}};
  \end{tikzpicture}\qquad
  \begin{tikzpicture}[scale=.46]  
  \fill[fill=pink, rounded corners] (2.2, .2) rectangle (2.8, 1.8)
  			(2.2, 1.2) rectangle (5.8, 1.8)
			(5.2, 1.2) rectangle (5.8, 2.8);
  \fill[fill=orange!30!white, rounded corners] (3.2, 2.2) rectangle (4.8, 2.8)
  			(4.2, 2.2) rectangle (4.8, 3.8)
			(4.2, 3.2) rectangle (6.8, 3.8)
			(6.2, 3.2) rectangle (6.8, 5.8)
			(6.2, 5.2) rectangle (8.8, 5.8)
			(8.2, 5.2) rectangle (8.8, 6.8)
			(8.2, 6.2) rectangle (9.8, 6.8);
  \fill[fill=cyan!20!white, rounded corners] (4.1, 4.2) rectangle (5.9, 4.8);
  \draw[color=gray!30!white] (7,7)--(8,7)
  		(5,6)--(10,6)
		(5,5)--(10,5)
		(3,4)--(8,4)
		(0,3)--(8,3)
		(1,2)--(6,2)
		(2,1)--(3,1)
		(1,3)--(1,4)
		(2,2)--(2,4)
		(3,1)--(3,4)
		(4,1)--(4,5)
		(5,1)--(5,5)
		(6,1)--(6,7)
		(7,2)--(7,7)
		(8,4)--(8,8)
		(9,4)--(9,7)
		(10,5)--(10,6);
  \draw[ thick] (2,1)--(1,1)--(1,3)--(3,3)
  		(1,2)--(0,2)--(0,4)--(3,4)--(3,5)--(4,5)
		(5,5)--(5,7)--(7,7)--(7,8)--(9,8)--(9,7)
		(6,2)--(8,2)--(8,4)--(10,4)--(10,5)--(11,5)--(11,6)--(10,6)
		(3,2)--(5,2)--(5,3)--(6,3)
		(4,4)--(6,4)--(6,5);
  \draw[line width=1.6pt, color=frenchblue!90!cyan] (2,0)--(3,0)--(3,1)--(6,1)--(6,3)--(7,3)--(7,5)--(9,5)--(9,6)--(10,6)--(10,7)--(8,7)--(8,6)--(6,6)--(6,5)--(4,5)--(4,3)--(3,3)--(3,2)--(2,2)--(2,0)--cycle;
   \node[] at (2.5, .5) {\color{red}$a_1$};
  \node[] at (3.5, 2.5) {\color{red}$a_2$};
   \node[] at (4.5, 4.5) {\color{red}$a_3$};
    \node[] at (5.5, 2.5) {\color{blue}$b_1$};
  \node[] at (5.5, 4.5) {\color{blue}$b_2$};
   \node[] at (9.5, 6.5) {\color{blue}$b_3$};
    \end{tikzpicture}
  \caption{A border strip $\gamma$ and the two sequences $(a_1,a_2,a_3)$ and
    $(b_1,b_2,b_3)$ on the left and the corresponding skew shape and its
    decomposition on the right.}
  \label{fig:to_skew}
\end{figure}

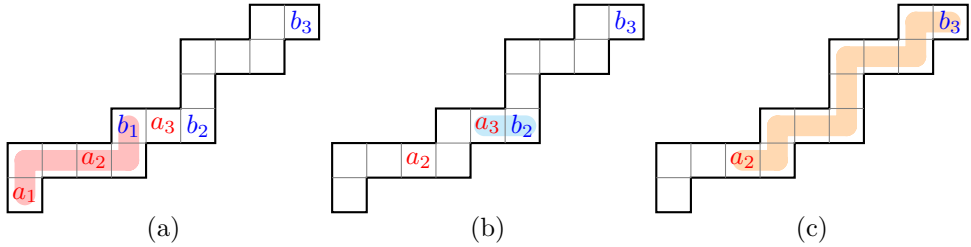
\begin{figure}
  \centering
   \begin{tikzpicture}[scale=.46]
    \fill[fill=pink, rounded corners] (.2, .2) rectangle (.8, 1.8)
  			(.2, 1.2) rectangle (3.8, 1.8)
			(3.2, 1.2) rectangle (3.8, 2.8);
  \draw[thick] (0,0)--(1,0)--(1,1)--(4,1)--(4,2)--(6,2)--(6,4)--(8,4)--(8,5)--(9,5)--(9,6)--(7,6)--(7,5)--(5,5)--(5,3)--(3,3)--(3,2)--(0,2)--(0,0)--cycle;
  \draw[color=gray]	(0,1)--(1,1)--(1,2)
		(2,1)--(2,2)
		(3,1)--(3,2)--(4,2)--(4,3)
		(5,2)--(5,3)--(6,3)
		(5,4)--(6,4)--(6,5)
		(7,4)--(7,5)--(8,5)--(8,6);
 \node[] at (.5, .5) {\color{red}$a_1$};
  \node[] at (2.5, 1.5) {\color{red}$a_2$};
   \node[] at (4.5, 2.5) {\color{red}$a_3$};
    \node[] at (3.5, 2.5) {\color{blue}$b_1$};
  \node[] at (5.5, 2.5) {\color{blue}$b_2$};
   \node[] at (8.5, 5.5) {\color{blue}$b_3$};
      \node[] at (4.5, -.5) {(a)};
  \end{tikzpicture}~
   \begin{tikzpicture}[scale=.46]
     \fill[fill=cyan!20!white, rounded corners] (4.1, 2.2) rectangle (5.9, 2.8);
  \draw[thick] (0,0)--(1,0)--(1,1)--(4,1)--(4,2)--(6,2)--(6,4)--(8,4)--(8,5)--(9,5)--(9,6)--(7,6)--(7,5)--(5,5)--(5,3)--(3,3)--(3,2)--(0,2)--(0,0)--cycle;
  \draw[color=gray]	(0,1)--(1,1)--(1,2)
		(2,1)--(2,2)
		(3,1)--(3,2)--(4,2)--(4,3)
		(5,2)--(5,3)--(6,3)
		(5,4)--(6,4)--(6,5)
		(7,4)--(7,5)--(8,5)--(8,6);
  \node[] at (2.5, 1.5) {\color{red}$a_2$};
   \node[] at (4.5, 2.5) {\color{red}$a_3$};
  \node[] at (5.5, 2.5) {\color{blue}$b_2$};
   \node[] at (8.5, 5.5) {\color{blue}$b_3$};
       \node[] at (4.5, -.5) {(b)};
  \end{tikzpicture}~
   \begin{tikzpicture}[scale=.46]
     \fill[fill=orange!30!white, rounded corners] (2.2, 1.2) rectangle (3.8, 1.8)
  			(3.2, 1.2) rectangle (3.8, 2.8)
			(3.2, 2.2) rectangle (5.8, 2.8)
			(5.2, 2.2) rectangle (5.8, 4.8)
			(5.2, 4.2) rectangle (7.8, 4.8)
			(7.2, 4.2) rectangle (7.8, 5.8)
			(7.2, 5.2) rectangle (8.8, 5.8);
  \draw[thick] (0,0)--(1,0)--(1,1)--(4,1)--(4,2)--(6,2)--(6,4)--(8,4)--(8,5)--(9,5)--(9,6)--(7,6)--(7,5)--(5,5)--(5,3)--(3,3)--(3,2)--(0,2)--(0,0)--cycle;
  \draw[color=gray]	(0,1)--(1,1)--(1,2)
		(2,1)--(2,2)
		(3,1)--(3,2)--(4,2)--(4,3)
		(5,2)--(5,3)--(6,3)
		(5,4)--(6,4)--(6,5)
		(7,4)--(7,5)--(8,5)--(8,6);
  \node[] at (2.5, 1.5) {\color{red}$a_2$};
   \node[] at (8.5, 5.5) {\color{blue}$b_3$};
      \node[] at (4.5, -.5) {(c)};
  \end{tikzpicture}
  \caption{The decomposition in Figure~\ref{fig:to_skew} is obtained recursively
    using these three border strips.}
  \label{fig:to_skew1}
\end{figure}

We now prove Theorem~\ref{thm:HG_conv}.

\begin{proof}[Proof of Theorem~\ref{thm:HG_conv}]
  We may assume $\det \left(s_{\alpha[a_j,b_i]}\right)_{i,j=1}^k \ne 0$ because
  otherwise there is nothing to prove. Since permuting the $a_i$'s or $b_i$'s
  only changes the sign of the determinant, we may assume $a_1<\dots<a_k$ and
  $b_1<\dots<b_k$. We first claim that the condition \eqref{eq:dyck} holds for
  all $a\in \ZZ$. For a contradiction, suppose that \eqref{eq:dyck} does not
  hold for some $a\in \ZZ$. Then we have $a_1<\dots<a_r\le a < a_{r+1}<\dots<
  a_k$ and $b_1<\dots<b_s\le a < b_{s+1}<\dots< b_k$ for some integers $0\le
  r<s\le k$. Then $s_{\alpha[a_j,b_i]}=0$ for all $1\le i\le s$ and $r+1\le j\le
  k$, which implies $\det \left(s_{\alpha[a_j,b_i]}\right)_{i,j=1}^k = 0$, which
  is a contradiction to the assumption given in the beginning of the proof.
  Therefore \eqref{eq:dyck} holds for all $a\in \ZZ$.

\begin{figure}
  \centering
   \begin{tikzpicture}[scale=.46]
   \draw[color=gray] (0,1)--(3,1)
   	    (3,2)--(4,2)
	    (5,3)--(8,3)
	    (5,4)--(9,4)
	    (7,5)--(9,5)
	    (1,0)--(1,2)
	    (2,0)--(2,2)
	    (3,1)--(3,2)
	    (4,2)--(4,3)
	    (5,2)--(5,3)
	    (6,2)--(6,5)
	    (7,2)--(7,5)
	    (8,3)--(8,6);
 \draw[line width=1.2pt] (0,0)--(3,0)--(3,1)--(4,1)--(4,2)--(8,2)--(8,3)--(9,3)--(9,6)--(7,6)--(7,5)--(5,5)--(5,3)--(3,3)--(3,2)--(0,2)--(0,0)--cycle; 
 \node[] at (.5, .5) {\color{red}$a_1$};
  \node[] at (2.5, 1.5) {\color{red}$a_2$};
   \node[] at (4.5, 2.5) {\color{red}$a_3$};
    \node[] at (3.5, 2.5) {\color{blue}$b_1$};
  \node[] at (5.5, 2.5) {\color{blue}$b_2$};
   \node[] at (8.5, 5.5) {\color{blue}$b_3$};
   \end{tikzpicture}\qquad\qquad
     \begin{tikzpicture}[scale=.46]
     \fill[fill=red!10!white] (1,0) rectangle (3,1)
     		(6,2) rectangle (8,4)
		(8,3) rectangle (9,5);
   \draw[thick, color=dredcolor] (1,0)--(3,0)--(3,1)
   		(2,0)--(2,1)
		(6,2)--(8,2)--(8,3)--(9,3)--(9,5)
		(7,2)--(7,4)
		(8,3)--(8,4)
		(6,3)--(8,3)
		(8,4)--(9,4);
     \draw[color=gray] (0,1)--(1,1)--(1,2)
     		(2,1)--(2,2)
		(3,1)--(3,2)--(4,2)--(4,3)
		(5,2)--(5,3)--(6,3)
		(5,4)--(6,4)--(6,5)
		(7,4)--(7,5)--(8,5)--(8,6);
     \draw[line width=1.3pt] (0,0)--(1,0)--(1,1)--(4,1)--(4,2)--(6,2)--(6,4)--(8,4)--(8,5)--(9,5)--(9,6)--(7,6)--(7,5)--(5,5)--(5,3)--(3,3)--(3,2)--(0,2)--(0,0)--cycle;
   \node[] at (.5, .5) {\color{red}$a_1$};
  \node[] at (2.5, 1.5) {\color{red}$a_2$};
   \node[] at (4.5, 2.5) {\color{red}$a_3$};
    \node[] at (3.5, 2.5) {\color{blue}$b_1$};
  \node[] at (5.5, 2.5) {\color{blue}$b_2$};
   \node[] at (8.5, 5.5) {\color{blue}$b_3$};
       \end{tikzpicture}
  \caption{A connected skew shape $\alpha$ and two sequences $a_1<a_2<a_3$ and
    $b_1<b_2<b_3$ on the left. The inner strip $\gamma$ of $\alpha$ is shown on
    the right with thick lines, where the cells in $\alpha\setminus\gamma$ are
    colored red. }
  \label{fig:inner_gamma}
\end{figure}
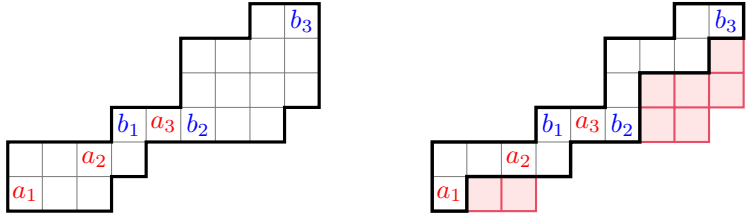

\begin{figure}
  \centering
     \begin{tikzpicture}[scale=.46]
      \fill[fill=pink, rounded corners] (.2, .2) rectangle (.8, 1.8)
  			(.2, 1.2) rectangle (3.8, 1.8)
			(3.2, 1.2) rectangle (3.8, 2.8);
  \fill[fill=orange!30!white, rounded corners] (1.2, 2.2) rectangle (2.8, 2.8)
  			(2.2, 2.2) rectangle (2.8, 3.8)
			(2.2, 3.2) rectangle (4.8, 3.8)
			(4.2, 3.2) rectangle (4.8, 5.8)
			(4.2, 5.2) rectangle (6.8, 5.8)
			(6.2, 5.2) rectangle (6.8, 6.8)
			(6.2, 6.2) rectangle (7.8, 6.8);
  \fill[fill=cyan!20!white, rounded corners] (2.1, 4.2) rectangle (3.9, 4.8);
        \fill[fill=red!10!white] (1,0) rectangle (3,1)
     		(5,3) rectangle (7,5)
		(7,4) rectangle (8,6);
   \draw[thick, color=dredcolor] (1,0)--(3,0)--(3,1)
   		(2,0)--(2,1)
		(5,3)--(7,3)--(7,4)--(8,4)--(8,6)
		(6,3)--(6,5)
		(7,4)--(7,5)
		(5,4)--(7,4)
		(7,5)--(8,5);
   \draw[color=gray] (0,1)--(1,1)
   			      (1,2)--(4,2)
			      (2,3)--(4,3)
			      (2,4)--(5,4)
			      (4,5)--(5,5)
			      (6,6)--(7,6)
			      (1,1)--(1,2)
			      (2,1)--(2,3)
			      (3,1)--(3,5)
			      (4,3)--(4,5)
			      (5,5)--(5,6)
			      (6,5)--(6,6)
			      (7,6)--(7,7);			   
  \draw[thick] (1,2)--(3,2)--(3,3)--(4,3)
  			(2,4)--(4,4)--(4,5);
    \draw[line width=1.6pt, color=frenchblue!90!cyan] (0,0)--(1,0)--(1,1)--(4,1)--(4,3)--(5,3)--(5,5)--(7,5)--(7,6)--(8,6)--(8,7)--(6,7)--(6,6)--(4,6)--(4,5)--(2,5)--(2,3)--(1,3)--(1,2)--(0,2)--(0,0)--cycle;
     \node[] at (.5, .5) {\color{red}$p_1$};
  \node[] at (1.5, 2.5) {\color{red}$p_2$};
   \node[] at (2.5, 4.5) {\color{red}$p_3$};
    \node[] at (3.5, 2.5) {\color{blue}$q_1$};
  \node[] at (3.5, 4.5) {\color{blue}$q_2$};
   \node[] at (7.5, 6.5) {\color{blue}$q_3$};
      \end{tikzpicture}
  \caption{The boundary of the skew shape $\lm$ is colored blue and the cells in
    $\nl$ are colored red. Here, using the $\alpha$, $\gamma$, $a_i$'s and
    $b_i$'s in Figure~\ref{fig:inner_gamma}, $\lm$ is the skew shape whose
    decomposition with respect to the cutting strip $\gamma$ is
    $\theta=(\theta_1,\theta_2,\theta_3)$ such that
    $\{p(\theta_1),p(\theta_2),p(\theta_3)\}=\{a_1,a_2,a_3\}$ and
    $\{q(\theta_1),q(\theta_2),q(\theta_3)\}=\{b_1,b_2,b_3\}$, and the connected
    components of $\nl$ are obtained from those of $\alpha/\gamma$ by
    appropriate diagonal shifts.}
  \label{fig:nl}
\end{figure}
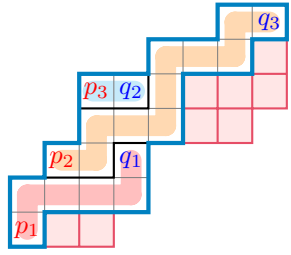

Now let $\gamma$ be the inner strip of $\alpha$ as shown in
Figure~\ref{fig:inner_gamma}. Then by Lemma~\ref{lem:gamma_to_la} there is a
skew shape $\lm$ whose decomposition with respect to the cutting strip $\gamma$
is $\theta=(\theta_1,\dots,\theta_k)$ such that
$\{p(\theta_1),\dots,p(\theta_k)\}=\{a_1,\dots,a_k\}$ and
$\{q(\theta_1),\dots,q(\theta_k)\}=\{b_1,\dots,b_k\}$. Let $\nu$ be the
partition obtained from $\lambda$ by attaching each connected component of
$\alpha\setminus\gamma$ below $\lambda$ after an appropriate diagonal shift, see
Figure~\ref{fig:nl}. The assumption that $\alpha[a_j,b_i]$ is a skew shape, for
all $i,j$, implies that $\mu$ and $\gamma$ are $\nl$-compatible. Therefore, by
Theorem~\ref{thm:main_HG} with $\ts_{\tau}$ specialized to $s_{\tau}$ for all
skew shapes $\tau$, we have
\begin{equation}\label{eq:det = s}
  \det \left( s_{(\gamma\oplus\nl)[a_j,b_i]}\right)_{i,j=1}^k = \pm s_{\nu/\mu}
  \prod_{s=1}^\ell s_{\alpha_s}^{r_s-1},
\end{equation}
where $\alpha_1,\dots,\alpha_\ell$ are the connected components of $\nl$, and
$r_s$ is the number of strips $\theta_i$ such that
$\Cont(\alpha_s)\subseteq\Cont(\theta_i)$.

By the construction, we have $(\gamma\oplus\nl)[a_j,b_i]=\alpha[a_j,b_i]$. Since
a product of skew Schur functions can be expressed as a single skew Schur
function, the theorem follows from \eqref{eq:det = s}.
\end{proof}

\section*{Acknowledgments}

The authors would like to thank Byung-Hak Hwang for helpful discussions
and Alejandro Morales for useful comments.



\begin{thebibliography}{10}

\bibitem{Bachmann_2020}
H.~Bachmann and S.~Charlton.
\newblock {Generalized Jacobi–Trudi determinants and evaluations of Schur
  multiple zeta values}.
\newblock {\em European Journal of Combinatorics}, 87:103133, 2020.

\bibitem{Bazin}
M.~Bazin.
\newblock Sur une question relative aux d\'eterminants.
\newblock {\em J. Math. Pures Appl.}, 16:145--160, 1851.

\bibitem{Chen_2005}
W.~Y.~C. Chen, G.-G. Yan, and A.~L.~B. Yang.
\newblock {Transformations of Border Strips and Schur Function Determinants}.
\newblock {\em Journal of Algebraic Combinatorics}, 21(4):379–394, 2005.

\bibitem{Doubilet_1974}
P.~Doubilet, G.-C. Rota, and J.~Stein.
\newblock {On the foundations of combinatorial theory: IX combinatorial methods
  in invariant theory}.
\newblock {\em Studies in Applied Mathematics}, 53(3):185–216, 1974.

\bibitem{Foley20}
A.~M. Foley and R.~C. King.
\newblock {Determinantal and Pfaffian identities for ninth variation skew Schur
  functions and $Q$-functions}.
\newblock {\it Preprint},
  \href{https://arxiv.org/abs/2002.11796v1}{arXiv:2002.11796v1}.

\bibitem{Giambelli}
G.~Giambelli.
\newblock Alcune propriet\`a delle funzioni simmetriche caratteristiche.
\newblock {\em Atti della R. Acc. delle Scienze di Torino}, 38:823--844, 1903.

\bibitem{Hamel_1995}
A.~Hamel and I.~Goulden.
\newblock {Planar decompositions of tableaux and Schur function determinants}.
\newblock {\em European Journal of Combinatorics}, 16(5):461–477, 1995.

\bibitem{Jin_2018}
E.~Y. Jin.
\newblock {Outside nested decompositions of skew diagrams and Schur function
  determinants}.
\newblock {\em European Journal of Combinatorics}, 67:239–267, 2018.

\bibitem{Lascoux1988}
A.~Lascoux and P.~Pragacz.
\newblock Ribbon {S}chur functions.
\newblock {\em European J. Combin.}, 9(6):561--574, 1988.

\bibitem{Macdonald_Schur}
I.~G. Macdonald.
\newblock Schur functions: theme and variations.
\newblock In {\em S{\'e}minaire {L}otharingien de {C}ombinatoire
  ({S}aint-{N}abor, 1992)}, volume 498 of {\em Publ. Inst. Rech. Math. Av.},
  pages 5--39. Univ. Louis Pasteur, Strasbourg, 1992.

\bibitem{Macdonald}
I.~G. Macdonald.
\newblock {\em Symmetric functions and {H}all polynomials}.
\newblock Oxford Mathematical Monographs. The Clarendon Press Oxford University
  Press, New York, second edition, 1995.
\newblock With contributions by A. Zelevinsky, Oxford Science Publications.

\bibitem{MPP2}
A.~H. Morales, I.~Pak, and G.~Panova.
\newblock Hook formulas for skew shapes {II}. {C}ombinatorial proofs and
  enumerative applications.
\newblock {\em SIAM J. Discrete Math.}, 31(3):1953--1989, 2017.

\bibitem{okada17:sylvester}
S.~Okada.
\newblock {Generalized Sylvester formulas and skew Giambelli identities}.
\newblock {\it Preprint},
  \href{https://arxiv.org/abs/1704.02585v1}{arXiv:1704.02585v1}.

\end{thebibliography}
\end{document}